\documentclass[a4paper]{amsart}
\usepackage[utf8]{inputenc}
\usepackage[margin = 1in]{geometry}
\usepackage{amsmath, amssymb, amsfonts, amsthm, tensor, dsfont}
\usepackage[foot]{amsaddr}
\usepackage[svgnames]{xcolor}
\usepackage[style = alphabetic, sorting = nyt]{biblatex}
\usepackage{tikz-cd}
\usepackage[all]{xy}
\usepackage{hyperref}
\usepackage[hyphenbreaks]{breakurl}
\usepackage{subfiles}
\usepackage{verbatim}

\hypersetup{colorlinks = true, citecolor = MediumBlue, linkcolor = Crimson}

\newtheorem{thm}{Theorem}[section]
\newtheorem{lem}[thm]{Lemma}

\newtheorem{cor}[thm]{Corollary}

\theoremstyle{definition}

\newtheorem{mydef}[thm]{Definition}

\newcommand{\C}{\mathbb{C}}
\newcommand{\R}{\mathbb{R}}

\newcommand{\E}{\mathcal{E}}

\newcommand{\cH}{\mathcal{H}}

\newcommand{\al}{\alpha}
\newcommand{\bb}{\beta}

\newcommand{\om}{\omega}
\newcommand{\ee}{\varepsilon}
\newcommand{\te}{\theta}
\newcommand{\tte}{\tilde{\theta}}
\newcommand{\la}{\lambda}
\newcommand{\Poe}{P_{\om_{\ee}}}

\newcommand{\tX}{\tilde{X}}
\newcommand{\tu}{\tilde{u}}
\newcommand{\PSH}{\text{PSH}}

\begin{document}

\title{Complete Geodesic Metrics in big classes}
\author{Prakhar Gupta }
\address{Department of Mathematics, University of Maryland, College Park, Maryland, USA}
\email{pgupta8@umd.edu}
\begin{abstract}
    Let $(X,\om)$ be a compact K\"ahler manifold and $\theta$ be a smooth closed real $(1,1)$-form that represents a big cohomology class. In this paper, we show that for $p\geq 1$, the high energy space $\E^{p}(X,\theta)$ can be endowed with a metric $d_{p}$ that makes $(\E^{p}(X,\theta),d_{p})$ a complete geodesic metric space. The weak geodesics in $\E^{p}(X,\theta)$ are the metric geodesic for $(\E^{p}(X,\theta), d_{p})$. Moreover, for $p > 1$, the geodesic metric space $(\E^{p}(X,\te), d_{p})$ is uniformly convex. 
\end{abstract}



\maketitle

\tableofcontents

\section{Introduction} \label{sec: introduction}
On a compact K\"ahler manifold $(X,\om)$, the problem of finding the canonical metric in the same cohomology class as $\om$ has a long history. Calabi defined the space 
\[
\mathcal{H}_{\om} = \{ u \in C^{\infty}(X)  : \om + dd^{c}u  > 0\}
\]
of functions that, up to normalization, is equivalent to the space of all K\"ahler metrics cohomologous to $\om$. In \cite{mabuchiriemannianstructure},\cite{semmesriemannianmetric}, and \cite{donaldsonriemannianmetric}, the authors discovered a Riemannian structure on $\mathcal{H}_{\om}$ whose geodesic equation is a homogeneous complex Monge-Amp\`ere equation in one higher dimension. In \cite{chenspaceofkahlermetrics}, Chen proved that this Riemannian structure gives rise to an honest metric $d_{2}$ on $\mathcal{H}_{\om}$ by showing that the $C^{1,1}$ geodesics joining endpoints are length minimizing. 

In \cite{darvascompletionofspaceofkahlerpotentials}, Darvas showed that the completion of $(\mathcal{H}_{\om}, d_{2})$ is given by $(\E^{2}(X,\om), d_{2})$ where $\E^{2}(X,\om)$ is the space of potentials with finite $L^{2}$-energy, confirming a conjecture of Guedj, \cite{guedj2014metriccompletion}. See Section~\ref{sec: preliminaries} to see the definition of finite energy spaces. He further showed that the potentials $u_{0}, u_{1} \in \E^{2}(X,\om)$ can be joined by a weak geodesic that lies in $\E^{2}(X,\om)$ and the path is a metric geodesic for $(\E^{2}(X,\om), d_{2})$. By a \emph{metric geodesic} on a metric space $(M,d)$, we mean a path $[0,1] \ni t \mapsto u_{t} \in M$ such that for any $t_{0}, t_{1} \in [0,1]$, $d(u_{t_{0}}, u_{t_{1}}) = |t_{0} - t_{1}|d(u_{0}, u_{1})$.  In \cite{darvasgeoemtryoffiniteenergy}, Darvas extended the result to Finsler metric structures on $\mathcal{H}_{\om}$. In particular, for the $L^{p}$-Finsler structure on $\mathcal{H}$, for $p\geq 1$, he obtained a metric $d_{p}$ on $\mathcal{H}_{\om}$ whose completion is $(\E^{p}(X,\om),d_{p})$. The case of $p = 1$ has found several applications in finding the canonical metrics (see \cite{darvasrubinsteinpropernessconjecture}, \cite{chenchengcsckI}, \cite{chenchengcscKII}).  In \cite{DarvasLuuniformconvexity}, the authors proved geodesic stability of the $K$-energy with respect to $d_{1}$ metric by approximation from $(\E^{p}(X,\om), d_{p})$ for $p> 1$ which they showed are \textit{uniformly convex}.

By finding a formula for the distance in terms of pluripotential theoretic functions, in \cite{darvas2021l1}, Darvas-Di Nezza-Lu showed that the space $\E^{1}(X,\theta)$, for $\theta$ representing a big cohomology class, has a complete geodesic metric $d_{1}$. In \cite{dinezza2018lp}, by approximating from the K\"ahler case, Di Nezza-Lu found a complete geodesic metric $d_{p}$ on $\E^{p}(X,\beta)$, where $\beta$ represents a big and nef cohomology class. In both cases, the weak geodesics are metric geodesics as well. See Section~\ref{subsec: recap of pluripotential theory} to see the definition of big and nef cohomology classes, and see Section~\ref{subsec: weak geodesics} to see the definition of weak geodesics.

In this paper, using a different approximation scheme, we are able to extend the result of \cite{dinezza2018lp} to the big setting:

\begin{thm}\label{thm: Main theorem} Given a smooth closed real $(1,1)$-from $\theta$ that represents a big cohomology class, the space $\E^{p}(X,\theta)$ admits a complete geodesic metric $d_{p}$. Moreover, the weak geodesics of $\E^{p}(X,\theta)$ are metric geodesics in $(\E^{p}(X,\te), d_{p})$. 
\end{thm}

When $\te$ is big and nef, then the metric $d_{p}$ constructed in Theorem~\ref{thm: Main theorem} agrees with the one constructed in \cite{dinezza2018lp}, which in turn agrees with the one constructed in \cite{darvasgeoemtryoffiniteenergy} when $\te$ is K\"ahler. In case $p = 1$, the metric $d_{1}$ in Theorem~\ref{thm: Main theorem} agrees with the metric constructed in \cite{darvas2021l1}. 

Several other works have explored the metric structure of finite energy classes in varying generality. In \cite{TrusianiL1metric}, Trusiani shows that the space $\E^{1}(X,\theta,\phi)$ has a complete metric $d_{1}$ where $\phi$ is a model singularity. See Section~\ref{subsec: prescribed singularity setting} to see the definitions in the prescribed singularity setting. In \cite{Xia2019MabuchiGO}, Xia showed that the space $\E^{p}(X,\theta,\phi)$ has a \emph{locally complete metric} $d_{p}$, moreover he asked if the space $(\E^{p}(X,\te,\phi),d_{p})$ is a geodesic metric space. Theorem~\ref{thm: Main theorem}, answers this question in the minimal singularity setting. Also, Theorem~\ref{thm: Complete metric space structure on analytic singularity space} answers this question when $\phi$ has analytic singularity type.  In \cite{darvas2021mabuchi}, Darvas showed that $\E_{\chi}(X,\om)$ has a complete metric $d_{\chi}$ where $\E_{\chi}(X,\om)$ is the \emph{low energy space}. In \cite{GuptaCompletemetrictopology}, the author showed that $\E_{\chi}(X,\theta,\phi)$ has a complete metric $d_{\chi}$ in the prescribed singularity setting. In all these works, the metric space was not shown to admit geodesics.

\subsection{Uniform Convexity}
In \cite{mabuchiriemannianstructure}, Mabuchi found that $\mathcal{H}_{\om}$ with the Riemannian structured obtained from 
\[
\langle \phi, \psi \rangle_{u} = \frac{1}{\text{Vol}(\om)} \int_{X} \phi \psi \om^{n}_{u}
\]
gives $\mathcal{H}$ a non-positively curved Riemannian structure. As the metric space structure of $\E^{p}(X,\om)$ was better understood, so was the nature of their non-positive curvature. 

In \cite{darvas2021mabuchi}, building on the work of Calabi-Chen \cite{CalabiChenSpaceofKahlermetricsII}, Darvas showed that $\E^{2}(X,\om)$ is non-positively curved in the sense of Alexandrov. In \cite{DarvasLuuniformconvexity}, Darvas-Lu proved uniform convexity of metric spaces $(\E^{p}(X,\om), d_{p})$ for $p > 1$. We prove that the approximation scheme used to construct the metric space $(\E^{p}(X,\te), d_{p})$ in the big case, preserves the uniform convexity. 

\begin{thm}\label{mainthm: uniform convexity in the big case}
If $\te$ represents a big cohomology class then the metric space $(\E^{p}(X,\te), d_{p})$ as defined in Theorem~\ref{thm: Main theorem} is uniformly convex. This means for $u, v_{0}, v_{1} \in \E^{p}(X,\te)$, if $v_{\la}$ is the weak geodesic joining $v_{0}$ and $v_{1}$, then
\begin{align*}
        d_{p}(u,v_{\lambda})^{2} &\leq (1-\la)d_{p}(u,v_{0})^{2} + \la d_{p}(u,v_{1})^{2} - (p-1)\la(1-\la)d_{p}(v_{0},v_{1})^{2}, \text{ if } 1< p \leq 2 \text{ and }\\
d_{p}(u,v_{\la})^{p} &\leq (1-\la)d_{p}(u,v_{0})^{p} + \la d_{p}(u,v_{1})^{p} - \la^{p/2}(1-\la)^{p/2}d_{p}(v_{0},v_{1})^{p}, \text{ if } p \leq 2.
    \end{align*}     
\end{thm}
This proves in particular that $(\E^{2}(X,\te), d_{2})$ is a CAT(0) space. This also shows that the weak geodesics are unique geodesics in $(\E^{p}(X,\te),d_{p})$ for $p > 1$. When $p = 1$, $(\E^{1}(X,\te), d_{1})$ does not have unique geodesics as follows from the comments following \cite[Theroem 4.17]{darvasgeoemtryoffiniteenergy}.

The fact that $(\E^{2}(X,\te), d_{2})$ is a CAT(0) space opens the avenue for studying gradient flows in this space. From the work of \cite{Mayer1998GradientFO}, if $G: \E^{2}(X,\te) \to (-\infty, \infty]$ is a convex $d_{2}$-lower semicontinuous functional, then we can run a weak gradient flow. From \cite{Back2012ThePP}, the gradient flow will converge \emph{$d_{2}$-weakly} to a minimizer of $G$ if the minimizer exists. If we can prove the expected convexity of the Mabuchi $K$-energy in the big case (see \cite{DiNezzaLugeodesicdistanceandmameasure} for the big and nef case), then we can run the weak Calabi flow and prove that the flow converges to a minimizer if it exists, as was done in the K\"ahler case in \cite{StreetsLongtimeexistenceofcalabiflow}, \cite{StreetsConsistencyofKenergyminimizingmovements}, and \cite{BDL17}.

Using similar methods as in the proof of Theorem~\ref{mainthm: uniform convexity in the big case}, in a forthcoming work \cite{GuptaGeodesicsegmentsofgeodesicrays} we prove the Buseman convexity of the metric spaces $(\E^{p}(X,\te), d_{p})$ for $p \geq 1$, opening the door to study the space of geodesic rays in $\E^{p}(X,\te)$,  as done in the K\"ahler setting in \cite{DarvasLuuniformconvexity}.

\subsection{Strategy of the proof}
We will give a brief overview of the proof whose details are in the rest of the paper. The crucial idea is that we can approximate the geometry of $\E^{p}(X,\te)$ from the geometry of $\E^{p}(X,\te, \psi_{k})$ where $\psi_{k}$ has analytic singularities and $\psi_{k} \nearrow V_{\te}$. Moreover, the geometry on $\E^{p}(X,\te,\psi_{k})$ can be imported from the geometry of $\E^{p}(\tilde{X},\bb)$ where $\bb$ represents a big and nef class. 

More precisely, we show that if $\psi \in \PSH(X,\theta)$ has analytic singularities, then the space $\E^{p}(X,\theta, \psi)$ has a complete geodesic metric $d_{p}$. We show this using a modification $\mu: \tilde{X} \to X$ that principalizes the singularities of $\psi$, that can be subtracted, giving us a bijection between $\E^{p}(\tilde{X}, \beta)$ and $\E^{p}(X,\te,\psi)$, where $\beta$ is a smooth closed real $(1,1)$-form on $\tilde{X}$ representing a big and nef cohomology class. Then we import the complete geodesic metric $d_{p}$ on $\E^{p}(\tilde{X},\beta)$, as found by Di Nezza-Lu in \cite{dinezza2018lp}, to $\E^{p}(X,\theta,\psi)$ using the bijection. 

Using Demailly's regularization theorem, we find a sequence of $\te$-psh functions $\psi_{k} \nearrow V_{\theta}$, where each $\psi_{k}$ has analytic singularities. Then we approximate the metric $d_{p}$ on $\E^{p}(X,\theta)$ from the metric $d_{p}$ on $\E^{p}(X,\theta, \psi_{k})$ and show that it is a complete geodesic metric. 

We prove the uniform convexity of the $d_{p}$ metric by approximation as well. The metric $d_{p}$ on $(\E^{p}(X,\bb)$ where $\bb$ represents a big and nef cohomology class, was constructed by approximation from the K\"ahler case. We show that the same approximation method carries over to show that the metric space $(\E^{p}(X,\bb), d_{p})$ is uniformly convex for $p > 1$. 

In the big case, we first prove a contraction property for the metrics $d_{p}$. In particular, we show in Theorem~\ref{thm: contraction property} that if $\psi \in\PSH(X,\te)$ has anlytic singularities, then the map $P_{\te}[\psi](\cdot) : \E^{p}(X,\te) \to \E^{p}(X,\te, \psi)$ is a contraction, i.e., 
$$d_{p}(P_{\te}[\psi](u_{0}), P_{\te}[\psi](u_{1})) \leq d_{p}(u_{0},u_{1}),$$ 
for any $u_{0}, u_{1} \in \E^{p}(X,\te)$. Using this contraction, and approximation from the analytic singularity setting, we show that the metric space $(\E^{p}(X,\te), d_{p})$ is uniformly convex for $p > 1$.

\subsection{Organization}\label{subsec: organization} In Section~\ref{sec: preliminaries}, we will recall key concepts from pluripotential theory and several results from the literature that we will use in our results. In Section~\ref{sec: analytic singularities to big and nef}, we will describe how to import metric geometry from big and nef classes to the potentials with prescribed analytic singularity through desingularization and subtracting the divisorial singularity. In Section~\ref{sec: metric in the big case} we will define the metric on $\E^{p}(X,\theta)$ by approximating it as described above. In Section~\ref{sec: properties}, we show that the metric obtained is geodesic and complete and we prove other relevant properties of the metric. In Section~\ref{sec: uniform convexity for big and nef} we prove uniform convexity of $(\E^{p}(X,\bb), d_{p})$ for $p > 1$ where $\bb$ represents a big and nef cohomology class. In Section~\ref{sec: contraction property and a consequence} we prove the contraction property that we use in Section~\ref{sec: uniform convexity in the big case} to prove uniform convexity in the big case. 

\subsection{Acknowledgements} I would like to thank my advisor Tam\'as Darvas for proposing this problem and for his continued guidance. I also want to thank Antonio Trusiani for useful conversations during the summer school in Le Croisic. I want to thank Antonio Trusiani and Mingchen Xia for asking interesting questions on the first draft of this paper some of which have improved this paper. While preparing this work, the author became aware of \cite{dinezza2023entropy}, and some results in the preliminaries section overlap with their work. This research was partially supported by NSF CAREER grant DMS-1846942.

\section{Preliminaries} \label{sec: preliminaries}
In this paper, $(X,\om)$ is a compact K\"ahler manifold of complex dimension $n$ and $\int_{X}\om^{n} = 1$. 

\subsection{Quick recap of pluripotential theory}\label{subsec: recap of pluripotential theory}
Given a smooth closed real $(1,1)$-form $\theta$, we say that an upper semicontinuous function $u : X \to \R\cup \{-\infty\}$ is a $\theta$-psh function if locally on $U \subset X$ where $dd^{c}g = \theta$, $u+g$ is plurisubharmonic. This implies that $\theta + dd^{c}u \geq 0$ as $(1,1)$-currents. We denote by $\PSH(X,\theta)$ the set of all $\theta$-psh functions that are not identically $-\infty$. 

We denote by $\{\theta\}$ the $H^{1,1}(X,\R)$ cohomology class of $\theta$. We say that $\theta$ represents a K\"ahler class, if there exists a smooth $\theta$-psh function $u$ such that $\theta  + dd^{c} u > 0$. $\theta$ represents a nef class if $\{\theta + \ee \om\}$ is a K\"ahler class for all $\ee > 0$. We say $\theta$ represents a big class if there exists a potential $u \in \PSH(X,\theta)$ such that $\theta+dd^{c}u \geq \ee \om$ for some small enough $\ee > 0$. If $u,v \in \PSH(X,\theta)$ satisfy $u \leq v + C$ for some constant $C$, then we say $u$ is more singular than $v$ and denote it by $u \preceq v$. If $\theta$ represents a big cohomology class, then 
\[
V_{\theta} = \sup\{ u \in \PSH(X,\theta) : u \leq 0\}
\]
is a $\theta$-psh function that has minimal singularities. From now on we fix a smooth closed real $(1,1)$-form $\theta$ that represents a big cohomology class. 

We say that $\psi \in \PSH(X,\theta)$ has analytic singularities of type $(\mathcal{I},c)$ if there exists a rational number $c > 0$ and a coherent ideal sheaf $\mathcal{I}$ such that for all $x \in X$, there exists a neighborhood $U \subset X$ of $x$ such that $\mathcal{I}$ is generated on $U$ by holomorphic functions $(f_{1},\dots, f_{N})$ and 
\[
\psi_{|U} = c \log \left( \sum_{j=1}^{N} |f_{j}|^{2} \right) + h
\]
where $h$ is a bounded function defined on $U$. From \cite[Lemma 2.4]{darvas2023transcendental}, we notice that analytic singularity is stable under max. This means that if $u,v \in \PSH(X,\theta)$ have analytic singularities, then $\max(u,v) \in \PSH(X,\theta)$ has analytic singularities as well. From Demailly's regularization result, there exists $\psi \in \PSH(X,\theta)$ such that $\theta + dd^{c} \psi > \ee \om$ and $\psi$ has analytic singularities. 

In \cite{Boucksom2008MongeAmpreEI}, the authors defined a non-pluripolar product of $\theta$-psh functions. If $u_{1},\dots,u_{n} \in \PSH(X,\theta)$, they defined their non-pluripolar product $\langle \theta_{u_{1}} \wedge \dots \wedge \theta_{u_{n}}\rangle$ as a non-pluripolar measure. For simplicity, we write $\theta^{n}_{u} := \langle \theta_{u} \wedge \dots \wedge \theta_{u}\rangle$. If $\theta$ is big, then we say $\int_{X}\theta^{n}_{V_{\theta}} = \text{Vol}(\theta)$. From \cite{Wittnystrommonotonicity}, and \cite{Darvasmonotonicity}, we notice that for any $u_{1} , \dots, u_{n} \in \PSH(X,\theta)$, $\int_{X}\langle \theta_{u_{1}} \wedge \dots \wedge \theta_{u_{n}} \rangle \leq \text{Vol}(\theta)$.

\subsection{Finite energy classes}\label{subsec: finite energy classes}

Finite energy classes in the K\"ahler setting were introduced by Guedj-Zeriahi \cite{GZweightedmongeampereenergy} to solve the Complex Monge-Amp\`ere equation on a compact K\"ahler manifold for a very general right-hand side. In this paper, we deal with $\E^{p}$ energy classes that we now describe. We define the space of potentials of full mass as
\[
\E(X,\theta) = \{ u \in \PSH(X,\theta) :  \int_{X}\theta^{n}_{u} = \int_{X}\theta^{n}_{V_{\theta}} \}.
\]
For $p \geq 1$, the potentials with finite $p$-energy are defined as 
\[
\E^{p}(X,\theta) = \{ u \in \E(X,\theta) : \int_{X}|u-V_{\theta}|^{p} \theta^{n}_{u}< \infty\}.
\]

\subsection{Prescribed singularity setting} \label{subsec: prescribed singularity setting}
Darvas-Di Nezza-Lu developed pluripotential theory in the prescribed singularity setting in several papers including \cite{Darvasmonotonicity}, \cite{darvaslogconcavity}, and \cite{darvas2019metric}. See \cite{darvas2023relative} to see the survey on this. Here we briefly recall the definitions of finite energy spaces in the prescribed singularity setting. We say that $\phi \in \PSH(X,\theta)$ with $\int_{X}\theta^{n}_{\phi} > 0$ is a model singularity type if 
\[
\phi = P_{\theta}[\phi] := \sup \{ u \in \PSH(X,\theta) : u \preceq \phi, u \leq 0\}. 
\]
Model singularities were introduced by Darvas-Di Nezza-Lu in \cite{Darvasmonotonicity} to solve the complex Monge-Amp\`ere equation with prescribed singularities. We can also define the finite energy classes relative to $\phi$. We denote by 
\[
\PSH(X,\theta, \phi) = \{ u \in \PSH(X,\theta)  : u \preceq \phi\}.
\]
The space of potentials of full mass relative to $\phi$ is 
\[
\E(X,\theta) = \{ u \in \PSH(X,\theta,\phi) : \int_{X} \theta^{n}_{u} = \int_{X}\theta^{n}_{\phi}\}.
\]
We define the space of $\phi$-relative finite $p$-energy potentials by
\[
\E^{p}(X,\theta, \phi) = \{ u \in \E(X,\theta, \phi) : \int_{X}|u-\phi|^{p}\theta^{n}_{u} < \infty\}.
\]

We would need the following result about the $\E^{p}(X,\te,\phi)$ spaces. 
\begin{thm}\label{thm: Ip converges along decreasing sequences}
    For $u,v \in \E^{p}(X,\te,\phi)$, we define 
    \[
    I_{p}(u,v) = \int_{X} |u-v|^{p}(\te^{n}_{u} + \te^{n}_{v}).
    \]
    If $u_{0}^{j}, u_{1}^{j}, u_{0}, u_{1} \in \E^{p}(X,\te,\phi)$ such that $u_{0}^{j} \searrow u_{0}$ and $u_{1}^{j} \searrow u_{1}$, then $I_{p}(u_{0}^{j}, u_{1}^{j}) \to I_{p}(u_{0}, u_{1})$ as $j \to \infty$. 
\end{thm}

\begin{proof}
    The fact that $I_{p}(u,v) < \infty$ follows from the arguments in \cite[Section 2]{GuptaCompletemetrictopology} by modifying the proof for the weight $\chi(t) = |t|^{p}$. 

    The same proof as in \cite[Theorem 4.1]{GuptaCompletemetrictopology}, shows that 
    \begin{equation}\label{eq: max Pythagorean theorem for Ip}
    I_{p}(u,v) = I_{p}(u,\max(u,v)) + I_{p}(v,\max(u,v)).        
    \end{equation}

    First, assume that $u_{0}^{j} \leq u_{1}^{j}$, so consequently $u_{0} \leq u_{1}$. Now we observe that the proof in \cite[Proposition 2.20]{Darvas2019GeometricPT}, works in the generality of prescribed singularity setting with big classes as well. Thus we obtain that in the case $u_{0}^{j} \leq u_{1}^{j}$ and $u_{0}\leq u_{1}$, 
    \[
    \int_{X}|u_{0}^{j} - u_{1}^{j}|^{p} \te^{n}_{u_{0}^{j}} \to \int_{X}|u_{0}-u_{1}|^{p}\te^{n}_{u_{0}} 
    \]
    and 
    \[
    \int_{X}|u_{0}^{j} - u_{1}^{j}|^{p}\te^{n}_{u_{1}^{j}} \to \int_{X} |u_{0}^{j} - u_{1}^{j}|^{p}\te^{n}_{u_{1}}
    \]
    as $j \to \infty$. Adding the two we get 
    \[
    I_{p}(u_{0}^{j}, u_{1}^{j}) \to I_{p}(u_{0},u_{1})
    \]
    as $j \to \infty$. More generally, if $u_{0}^{j} \searrow u_{0}$ and $u_{1}^{j}\searrow u_{1}$, then $\max(u_{0}^{j}, u_{1}^{j}) \searrow \max(u_{0}, u_{1})$. Now the potentials $u_{0}^{j} \leq \max(u_{0}^{j}, u_{1}^{j})$ and $u_{1}^{j} \leq \max(u_{0}^{j}, u_{1}^{j})$. Thus 
    \[
    I_{p}(u_{0}^{j}, \max(u_{0}^{j}, u_{1}^{j})) \to I_{p}(u_{0}, \max(u_{0},u_{1}))
    \]
    and 
    \[
    I_{p}(u_{1}^{j}, \max( u_{0}^{j}, u_{1}^{j})) \to I_{p}(u_{1}, \max(u_{0}, u_{1}))
    \]
    as $j \to \infty$. Adding the two, and using Equation~\eqref{eq: max Pythagorean theorem for Ip}, we get 
    \[
    I_{p}(u_{0}^{j},u_{1}^{j}) \to I_{p}(u_{0}, u_{1})
    \]
    as $j \to \infty$. 
\end{proof}
\subsection{PSH Envelopes}\label{subsec: envelopes}
Given a measurable function $f: X \to \R\cup \{\pm \infty\}$, and $\te$ smooth real closed $(1,1)$-from representing a big cohomology class, we define 
\[
P_{\te}(f) = (\sup\{ u \in \PSH(X,\te) : u \leq f\})^{*},
\]
where $u^{*}(x) = \limsup_{y\to x}u(y)$ denotes the upper semicontinuous regularization. We say $P_{\te}(f) = -\infty$, if the candidate set is empty, otherwise $P_{\te}(f) \in \PSH(X,\te)$. In general, $P_{\te}(f) \leq f$ away from a pluripolar set as the upper semicontinuous regularization only changes the function away from a pluripolar set. If $f$ is upper semicontinuous, then $P_{\te}(f) \leq f$ everywhere. If $f,g : X \to \R\cup\{\pm\infty\}$ are two measurable functions, we define the rooftop envelope $P_{\te}(f,g):= P_{\te}(\min\{f,g\})$. 

Given $f$ and $\te$ as above, and $\phi \in \PSH(X,\te)$, we  define the envelope with respect to the singularity type of $\phi$ by
\[
P_{\te}[\phi](f) := \left(\lim_{C\to\infty}P_{\te}(\phi+  C, f)\right)^{*}.
\]
If $f$ is bounded, then $P_{\te}(\phi+  C,f)$ is an increasing sequence of $\te$-psh functions that are bounded from above by $f^{*}$, thus the limit in the above equation exits. Moreover, $P_{\te}[\phi](\cdot)$ depends only on the singularity type of $\phi$. The function $P_{\te}[\phi](\cdot)$ also satisfies the following concavity property. We recall

\begin{lem}[{\cite[Lemma 2.12]{darvas2023relative}}]\label{lem: concavity of envelope operator}
    Given a continuous function $f : X \to \R$, the operator $\PSH(X,\te) \ni u \mapsto P_{\te}[u](f) \in \PSH(X,\te)$ is concave. This means for $t \in (0,1)$, 
    \[
    tP_{\te}[u](f) + (1-t)P_{\te}[v](f) \leq P_{\te}[tu + (1-t)v](f).
    \]
\end{lem}

If $f \in C^{1,\bar{1}}(X)$, which means that $f$ has bounded Laplacian, then we have good control on the Monge-Amp\`ere measures of the envelopes $P_{\te}[\phi](f)$. For that we recall, 
\begin{thm}[\cite{EleonoraTrapaniMAmeasureoncontactsets} ]\label{thm: measure on contact sets}
    If $\theta$ represents a big cohomology class, $\phi \in \PSH(X,\theta)$, and $f \in C^{1,\bar{1}}(X)$ , then 
    \[
    \theta^{n}_{P_{\theta}[\phi](f)} = \mathds{1}_{\{ P_{\theta}[\phi](f) = f\}} \theta^{n}_{f}.
    \]
\end{thm}

In the same paper, the authors also prove 

\begin{thm}[{\cite[Proposition 3.5]{EleonoraTrapaniMAmeasureoncontactsets}}]\label{measures of rooftop envelope potentials}
    If $f_{0},f_{1} \in C^{1,\bar{1}}(X)$, and if we denote by $\Lambda_{0} = \{ P_{\theta}(f_{0},f_{1}) = f_{0}\}$ and $\Lambda_{1} = \{P_{\theta}(f_{0},f_{1}) = f_{1}\}$, then
    \[
    \theta^{n}_{P_{\theta}(f_{0},f_{1})} = \mathds{1}_{\Lambda_{0}}\theta^{n}_{f_{0}} + \mathds{1}_{\Lambda_{1}\setminus\Lambda_{0}}\theta^{n}_{f_{1}}.
    \]
\end{thm}

A corollary of this result is that 
\begin{cor}\label{measures of rooftop potentials are usually nice}If $u_{0} = P_{\theta}(f_{0})$ and $u_{1} = P_{\theta}(f_{1})$ for $f_{0},f_{1} \in C^{1,\bar{1}}(X)$, then except for at most countably many $\tau \in \R$,
\[
    \theta^{n}_{P_{\theta}(u_{0},u_{1} + \tau)} = \mathds{1}_{\{P_{\theta}(u_{0},u_{1}+\tau) = u_{0}\}}\theta^{n}_{u_{0}} + \mathds{1}_{\{P_{\theta}(u_{0},u_{1}+\tau) = u_{1} + \tau\}}\theta^{n}_{u_{1}}.
    \]
\end{cor}
\begin{proof}
    Since the total measure of $\theta^{n}_{f_{1}}$ is finite, except for countably many $\tau \in \R$, $\theta^{n}_{f_{1}}(\{f_{0} = f_{1} + \tau\}) = 0$. Therefore, except for countably many $\tau \in \R$, 
    \[
    \theta^{n}_{P_{\theta}(f_{0},f_{1}+\tau)} = \mathds{1}_{\{ P_{\theta}(f_{0}, f_{1} + \tau) = f_{0}\}} \theta^{n}_{f_{0}} + \mathds{1}_{\{P_{\theta}(f_{0}, f_{1} + \tau) = f_{1} + \tau \}} \theta^{n}_{f_{1}}.
    \]

    Notice that $P_{\theta}(f_{0},f_{1} + \tau) = P_{\theta}(u_{0}, u_{1} + \tau)$ and Theorem~\ref{thm: measure on contact sets} says that $\theta^{n}_{u_{0}} = \mathds{1}_{\{ P_{\theta}(f_{0}) = f_{0}\}}\theta^{n}_{f_{0}}$ and $\theta^{n}_{u_{1}} = \mathds{1}_{\{P_{\theta}(f_{1}) = f_{1} \} } \theta^{n}_{f_{1}}$. We use this to write 
    \begin{align*}
           &\mathds{1}_{\{P_{\theta}(u_{0}, u_{1} + \tau) = u_{0}\}}\theta^{n}_{u_{0}} + \mathds{1}_{\{P_{\theta}(u_{0}, u_{1}+ \tau) = u_{1}+\tau\}}\theta^{n}_{u_{1}} \\ 
        = &\mathds{1}_{\{P_{\theta}(u_{0}, u_{1} + \tau) = u_{0}\}}\mathds{1}_{\{ P_{\theta}(f_{0}) = f_{0}\}}\theta^{n}_{f_{0}}+ \mathds{1}_{\{P_{\theta}(u_{0}, u_{1}+ \tau) = u_{1}+\tau\}}\mathds{1}_{\{P_{\theta}(f_{1}) = f_{1} \} } \theta^{n}_{f_{1}} \\
        =& \mathds{1}_{\{ P_{\theta}(f_{0}, f_{1} + \tau) = f_{0}\}} \theta^{n}_{f_{0}} + \mathds{1}_{\{P_{\theta}(f_{0}, f_{1} + \tau) = f_{1} + \tau \}} \theta^{n}_{f_{1}} \\
        = &\te^{n}_{P_{\te}(f_{0}, f_{1}+ \tau)} \\
        = &\theta^{n}_{P_{\theta}(u_{0},u_{1}+\tau)}
    \end{align*}
    for all but countably many $\tau \in \R$.
\end{proof}

\subsection{Weak geodesics and rooftop envelopes}\label{subsec: weak geodesics}
Following Berndtsson \cite{Berndtssonweakgeodesics} and Darvas-Di Nezza-Lu \cite{Darvas2017OnTS}, we define weak geodesics as follows. Let $X$ be a compact K\"ahler manifold and let $\theta$ represent a big cohomology class. Let $S = (0,1) \times \R \subset \C$ be the vertical strip in the complex plane. Let $\pi: X \times S \to X$ be the projection map. For $u_{0}, u_{1} \in \PSH(X,\theta)$, a path  $(0,1) \ni t \mapsto v_{t} \in \PSH(X,\theta)$ is a subgeodesic joining $u_{0}$ and $u_{1}$ if the map 
\[
X\times S \ni (x,z) \mapsto v_{\text{Re}(z)}(x)
\]
is a $\pi^{*}\theta$-psh function on $X\times S$ and $\limsup_{t \to 0,1}v_{t} \leq u_{0,1}$. We denote by 
\[
\mathcal{S} = \{ (0,1) \ni t \mapsto v_{t} \in \PSH(X,\theta) : v_{t} \text{ is a subgeodesic joining } u_{0} \text{ and } u_{1}\}.
\]
For arbitrary $u_{0},u_{1} \in \PSH(X,\theta)$, there may not be any subgeodesics joining them. If $u_{0}, u_{1} \in \E(X,\theta)$, then $P_{\theta}(u_{0},u_{1}):= \sup\{u\in \PSH(X,\theta) : u \leq u_{0},u_{1}\}\in \E(X,\theta)$ (see \cite[Theorem 2.10]{Darvas2017OnTS}), so the path $t \mapsto P_{\theta}(u_{0},u_{1})$ is a subgeodesic.

In the case $\mathcal{S}$ is not empty, we define the \emph{weak geodesic} joining $u_{0}$ and $u_{1}$ by
\[
u_{t}(x) = \sup_{v \in \mathcal{S}}v_{t}(x).
\]

Each subgeodesic $v_{t}$ is convex in the $t$-variable. Thus $v_{t} \leq (1-t)u_{0} + t u_{1}$. Taking supremum over all $v \in \mathcal{S}$, we get $u_{t} \leq (1-t)u_{0} + t u_{1}$. Now taking limit $t \to 0,1$ we get $\lim_{t \to 0,1}u_{t} \leq u_{0,1}$. Even if $X \times S \ni (x,z) \mapsto u_{\text{Re}(z)}(x)$ is not $\pi^{*}\theta$-psh, its upper semicontinuous regularization $u^{*}$ is $\pi^{*}\theta$-psh. But for $u^{*}_{t}$, we observe that $u_{t}^{*} \leq ((1-t)u_{0} + tu_{1})^{*} = (1-t)u_{0}  + tu_{1}$. Taking limit to 0 or 1 we get $\lim_{t}u_{t}^{*} \leq u_{0,1}$. Thus $u^{*}_{t}$ is a candidate for $\mathcal{S}$. Hence we do not take the upper semicontinuous regularization in the definition of weak geodesic $u_{t}$. 

If $u_{0}, u_{1} \in \E^{p}(X,\theta)$ then by \cite[Theorem 2.10]{Darvas2017OnTS} $P_{\theta}(u_{0},u_{1}) \in \E^{p}(X,\theta)$. This means the weak geodesic $u_{t}$ joining $u_{0},u_{1}$ satisfy $u_{t} \in \E^{p}(X,\theta)$. The same result holds when $u_{0}, u_{1} \in \E^{p}(X,\theta,\phi)$ due to \cite[Theorem 2.9]{GuptaCompletemetrictopology}.

We recall the following useful lemmas from \cite{Darvas2019GeometricPT}. The results in op. cit. are for the K\"ahler case, but the proofs go through for the big case without change. 

\begin{lem}[{\cite[Lemma 3.16]{Darvas2019GeometricPT}}]\label{geodesics determine rooftop envelops}
Let $u_{0}, u_{1} \in  \PSH(X,\theta)$ and let $u_{t}$ be the weak geodesic joining $u_{0}$ and $u_{1}$. Then for any $\tau \in \R$,
\[
\inf_{t\in (0,1)}(u_{t} - t\tau) = P_{\te}(u_{0}, u_{1} - \tau).
\]
\end{lem}

\begin{proof}
    Since $t \mapsto v_{t} := u_{t} - t\tau$ is the weak geodesic joining $u_{0}$ and $u_{1} - \tau$, it is enough to prove the result for $\tau = 0$. 

    Since $P_{\theta}(u_{0},u_{1}) \leq u_{0},u_{1}$, the map $t \mapsto w_{t}:= P_{\te}(u_{0},u_{1})$ is a weak subgeodesic joining $u_{0}$ and $u_{1}$, therefore $P_{\theta}(u_{0},u_{1}) = w_{t} \leq u_{t}$ for all $t$. Therefore, $P_{\theta}(u_{0},u_{1}) \leq \inf_{t\in(0,1)}(u_{t})$.  

    For the other direction, we notice that Kiselman's minimum principle \cite[Chapter 1, Theorem 7.5]{demaillyagbook} implies $w:= \inf_{t \in (0,1)} (u_{t}) \in \PSH(X,\theta)$. Since $w \leq u_{0},u_{1}$, we have $w \leq P_{\theta}(u_{0},u_{1})$. 
\end{proof}

\begin{lem}[{\cite[Lemma 3.17]{Darvas2019GeometricPT}}]\label{rooftop envelops set and derivatives of geodesics}
Let $u_{0}, u_{1} \in \PSH(X,\theta)$ have minimal singularity and let $u_{t}$ be the weak geodesic joining $u_{0}$ and $u_{1}$. Then for any $\tau \in \R$,
\[
\{ \dot{u}_{0} \geq \tau \} = \{ P_{\te}(u_{0}, u_{1} - \tau) = u_{0}\}
\]
on $X \setminus \{V_{\te} = -\infty\}$.
\end{lem}

\begin{proof}
    Since $u_{0}, u_{1}$ have minimal singularity, $P_{\te}(u_{0}, u_{1}- \tau)$ and $u_{t}$ have minimal singularity as well. Thus on $X \setminus \{V_{\te} = -\infty\}$, $u_{0}, u_{1}, P_{\te}(u_{0}, u_{1} - \tau)$ are all finite. By the previous lemma, $\inf_{t \in (0,1)} (u_{t} - t\tau) = P_{\theta}(u_{0}, u_{1} - \tau)$. Thus for $x \in X$, $P_{\te}(u_{0}, u_{1} - \tau)(x) = u_{0}(x)$ iff $\inf_{t \in (0,1)} (u_{t} - t\tau)(x) = u_{0}(x)$. Since $(u_{t} - t\tau)(x)$ is convex in $t$, this equality is possible iff $\dot{u}_{0}(x) \geq \tau$. 
\end{proof}

Combining Lemma~\ref{rooftop envelops set and derivatives of geodesics} and Corollary~\ref{measures of rooftop potentials are usually nice} we get the following result. 
\begin{thm} \label{thm: left hand speed same as right hand speed}
    Let $u_{0} = P_{\theta}(f_{0})$ and $u_{1} = P_{\theta}(f_{1})$ for $f_{0}, f_{1} \in C^{1,\bar{1}}(X)$. If $u_{t}$ is the weak geodesic joining $u_{0}$ and $u_{1}$, then for all $p\geq 1$, 
    \[
    \int_{X}|\dot{u}_{0}|^{p}\theta^{n}_{u_{0}} = \int_{X}|\dot{u}_{1}|^{p}\theta^{n}_{u_{1}}.
    \]
\end{thm}
\begin{proof}
    The proof is the same as in \cite[Lemma 3.30]{Darvas2019GeometricPT}. We will show that 
    \[
    \int_{\{\dot{u}_{0} > 0\}} |\dot{u}_{0}|^{p} \theta^{n}_{u_{0}} = \int_{\{ \dot{u}_{1} > 0\}} |\dot{u}_{1}|^{p} \theta^{n}_{u_{1}},
    \]
    and a similar proof shows that 
    \[
    \int_{\{\dot{u}_{0} < 0\}} |\dot{u}_{0}|^{p} \theta^{n}_{u_{0}} = \int_{\{\dot{u}_{1} < 0\}} |\dot{u}_{1}|^{p}\theta^{n}_{u_{1}}.
    \]
    \begin{align*}
        \int_{\{\dot{u}_{0} > 0\}} |\dot{u}_{0}|^{p}\theta^{n}_{u_{0}} &= p \int_{0}^{\infty} \tau^{p-1}\theta^{n}_{u_{0}}(\{ \dot{u}_{0} \geq \tau\})d\tau \\
        &= p\int_{0}^{\infty}\tau^{p-1}\theta^{n}_{u_{0}}(\{P_{\te}(u_{0},u_{1}-\tau)\}=u_{0})d\tau.
         \intertext{Corollary~\ref{measures of rooftop potentials are usually nice} imply that $\text{Vol}_{\te}(X) = \te^{n}_{u_{0}}(\{P_{\te}(u_{0},u_{1}-\tau)=u_{0}\}) + \te^{n}_{u_{1}}(\{P_{\te}(u_{0},u_{1}-\tau) = u_{1} - \tau\})$ which gives}
        &= p\int_{0}^{\infty} \tau^{p-1}( \text{Vol}_{\theta}(X) - \theta^{n}_{u_{1}}(\{P_{\te}(u_{0}, u_{1} - \tau) = u_{1} - \tau\})) d\tau\\
        &= p \int_{0}^{\infty} \tau^{p-1}\theta^{n}_{u_{1}}(\{P_{\te}(u_{0}+\tau, u_{1}) < u_{1}\}) d\tau. 
        \intertext{Applying Lemma~\ref{rooftop envelops set and derivatives of geodesics} to the reverse geodesic joining $u_{1}$ and $u_{0}$, we get $\{P_{\te}(u_{0}+ \tau, u_{1})<u_{1}\} = \{\dot{u}_{1} > \tau\}$. Thus}
        &= p \int_{0}^{\infty} \tau^{p-1}\theta^{n}_{u_{1}}(\{\dot{u}_{1} > \tau\})d\tau \\
        &=\int_{\{\dot{u}_{1} > 0\}} |\dot{u}_{1}|^{p}\theta^{n}_{u_{1}}.
    \end{align*}
\end{proof}

The following Lemma from \cite{darvas2019metric} will be useful in constructing some approximations. 
\begin{lem}[{\cite[Lemma 4.3]{darvas2019metric}}]\label{The surprising Lemma from DDLs}
    Let $u, v \in \PSH(X,\theta)$ such that $u\leq v$ and  $\int_{X}\theta^{n}_{u} > 0$ and $b \in \left( 1, \left( \frac{\int_{X}\theta^{n}_{v}}{\int_{X}\theta^{n}_{v} - \int_{X}\theta^{n}_{u}} \right)^{\frac{1}{n}}\right)$, then $P_{\theta}(bu + (1-b)v) \in \PSH(X,\theta)$.  Here, 
    \[
    P_{\te}(bu + (1-b)v) = (\sup\{ h \in \PSH(X,\te) : h \leq bu + (1-b)v)\})^*
    \]
    where $f^*(x) = \limsup_{y \to x} f(y)$ is the upper semicontinuous regularization of $f$. 
\end{lem}

Another useful result we need is 
\begin{lem}[{\cite[Theorem 2.6]{darvas2023relative}}]\label{lem: weak convergence of measures}
    Let $\te^{1}, \dots \te^{n}$ be smooth real closed $(1,1)$-forms representing a big cohomology class and let $u_{j}, u_{j}^{k} \in \PSH(X,\te)$ be such that $u_{j}^{k} \to u_{j}$ in capacity as $k \to \infty$ for all $j \in \{1, \dots, n\}$. If $\chi_{k}, \chi \geq 0$ are quasi-continuous functions that are uniformly bounded and $\chi_{k} \to \chi$ in capacity, then 
    \[
    \liminf_{k \to \infty} \int_{X} \chi_{k} \te^{1}_{u_{1}^{j}} \wedge \dots \wedge \te^{n}_{u_{n}^{k}} \geq \int_{X} \chi \te^{1}_{u_{1}} \wedge \dots \wedge \te^{n}_{u_{n}}.
    \]
    Moreover if 
    \[
    \int_{X} \te^{1}_{u_{1}} \wedge \dots \wedge \te^{n}_{u_{n}} \geq \limsup_{k\to \infty} \int_{X} \te^{1}_{u_{1}^{k}} \wedge \dots \wedge \te^{n}_{u_{n}^{k}},
    \]
    then the measures
    \[
    \chi_{k}\te^{1}u_{1}^{k} \wedge \dots \wedge \te^{n}_{u_{n}^{k}} \to \chi \te^{1}_{u_{1}} \wedge \dots \wedge \te^{n}_{u_{n}}
    \]
    weakly. 
\end{lem}

\subsection{Modifications}
A holomorphic map $\mu: \tilde{X} \to X$ between compact K\"ahler manifolds $(\tilde{X}, \tilde{\om})$ and $(X,\om)$ is called a modification if outside a closed analytic set $E \subset \tilde{X}$, $\mu:\tilde{X} \setminus E \to X\setminus \mu(E)$ is a biholomorphism and $\mu(E) \subset X$ is also a closed analytic subset. We say that $E$ is the exceptional set, and $\mu(E)$ is the center of the modification. In this paper, modifications arise from resolving singularities of quasi plurisubharmonic functions with analytic singularity type. See Section~\ref{sec: analytic singularities to big and nef} for more details.

If $\theta$ is a smooth closed $(1,1)$-form on $X$ representing a big class, then $\mu^{*}\theta$ is also a big class on $\tilde{X}$ (see \cite[Proposition 4.12]{Boucksomvolumeofalinebundle}). If $u \in \PSH(X,\theta)$, then $u \circ \mu \in \PSH(\tilde{X},\mu^{*}\theta)$. In this particular case, the reverse is also true. 
\begin{lem}\label{lem: on pushforward of psh functions}
    Let $\mu: \tilde{X} \to X$ be a modification with exceptional set $E$ and center $\mu(E)$. If $\theta$ represents a big cohomology class on $X$ and $v \in \PSH(\tilde{X}, \mu^{*}\theta)$, then there exists a unique $u \in \PSH(X,\theta)$ such that $v = u\circ \mu$. 
\end{lem}

\begin{proof}
    Since $\mu : \tilde{X} \setminus E \to X \setminus \mu(E)$ is a biholomorphism, we know $v \circ \mu^{-1}$ is a $\theta$-psh function on $X\setminus \mu(E)$. Since $\mu(E)$ is an analytic set and $v \circ \mu^{-1}$ is bounded from above, it extends over $\mu(E)$ to all of $X$. We call this extension $u$. Thus there exists $u \in \PSH(X,\theta)$ such that on $\tilde{X}\setminus E$, $u\circ \mu = v$. Since both $u\circ \mu$ and $v$ are $\mu^{*}\theta$-psh functions that agree almost everywhere, they must agree everywhere. Thus $u\circ \mu = v$. By the same argument, $u$ is unique as well.
\end{proof}

In general, we can pullback smooth forms and push forward currents. However, for positive $(1,1)$-currents, we can define the pullback as follows. If $u\in \PSH(X,\theta)$, then $\mu^{*}(\theta_{u}) := \mu^{*}\theta + dd^{c}u\circ \mu$. Moreover, it satisfies $\mu_{*}\mu^{*}\theta_{u} = \theta_{u}$. We recall 

\begin{thm}[{\cite[Theorem 3.1]{DiNezzastabilityofmongeampereclasses}}]\label{thm :DiNezza on pluripolar product being invariant} If $\mu : \tilde{X} \to X$ is a modification and $\theta_{1} \dots \theta_{n}$ are real smooth closed $(1,1)$-forms on $\tilde{X}$ representing big cohomology classes and $u_{j} \in \PSH(\tilde{X},\theta_{j})$, then 
\[
\mu_{*}\langle \theta_{1,u_{1}} \wedge \dots \wedge \theta_{n,u_{n}} \rangle = \langle \mu_{*}\theta_{1,u_{1}} \wedge \dots \wedge \mu_{*} \theta_{n,u_{n}}\rangle.
\]
\end{thm}

Applying this theorem to $\mu^{*}\theta_{u}$, we obtain that $\mu_{*}((\mu^{*}\theta)_{u\circ \mu}^{n}) = \theta^{n}_{u}$.

\subsection{Spaces of finite entropy}\label{subsec: space of finite entropy}
If $\theta$ represents a big class, and $\phi \in \PSH(X,\theta)$ is a model potential, we say that $u \in \PSH(X,\theta,\phi)$ has finite entropy if the corresponding non-pluripolar measure $\theta^{n}_{u}$ has finite entropy with respect to the background K\"ahler volume form $\om^{n}$. We define 
\[
\text{Ent}(\om^{n}, \theta^{n}_{u}) = \int_{X} \log \left( \frac{ \theta^{n}_{u}}{\om^{n}} \right) \theta^{n}_{u}
\]
if $\theta^{n}_{u}$ has a density with respect to $\om^{n}$ and the entropy is $+\infty$ otherwise. 

We denote by 
\[
\text{Ent}(X,\theta,\phi) = \{ u \in \E(X,\theta, \phi) : \int_{X} \log \left( \frac{\theta^{n}_{u}}{\om^{n}}\right) \theta^{n}_{u} < \infty\}.
\]

The following lemma tells us that pulling back a potential of finite entropy under a modification still has finite entropy. This observation is also made in \cite{dinezza2023entropy}. We give a proof here for completeness. 

\begin{lem}\label{lem: one way finite entropy}
    If $\mu : \tilde{X} \to X$ is a modification and $u \in \text{Ent}(X,\theta)$ has finite entropy, then $u\circ \mu \in \text{Ent}(\tilde{X}, \mu^{*}\theta)$. 
\end{lem}

\begin{proof}
     Let $\theta^{n}_{u} = f \om^{n}$ . Also assume that $\tilde{\om}$, the K\"ahler form on $\tilde{X}$ has $\int_{\tilde{X}} \tilde{\om}^{n} = 1$. Let $g \in C^{\infty}(X)$ be the function such that $(\mu^{*}\om)^{n} = g \tilde{\om}^{n}$. Then $(\mu^{*}\theta + dd^{c}u\circ \mu)^{n} = f\circ \mu (\mu^{*}\om)^{n} = f\circ \mu \cdot g \tilde{\om}^{n}$. 
    To show that $u \circ \mu$ has finite entropy, we need to show that 
    \[
    \int_{\tilde{X}}\log(f\circ \mu \cdot g ) f\circ \mu \cdot g \tilde{\om}^{n}
    \]
    is bounded from above. 
    \begin{align*}
        \int_{\tilde{X}}\log(f\circ \mu \cdot g ) f\circ \mu \cdot g \tilde{\om}^{n} &= \int_{\tilde{X}} \log(f\circ \mu \cdot g) f\circ \mu (\mu^{*}\om)^{n} \\
        &= \int_{\tilde{X}} \log(f\circ \mu)f\circ \mu (\mu^{*}\om)^{n} + \int_{\tilde{X}}\log(g) f\circ \mu (\mu^{*}\om)^{n}.
        \intertext{Since $g$ is bounded from above, we have $\log(g) \leq C$ where $C$ is a constant. In the first integral, we can push it forward to $X$}
        &\leq \int_{X} f\log(f) \om^{n} +C \int_{\tilde{X}} f\circ \mu (\mu^{*}\om)^{n} \\
        &= \int_{X}f\log(f) \om^{n} + C \int_{X} f\om^{n} \\
        &= \int_{X}f\log (f) \om^{n} + C\int_{X}\theta^{n}_{u}.
    \end{align*}
    Since the entropy of $\theta^{n}_{u} = f\om^{n}$ is bounded, the above integral is finite. Hence $u \circ \mu \in \text{Ent}(\tilde{X},\mu^{*}\theta)$. 
\end{proof}

We recall another result from \cite{dinezza2023entropy}.
\begin{lem}[{\cite[Proposition 2.3]{dinezza2023entropy}}]\label{lem: quasi-smoooth potentials have bounded entropy}
    If $f \in C^{1,\bar{1}}(X)$, then $P_{\theta}[\phi](f) \in \text{Ent}(X,\theta,\phi)$.
\end{lem}

\subsection{Monge-Amp\`ere energy} 
For a smooth closed real $(1,1)$-form $\theta$ that represents a big cohomology class, we define the Monge-Amp\`ere Energy for $u \in \PSH(X,\theta)$ with  minimal singularities by 
\[
I(u) = \frac{1}{(n+1)} \sum_{j=0}^{n} \int_{X} (u-V_{\theta}) \theta^{j}_{u} \wedge \theta^{n-j}_{V_{\theta}}.
\]

We recall 
\begin{thm}[{\cite[Theorem 3.12]{Darvas2017OnTS}}]\label{thm: linearity of MongeAmpere along geodesic in big case}
    If $u_{0}, u_{1} \in \PSH(X,\theta)$ have minimal singularities, then the Monge-Amp\`ere energy is linear along the weak geodesic. More precisely, if $u_{t}$ is the weak geodesic joining $u_{0}$ and $u_{1}$, then 
    \[
    I(u_{t}) = (1-t)I(u_{0}) + tI(u_{1}).
    \]
\end{thm}

\subsection{Metric geometry in the big and nef case}\label{subsec: metric geometry in big and nef case}
The metric geometry of $\E^{p}(X,\bb)$, when $\bb$ represents a big and nef cohomology class, was studied by Di Nezza-Lu in \cite{dinezza2018lp}. We will briefly describe how they defined the $d_{p}$ metric on $\E^{p}(X,\bb)$. They defined 
\begin{equation}\label{eq: definition of Hbb}
\cH_{\bb} = \{ u \in \PSH(X,\bb) \, | \, u = P_{\bb}(f) \text{ for } f \in C(X) \text{ such that } dd^{c}f \leq C(f)\om\}. \end{equation}

As $\bb$ is big and nef, $\om_{\ee}:= \bb + \ee\om$ represents a K\"ahler class, although it may not be a K\"ahler form. The metric $d_{p}$ on $\cH_{\bb}$ is defined by approximation from $\E^{p}(X,\om_{\ee})$. In particular, if $u_{0}, u_{1} \in \cH_{\bb}$, such that $u_{0} = P_{\bb}(f_{0})$ and $u_{1} = P_{\bb}(f_{1})$, then we define $u_{0, \ee} = P_{\om_{\ee}}(f_{0})$ and $u_{1,\ee} = P_{\om_{\ee}}(f_{1})$ and 
\[
d_{p}(u_{0}, u_{1}) := \lim_{\ee \to 0} d_{p}(u_{0,\ee}, u_{1,\ee}).
\]
More generally, on $\E^{p}(X,\bb)$, the metric $d_{p}$ is defined by approximation from $\cH_{\bb}$. In particular, if $u_{0}, u_{1} \in \E^{p}(X,\bb)$, then we can find $u_{0}^{j}, u_{1}^{j} \in \cH_{\bb}$ such that $u_{0}^{j} \searrow u_{0}$ and $u_{1}^{j} \searrow u_{1}$ and we define 
\[
d_{p}(u_{0}, u_{1}) := \lim_{j\to \infty} d_{p}(u_{0}^{j}, u_{1}^{j}).
\]

In \cite{dinezza2018lp}, Di Nezza-Lu proved

\begin{thm}[\cite{dinezza2018lp}]
    If $\bb$ represents a big and nef cohomology class, then the function $d_{p}$ defined as above is a complete geodesic metric on $\E^{p}(X,\bb)$. They also showed in the proof of \cite[Theorem 3.17]{dinezza2018lp} that the weak geodesic $u_{t}$ joining $u_{0}, u_{1} \in \E^{p}(X,\bb)$ are metric geodesics as well.
\end{thm}

We list some properties of $(\E^{p}(X,\bb),d_{p})$ from their paper that we will frequently use. 

\begin{thm}[Pythagorean identity, {\cite[Theorem 3.14]{dinezza2018lp}}]\label{thm: Pythagorean identity in big and nef} If $u,v \in \E^{p}(X,\bb)$, then 
\[
d_{p}^{p} (u,v) = d_{p}^{p}(u, P_{\bb}(u,v)) + d_{p}^{p}(v, P_{\bb}(u,v)).
\]
\end{thm}
For $u_{0}, u_{1} \in \E^{p}(X,\bb)$ we define 
\[
I_{p}(u,v) = \int_{X} |u-v|^{p}(\bb^{n}_{u} + \bb^{n}_{v}).
\]
The following theorem shows that $I_{p}$ controls the distance $d_{p}$. 

\begin{thm}[{\cite[Proposition 3.12]{dinezza2018lp}}]\label{thm: control by p-energy} Given $u_{0}, u_{1} \in \E^{p}(X,\bb)$, there exists a constant $C> 1$ that depends only on the dimension, such that 
\[
\frac{1}{C} I_{p}(u_{0},u_{1}) \leq d_{p}^{p} (u_{0}, u_{1}) \leq C I_{p}(u_{0},u_{1}).
\]
\end{thm}

We recall the following

\begin{thm}[{\cite[Theorem 1.2]{DiNezzaLugeodesicdistanceandmameasure}}]\label{thm: distance for finite entropy} If $\bb$ represents a big and nef cohomology class, $u_{0},u_{1} \in \text{Ent}(X,\bb)$ have minimal singularity type, and $u_{t}$ is the weak geodesic joining $u_{0}$ and $u_{1}$, then 
\[
d_{p}(u_{0}, u_{1}) = \int_{X} |\dot{u}_{t}|^{p}\bb^{n}_{u_{t}} \qquad \text{ for all } t \in [0,1].
\]
\end{thm}

In the case when $p = 1$, we have some special properties for the distance $d_{1}$. 

\begin{thm}[{\cite[Theorem 3.18]{dinezza2018lp}}]\label{thm: d1 in big and nef} If $u_{0}, u_{1} \in \E^{1}(X,\bb)$, then 
\[
d_{1}(u_{0}, u_{1}) = I(u_{0}) + I(u_{1}) -2 I(P_{\bb}(u_{0},u_{1})).
\]  
\end{thm}

This theorem allows us to have the following stronger result when the potentials $u_{0}$ and $u_{1}$ are comparable. 
\begin{lem}\label{lem: comparable potential in E1 and big and nef}
    If $\bb$ represents a big and nef cohomology class,  $u_{0}, u_{1} \in \text{Ent}(X,\bb)$ having minimal singularity satisfy $u_{0} \leq u_{1}$, and $u_{t}$ is the weak geodesic joining $u_{0}$ and $u_{1}$, then 
    \[
    I(u_{1}) - I(u_{0}) = \int_{X} \dot{u}_{t}\bb^{n}_{u_{t}} \qquad \text{ for all } t \in [0,1].
    \]
\end{lem}

\begin{proof}
    Since the path $w_{t} \mapsto u_{0}$ is a subgeodesic joining $u_{0}$ and $u_{1}$, therefore $u_{t} \geq u_{0}$. This means  that $\dot{u}_{0} \geq 0$. By convexity of $u_{t}$ in the $t$ variable, we get that $0 \leq \dot{u}_{0} \leq \dot{u}_{t}$.  

    Since $u_{0} \leq u_{1}$, we have $P_{\bb}(u_{0},u_{1}) = u_{0}$. Thus Theorem~\ref{thm: d1 in big and nef} implies 
    \[
    d_{1}(u_{0}, u_{1}) = I(u_{1}) - I(u_{0}).
    \]
    On the other hand, Theorem~\ref{thm: distance for finite entropy} along with the observation that $\dot{u}_{t} \geq 0$ imply that
    \[
    d_{1}(u_{0}, u_{1}) = \int_{X}|\dot{u}_{t}|\bb^{n}_{u_{t}} = \int_{X}\dot{u}_{t} \bb^{n}_{u_{t}} \qquad \text{ for all } t \in [0,1].
    \]
    Combining the two expressions for $d_{1}(u_{0},u_{1})$ we get 
    \[
    I(u_{1}, u_{0}) = \int_{X} \dot{u}_{t} \bb^{n}_{u_{t}} \qquad \text{ for all } t \in [0,1].
    \]
\end{proof}

When $\theta$ is big, and not necessarily nef, we can have the above result in a slightly restrictive setting as in the following lemma. 

\begin{lem}\label{lem: monge ampere energy and geodesic in big case}
    Let $\theta$ represent a big cohomology class and let $u_{0} = P_{\theta}(f_{0})$ and $u_{1} = P_{\theta}(f_{1})$ for $f_{0},f_{1} \in C^{1,\bar{1}}(X)$ satisfy $u_{0} \leq u_{1}$. If $u_{t}$ is the weak geodesic joining $u_{0}$ and $u_{1}$, then 
    \[
        I(u_{1}) - I(u_{0}) = \int_{X} \dot{u}_{0}\theta^{n}_{u_{0}} = \int_{X} \dot{u}_{1} \theta^{n}_{u_{1}}.
    \]
\end{lem}

\begin{proof}
    The proof extends the ideas in the proof of \cite[Proposition 3.18]{dinezza2018lp} to the big case. The idea is to use Theorem~\ref{thm: left hand speed same as right hand speed}, Theorem~\ref{thm: linearity of MongeAmpere along geodesic in big case} along with \cite[Theorem 2.4]{darvas2021l1} which says that for $u, v \in \PSH(X,\theta)$ with minimal singularity type, $\int_{X} (u-v)\theta^{n}_{u} \leq I(u) - I(v) \leq \int_{X} (u-v)\theta^{n}_{v}$. 

    By convexity of the geodesic $u_{t}$ in the $t$-direction, we have $0 \leq \dot{u}_{0} \leq \dot{u}_{t} \leq \dot{u}_{1}$. Thus $u_{t}$ is increasing with $t$. Thus we have 
    \begin{align*}
        \int_{X} \dot{u}_{0} \theta^{n}_{u_{0}} &= \int_{X} \lim_{t \to 0} \frac{u_{t} - u_{0}}{t} \theta^{n}_{u_{0}} \\
        &= \lim_{t\to 0} \int_{X} \frac{u_{t} - u_{0}}{t} \theta^{n}_{u_{0}} \\
        &\geq \lim_{t \to 0} \frac{I(u_{t}) - I(u_{0})}{t} \\
        &= \lim_{t\to 0} I(u_{1}) - I(u_{0}).
    \end{align*}
    In the second line, we could exchange limit with integral because of the convexity of $u_{t}$ in the $t$ variable and the monotone convergence theorem. In the third line, we used the inequality mentioned above, and in the last line, we used that $I$ is affine along the weak geodesics. Similarly, we can show that 
    \begin{align*}
        \int_{X} \dot{u}_{1} \theta^{n}_{u_{1}} &= \int_{X} \lim_{t\to 1} \frac{u_{1} - u_{t}}{1-t} \theta^{n}_{u_{1}} \\
        &= \lim_{t\to 1} \int_{X} \frac{u_{1}- u_{t}}{1-t} \theta^{n}_{u_{1}} \\
        &\leq \lim_{t \to 1} \frac{I(u_{1}) - I(u_{t})}{1-t} \\
        &= I(u_{1}) - I(u_{0}).
    \end{align*}
    Combining these two we get $\int_{X} \dot{u}_{0}\theta^{n}_{u_{0}} \geq I(u_{1}) - I(u_{0}) \geq \int_{X} \dot{u}_{1} \theta^{n}_{u_{1}}$. Combining with Theorem~\ref{thm: left hand speed same as right hand speed} for $p=1$, we get that 
    \[
    \int_{X} \dot{u}_{0}\theta^{n}_{u_{0}} = I(u_{1}) - I(u_{0}) = \int_{X}\dot{u}_{1}\theta^{n}_{u_{1}}.
    \]
\end{proof}

\section{From the Big and Nef to the Prescribed Analytic Singularity } \label{sec: analytic singularities to big and nef}
 $(X,\om)$ be a compact K\"ahler manifold and $\theta$ be a closed smooth $(1,1)$-form representing a big cohomology class. We fix $\psi \in \PSH(X,\theta)$ a model potential that has analytic singularities of type $(\mathcal{I},c)$. By Hironaka's embedded desingularization theorem, we can find a modification $\mu: \tilde{X} \to X$ such that $\mu^{*}\mathcal{I} = \mathcal{O}(-E)$ where $E = \sum_{i} \lambda_{i} E_{i}$ is a simple normal crossing divisor.   We can choose metrics $h_{i}$ on $\mathcal{O}(E_{i})$ and canonical sections $s_{i}$ of $\mathcal{O}(E_{i})$. Let $R_{h_{i}}$ be the curvature for the metrics $h_{i}$ on $\mathcal{O}(E_{i})$. We denote 
\[
|s|^{2}_{h} = \prod_{i=1}^{k}|s_{i}|^{2\lambda_{i}}_{h_{i}} \quad \text{ and } \quad R_{h} = \sum_{i=1}^{k} \lambda_{i} R_{h_{i}}
\]

Thus for this modification, we have 
\[
\psi \circ \mu = c \log |s|^{2}_{h} + g
\]
where $g$ is a bounded function. See \cite[Section 5.9]{demaillyanalyticmethods} for more details. 

Now, $\mu^{*} \theta  +dd^{c}\psi \circ \mu \geq 0$. Thus $\mu^{*}\theta + c dd^{c}\log|s|^{2}_{h} + dd^{c}g \geq 0$. By the Ponicar\'e-Lelong formula
\[
[[E]] = R_{h} + dd^{c}\log|s|^{2}_{h},
\]
where $[[E]]$ is the current of integration along $E$, we can write $\mu^{*}\theta -cR_{h} + c[[E]] + dd^{c}g \geq 0$. Define
\begin{equation}\label{eq: equation for theta tilde}
    \tilde{\theta} = \mu^{*}\theta - cR_{h},
\end{equation}
so that 
\begin{equation}\label{eq: primary modification equation}
\mu^{*}\theta + dd^{c}(\psi\circ \mu) = \tilde{\theta} +c[[E]]+dd^{c}g. 
\end{equation}

 On $\tilde{X}\setminus E$ (we abuse the notation to denote by $E$  the analytic set on which the divisor $E$ is supported), $\tilde{\theta} + dd^{c} g \geq 0$. As $g$ is bounded from above, $g$ extends uniquely to all of $\tilde{X}$ to a $\tilde{\theta}$-psh function $g$. Thus $\tilde{\theta} + dd^{c}g \geq 0$ on all of $\tilde{X}$. Since $g$ is a bounded $\tte$-psh function, $\tilde{\theta}$ represents a nef class. This follows from the following argument using Demailly's regularization theorem. 

 Since $\tilde{\theta} + dd^{c} g\geq 0$, we have $\tilde{\theta} + \ee \tilde{\om} + dd^{c}g \geq \ee \tilde{\om}$, where $\tilde{\om}$ is an arbitrary K\"ahler form on $\tilde{X}$. Demailly's regularization theorem implies there is a K\"ahler potential $\psi$ in the class $\{\theta + \ee \tilde{\om}\}$ with analytic singularities such that $\psi \geq g$ which is smooth outside its singular locus. Since $g$ is bounded from below, $\psi$ has no singular locus, thus $\psi$ is a smooth K\"ahler potential in $\{\tilde{\theta} + \ee \tilde{\om}\}$, so $\tilde{\theta} +\ee\tilde{\om}$ is a K\"ahler class. This shows that $\tilde{\theta}$ is nef. 

We can go back and forth between the spaces $\PSH(X,\theta,\psi)$ and $\PSH(\tilde{X},\tilde{\theta})$ that preserves various pluripotential theoretic relationships.The following theorem describes the correspondence between $\PSH(X,\theta,\psi) \leftrightarrow \PSH(\tilde{X},\tilde{\theta})$. This correspondence is well known in the community (see \cite[Lemma 4.3]{darvas2023twisted} and \cite[Section 4.1]{Trusianicontinuitymethod}), but we write a proof here for completeness, as our definition of analytic singularities is slightly more general than in \cite{Trusianicontinuitymethod}. 

\begin{thm}\label{thm: bijection between analytic singularity and big and nef}
Let $\theta$ represent a big cohomology class on $X$ and $\psi \in \PSH(X,\theta)$ has analytic singularities. Let $\mu: \tilde{X} \to X$ be the desingularization of the singularities of $\psi$ and $\tilde{\theta}$ be a closed smooth $(1,1)$-form on $\tilde{X}$ as described above. Then the map $\PSH(X,\theta, \psi) \ni   u \mapsto \tilde{u}:= (u-\psi)\circ \mu + g \in \PSH(\tilde{X}, \tilde{\theta})$ is an order-preserving bijection. 
\end{thm}

\begin{proof}
    Let $u \in \PSH(X,\theta,\psi)$. On $X\setminus E$, $dd^{c}\log|s|^{2}_{h} +R_{h} = 0$. Thus on $X\setminus E$, $u\circ \mu - c\log|s|^{2}_{h}$ is a $(\mu^{*}\theta - cR_{h})$-psh function. As $u\circ \mu - c\log|s|^{2}_{h} = u\circ \mu - \psi \circ \mu + g$ and $(u-\psi)\circ \mu$ is bounded from above, we get $(u-\psi)\circ \mu + g$ is bounded from above, so it extends to a $(\mu^{*}\theta - cR_{h})$-psh function on all of $\tilde{X}$. As $\tilde{\theta} = \mu^{*}\theta - cR_{h}$, $(u-\psi)\circ \mu + g$ is $\tilde{\theta}$-psh.

Now we go in the other direction. Let $v \in \PSH(\tilde{X},\tilde{\theta})$. So $\tilde{\theta} + dd^{c}v \geq 0$.  From Equation~\ref{eq: primary modification equation},
\[
\mu^{*}\theta - c[[E]] + dd^{c}(\psi\circ \mu-g + v) \geq 0.
\]
Thus 
\[
\mu^{*}\theta + dd^{c}(\psi\circ \mu-g + v) \geq 0.
\]
Thus $(\psi\circ \mu -g+ v)$ is a $\mu^{*}\theta$-psh function. From Lemma~\ref{lem: on pushforward of psh functions}, we see that there exists a unique $u \in \PSH(X,\theta)$ such that $u\circ \mu = \psi\circ\mu-g + v$.  On $X\setminus \mu(E)$, $u = \psi -g\circ \mu^{-1} + v\circ \mu^{-1} \leq \psi +C$. Thus this inequality holds everywhere. Thus $u \in \PSH(X,\theta,\psi)$. 

Clearly, the map $u \mapsto (u-\psi)\circ \mu+g$ is order-preserving. 
\end{proof}

\begin{cor}\label{cor: minimal singularity to minimal singularity}
In the bijection, $\PSH(X,\theta,\psi) \ni u \mapsto \tilde{u} := (u-\psi)\circ \mu + g \in \PSH(\tilde{X}, \tilde{\theta})$, $u$ has the same singularity type as $\psi$ if and only if $\tilde{u}$ has minimal singularity type. 
\end{cor}

\begin{proof}
    If $u$ has the same singularity type as $\psi$, then for some $C$, $\psi - C \leq u$. Thus $-C \leq u-\psi$. Thus $-C \leq (u-\psi)\circ \mu = \tilde{u} - g$. As $g$ is bounded, we get $\tilde{\mu}$ has the minimal singularity type. 

    Similarly, if $\tilde{u}$ has the minimal singularity type, then $-C \leq (u-\psi)\circ \mu + g$. Thus on $X\setminus \mu(E)$, $-C \leq u-\psi + g\circ \mu^{-1}$ that implies $\psi - C' \leq u$ as $g$ is bounded. Since both are $\theta$-psh functions, the inequality holds everywhere, therefore $\psi - C' \leq u$, hence $u$ has the same singularity type as $\psi$. 
\end{proof}

Now we will describe how the bijection described above preserves the non-pluripolar product. 
\begin{thm}\label{thm: pushforward of non-pluripolar measure}
    Given $u_{1}, \dots, u_{n} \in \PSH(X,\theta,\psi)$, and corresponding $\tilde{u}_{j} := (u_{j} - \psi)\circ \mu + g \in \PSH(\tilde{X},\tilde{\theta})$, their non-pluripolar product satisfy 
    \[
    \mu_{*}\langle \tilde{\theta}_{\tilde{u}_{1}} \wedge \dots \wedge \tilde{\theta}_{\tilde{u}_{n}}\rangle = \langle\theta_{u_{1}} \wedge \dots \wedge \theta_{u_{n}}\rangle.
    \]
\end{thm}

\begin{proof}From Equation~\eqref{eq: primary modification equation}, we can write 
\[
\tilde{\theta}  = \mu^{*}\theta - c[[E]] + dd^{c}\psi\circ\mu - dd^{c}g.
\]

Adding $dd^{c}\tilde{u}_{j}$ both sides we get 
\[
\tilde{\theta} + dd^{c}\tilde{u}_{j} = \mu^{*}\theta + dd^{c}u_{j}\circ \mu - c[[E]].
\]
Taking the non-pluripolar part, we get 
\[
\langle \tilde{\theta} + dd^{c}\tilde{u}_{j} \rangle = \langle \mu^{*}\theta + dd^{c}u_{j}\circ \mu \rangle.
\]
Now we take the non-pluripolar product to get 
\[
\langle \tilde{\theta}_{\tilde{u}_{1}} \wedge \dots \wedge \tilde{\theta}_{\tilde{u}_{n}} \rangle = \langle\mu^{*}(\theta_{u_{1}}) \wedge \dots \wedge \mu^{*}(\theta_{u_{n}})\rangle.
\]
Taking push-forward of both the measures, applying Theorem~\ref{thm :DiNezza on pluripolar product being invariant}, and observing that $\mu_{*}\mu^{*}(\theta_{u_{j}}) = \theta_{u_{j}}$ we get 
\[
\mu_{*}\langle \tilde{\theta}_{\tilde{u}_{1}} \wedge \dots \wedge \tilde{\theta}_{\tilde{u}_{n}}\rangle = \langle \theta_{u_{1}} \wedge \dots \wedge \theta_{u_{n}} \rangle
\]
as desired. 
\end{proof}

A consequence of the above theorem is that the bijection $\PSH(X,\theta,\psi) \leftrightarrow \PSH(\tilde{X},\tilde{\theta})$ preserves the mass and the finite energy classes of the potentials. 

\begin{cor}\label{cor: bijection preserves mass and energy}
Under the bijection $\PSH(X,\theta,\psi) \ni u \mapsto \tilde{u}:= (u-\psi)\circ \mu + g \in \PSH(\tilde{X},\tilde{\theta})$, we have $\int_{X} \theta^{n}_{u} = \int_{\tilde{X}} \tilde{\theta}^{n}_{\tilde{u}}$ and 
\[
\int_{X} |u-\psi|^{p}\theta^{n}_{u} < \infty  \iff \int_{\tilde{X}} |\tilde{u}-V_{\tilde{\theta}}|^{p} \tilde{\theta}_{\tilde{u}} < \infty
\]
Thus the map $u \mapsto \tilde{u}$ is also a bijection between $\E^{p}(X,\theta,\psi)$ and $\E^{p}(\tilde{X},\tilde{\theta})$. 
\end{cor}

\begin{proof}
    Applying Theorem~\ref{thm: pushforward of non-pluripolar measure} to any potential $u \in \PSH(X,\theta,\psi)$, we get that $\mu_{*}\tilde{\theta}_{\tilde{u}}^{n} = \theta^{n}_{u}$. Integrating it we find
    \[
    \int_{X} \theta^{n}_{u} = \int_{X} \mu_{*} \tilde{\theta}^{n}_{\tilde{u}} = \int_{\tilde{X}} \tilde{\theta}^{n}_{\tilde{u}}. 
    \]
    Thus $u$ and $\tilde{u}$ have the same mass. Similarly, integrating the function $|u-\psi|^{p}$ we get 
    \[
    \int_{X}|u-\psi|^{p}\theta^{n}_{u} = \int_{X} |u-\psi|^{p}\mu_{*}\tilde{\theta}^{n}_{\tilde{u}} = \int_{\tilde{X}} |(u-\psi)\circ \mu|^{p} \tilde{\theta}^{n}_{\tilde{u}} = \int_{\tilde{X}}|\tilde{u}-g|^{p}\tilde{\theta}^{n}_{\tilde{u}}.
    \]
    Now if $\int_{\tilde{X}} |\tilde{u} - V_{\tilde{\theta}}|^{p}\tte_{\tilde{u}}^{n} < \infty$, then 
    \[
    \int_{\tX} |\tu - g|^{p}\tte^{n}_{\tu} = \int_{\tX} |\tu - V_{\tte} - (g- V_{\tte})|^{p} \tte^{n}_{\tu} \leq 2^{p-1}\left( 
    \int_{\tX} |\tu - V_{\tte}|^{p}\tte^{n}_{\tu} + \int_{\tX} |g-V_{\tte}|^{p} \tte^{n}_{\tu}\right) < \infty.
    \]
    Here we used the Minkowski's inequality ($|a+b|^{p} \leq 2^{p-1}(|a|^{p} + |b|^{p})$ if $p\geq 1$), and the fact that $g$ and $V_{\tte}$ are bounded functions. We can show the other side in the same manner. If $\int_{\tX}|\tu - g|^{p}\tte^{n}_{\tu} < \infty$, then 
    \[
    \int_{\tX}|\tu-V_{\tte}|^{p}\tte^{n}_{\tu} = \int_{\tX} |\tu-g+g-V_{\tte}|^{p}\tte^{n}_{\tu} \leq 2^{p-1}\left( \int_{\tX}|\tu - g|^{p} \tte^{n}_{\tu} + \int_{\tX} |g-V_{\tte}|^{p}\tte^{n}_{\tu}\right) < \infty.
    \]
\end{proof}

The bijection $\PSH(X,\theta,\psi) \leftrightarrow \PSH(\tilde{X},\tilde{\theta})$ does not preserve the finite entropy classes in both directions. But we have

\begin{lem}\label{lem: entropy in desingularization}
    If $u \in \PSH(X,\theta,\psi)$ has finite entropy. Then $\tilde{u} = (u-\psi)\circ \mu + g\in \PSH(\tilde{X},\tilde{\theta})$ has finite entropy as well.
\end{lem}

\begin{proof}
    Recall from Lemma~\ref{lem: one way finite entropy} that $u\circ \mu \in \text{Ent}(X,\mu^{*}\theta)$. Thus the measure $\langle (\mu^{*}\theta  + dd^{c}u \circ \mu)^{n} \rangle$ has finite entropy with respect to the background K\"ahler volume form $\tilde{\om}^{n}$ on $\tilde{X}$. As non-pluripolar measures, we know that $\tilde{\theta}_{\tilde{u}}^{n} = \langle(\mu^{*}\theta + dd^{c}u\circ \mu)^{n} \rangle$, we get that the measure $\tilde{\theta}_{\tilde{u}}^{n}$ has finite entropy as well. Thus $\tilde{u}$ has finite entropy in $\PSH(\tilde{X},\tilde{\theta})$ as well.
\end{proof}

The bijective correspondence between $\PSH(\tilde{X},\tilde{\theta}) \leftrightarrow \PSH(X,\theta,\psi)$ preserves the weak geodesics. 

\begin{thm}\label{thm: Mabuchi geodesics are preserved}
    If $u_{t} \in \PSH(X,\theta,\psi)$ is the weak geodesic joining $u_{0}, u_{1} \in \PSH(X,\theta, \psi)$, then $\tilde{u}_{t} \in\PSH(\tilde{X},\tilde{\theta})$ is the weak geodesic joining $\tilde{u}_{0}, \tilde{u}_{1} \in \PSH(\tilde{X},\tilde{\theta})$. 
\end{thm}

\begin{proof}
    First, we will show that a subgeodesic $(0,1) \ni t \to v_{t} \in \PSH(X,\theta,\psi)$ maps to a subgeodesic $(0,1) \ni t \to \tilde{v}_{t} := (v_{t}-\psi)\circ \mu +g \in \PSH(\tilde{X},\tilde{\theta})$ and vice versa.
    
    We consider the following diagram of maps. 
    \[
    \begin{tikzcd}
        \tilde{X} \times S \arrow[swap]{d}{\mu\times \text{id}}\arrow{r}{\tilde{\pi}} & \tilde{X} \arrow{d}{\mu}\\
        X \times S \arrow{r}{\pi} & X
    \end{tikzcd}
    \]
    We will show that the map $\tilde{X} \times S \ni (x,z) \mapsto (v_{\text{Re}(z)} - \psi)\circ \mu(x) + g(x)$ is $\tilde{\pi}^{*}\tilde{\theta}$-psh map. As earlier, we will show that it is true on $(\tilde{X} \setminus E) \times S$ and then use boundedness of $(v_{\text{Re}(z)} - \psi)\circ \mu + g$ to conclude that it's true on all of $\tilde{X} \times S$. 

    Since $t \mapsto v_{t} \in \PSH(X,\theta,\psi)$ is a subgeodesic we get 
    \begin{align*}
        \pi^{*}\theta + dd^{c}v_{\text{Re}(z)}(x) &\geq 0.
        \intertext{Pull it back by $\mu \times \text{id}$ so that }
        (\mu\times \text{id})^{*}\pi^{*}\theta + dd^{c}(v_{\text{Re}(z)}\circ \mu(x)) &\geq 0.
        \intertext{Using the fact that $\pi \circ (\mu \times \text{id}) = \mu \circ \tilde{\pi}$ we get}
        \tilde{\pi}^{*}\mu^{*}\theta + dd^{c}(v_{\text{Re}(z)}\circ \mu(x)) &\geq 0 \\
        \implies \tilde{\pi}^{*}(\mu^{*} \theta + dd^{c}\psi) + dd^{c}((v_{\text{Re}(z)} - \psi) \circ \mu(x)) &\geq 0.
        \intertext{Since on $\tilde{X} \setminus E$, $\mu^{*}\theta + dd^{c}\psi = \tilde{\theta}+ dd^{c}g $ (see Equation~\eqref{eq: primary modification equation}) we get}
        \tilde{\pi}^{*}\tilde{\theta} + dd^{c}((v_{\text{Re}(z)} - \psi)\circ \mu(x) + g(x)) &\geq 0.
    \end{align*}
    Thus we see that the function $(\tilde{X}\setminus E) \times S \ni (x,z) \mapsto (v_{\text{Re}(z)} - \psi)\circ \mu(x) +g(x)$ is $\tilde{\pi}^{*}\tilde{\theta}$-psh function. Since the function is also bounded from above it extends to all of $\tilde{X}\times S$. Thus $(0,1) \ni t \mapsto (v_{t} - \psi)\circ \mu + g \in \PSH(\tilde{X},\tilde{\theta})$  is a subgeodesic.  

    Now we see the other direction. Let $(0,1) \ni t \mapsto \tilde{v}_{t} \in \PSH(\tilde{X},\tilde{\theta})$ is a subgeodesic.  This means $\tilde{\pi}^{*}\tilde{\theta} + dd^{c}\tilde{v}_{\text{Re}(z)}(x) \geq 0$. We saw earlier that for each $\tilde{v}_{t}$ there exists a unique $v_{t} \in \PSH(X,\theta,\psi)$ such that $(v_{t} - \psi)\circ \mu + g = \tilde{v}_{t}$. We need to show that $(0,1) \ni t \mapsto v_{t}$ is a subgeodesic. 

    To see this notice
    \begin{align*}
        \tilde{\pi}^{*}\tilde{\theta} + dd^{c}\tilde{v}_{\text{Re}(z)}(x) &\geq 0. 
        \intertext{Since $\tilde{\theta} = \mu^{*}\theta - [[E]] + dd^{c}\psi\circ \mu - dd^{c}g$ from Equation~\eqref{eq: primary modification equation}, above equation implies}
        \tilde{\pi}^{*}\mu^{*}\theta + dd^{c}(\tilde{v}_{\text{Re}(z)} + \psi\circ \mu - g)(x) &\geq 0. 
        \intertext{Now use that $v_{t} \circ \mu = \psi\circ \mu + \tilde{v}_{t} - g$ and commutation of the diagram, to see}
        (\mu \times \text id)^{*}\pi^{*}\theta + dd^{c}v_{\text{Re}(z)} \circ \mu (x) &\geq 0.
        \intertext{Now pushforward by $(\mu\times\text{id})_{*}$ to $X\times S$ to see}
        \pi^{*}\theta + dd^{c}v_{\text{Re}(z)}(x) &\geq 0.
    \end{align*}
    Hence $X \times S \ni (x,z) \mapsto v_{\text{Re}(z)}(x)$ is a $\pi^{*}\theta$-psh function. Thus $(0,1)\ni t \mapsto v_{t}$ is a subgeodesic. 

    Since subgeodesics correspond to subgeodesics under the correspondence $\PSH(X,\theta,\psi) \leftrightarrow \PSH(\tilde{X},\tilde{\theta})$, and geodesics are just supremum over subgeodesics, we get that the geodesics correspond to geodesics as well. In particular, if $u_{0}, u_{1} \in \PSH(X,\theta,\psi)$ and $u_{t} \in \PSH(X,\theta,\psi)$ is a geodesic joining $u_{0}$ and $u_{1}$, then $\tilde{u}_{t} = (u-\psi)\circ \mu + g \in \PSH(\tilde{X},\tilde{\theta})$ is the geodesic joining $\tilde{u}_{0}$ and $\tilde{u}_{1}$.
\end{proof}

Our next theorem allows us to extend Theorem~\ref{thm: linearity of MongeAmpere along geodesic in big case} in the case of prescribed singularity setting. 

\begin{thm}\label{thm: Monge ampere is linear in analytic singularity case} 
    If $u_{0}, u_{1} \in \PSH(X,\theta,\psi)$ have the same singularity type as $\psi$, and if $u_{t}$ is the weak geodesic joining $u_{0}$ and $u_{1}$, then 
    \[
    I(u_{t}) = (1-t)I(u_{0}) + tI(u_{1}).
    \]
\end{thm}

\begin{proof}
    First, we will show that the correspondence between $\PSH(X,\theta,\psi)$ and $\PSH(\tilde{X},\tilde{\theta})$ preserves the Monge-Am\'ere energy up to a constant. Second, we use Theorem~\ref{thm: linearity of MongeAmpere along geodesic in big case} to obtain that the Monge-Amp\`ere energy is linear along $u_{t}$.

    Take $u \in \PSH(X,\te,\psi)$ with the same singularity type as $\psi$ and let $\tilde{u} := (u-\psi)\circ \mu + g \in \PSH(\tilde{X},\tilde{\theta})$. Then Theorem~\ref{thm: pushforward of non-pluripolar measure}, tells us that
    \begin{equation}\label{eq: measures in monge ampere energy}
    \mu_{*} (\tilde{\theta}_{\tilde{u}}^{j} \wedge \tilde{\theta}^{n-j}_{g}) = \theta^{j}_{u} \wedge \theta^{n-j}_{\psi}.     
    \end{equation}
    The Monge-Amp\`ere energy of $u$ is given by 
    \[
    I(u) = \frac{1}{n+1} \sum_{j=0}^{n}\int_{X} (u-\psi) \theta^{j}_{u} \wedge \theta^{n-j}_{\psi}.
    \]
    From Equation~\eqref{eq: measures in monge ampere energy} we get
    \[
    I(u) = \frac{1}{n+1} \sum_{j=0}^{n} \int_{\tilde{X}} (u-\psi) \mu_{*}(\tilde{\theta}_{\tilde{u}}^{j} \wedge \tilde{\theta}^{n-j}_{g})= \int_{\tilde{X}} (u-\psi)\circ \mu \,\tilde{\theta}^{j}_{\tilde{u}} \wedge \tilde{\theta}_{g}^{n-j}.
    \]
    Thus we have, 
    \[
    I(u) = \frac{1}{n+1}\sum_{j=0}^{n} \int_{\tilde{X}} (\tilde{u}-g) \,\tilde{\theta}_{\tilde{u}}^{j}\wedge \tilde{\theta}_{g}^{n-j}  =I(\tilde{u}) - I(g).
    \]

    By Theorem~\ref{thm: Mabuchi geodesics are preserved} we know that $\tilde{u}_{t}$ is a geodesic joining $\tilde{u}_{0}$ and $\tilde{u}_{1}$ and by Corollary~\ref{cor: minimal singularity to minimal singularity}, $\tilde{u}_{0}$ and $\tilde{u}_{1}$ have minimal singularity in $\PSH(\tilde{X},\tilde{\theta})$. Thus we can use Theorem~\ref{thm: linearity of MongeAmpere along geodesic in big case}, to get that $I(\tilde{u}_{t}) = (1-t)I(\tilde{u}_{0}) + t I(\tilde{u}_{0})$. From the calculation above, we have 
    \[
    I(u_{t}) = I(\tilde{u}_{t}) - I(g) = (1-t)I(\tilde{u}_{0}) + tI(\tilde{u}_{1}) - I(g) = (1-t) I(u_{0}) + tI(u_{0})
    \]
    as desired.
\end{proof}

\subsection{Metric space structure on $\E^{p}(X,\theta,\psi)$}\label{subsec: metric space structure on analytic singularity}

In this section, we will import the metric space structure on $\E^{p}(X,\theta,\psi)$ from the metric space structure in $\E^{p}(\tilde{X},\tilde{\theta})$ when $\psi$ is a model singularity with analytic singularity type. 

We can define the distance between $u_{0}, u_{1} \in \E^{p}(X,\theta,\psi)$ as follows. Let $\tilde{u}_{0} = (u_{0}- \psi)\circ \mu + g$ and $\tilde{u}_{1} = (u_{1} - \psi)\circ \mu + g$ be the corresponding potentials in $\E^{p}(\tilde{X},\tilde{\theta})$. Corollary~\ref{cor: bijection preserves mass and energy} tells us that $\tilde{u}_{0}$, $\tilde{u}_{1} \in \E^{p}(\tilde{X},\tilde{\theta})$. So we can define 
\begin{equation}\label{eq: definition of dp in analytic singualarity setting}
d_{p}(u_{0},u_{1}) := d_{p}(\tilde{u}_{0}, \tilde{u}_{1}).    
\end{equation}

Theorem~\ref{thm: metric in analytic singularity type doesn't depend on resulution} below shows that the metric as defined above, does not depend on the choice of the resolution of the singularities of $\psi$.

\begin{thm}\label{thm: Complete metric space structure on analytic singularity space}
The map $d_{p}$ as defined by Equation~\eqref{eq: definition of dp in analytic singualarity setting} makes $\E^{p}(X,\theta,\psi)$ a complete geodesic metric space.
\end{thm}

\begin{proof}
    $(\E^{p}(X,\theta,\psi), d_{p})$ is a complete metric space because $(\E^{p}(\tilde{X},\tilde{\theta}), d_{p})$ is a complete metric space. Moreover, if $u_{t} \in \E^{p}(X,\theta,\psi)$ is the weak geodesic joining $u_{0}, u_{1} \in \E^{p}(X,\theta,\psi)$, we claim $u_{t}$ is also the metric geodeisc. This means that for $0 \leq t \leq s \leq 1$, we have $d_{p}(u_{t},u_{s}) = |t-s|d_{p}(u_{0},u_{1})$. 

    We know $\tilde{u}_{0}, \tilde{u}_{1} \in \E^{p}(\tilde{X},\tilde{\theta})$. In the proof of \cite[Theorem 3.17]{dinezza2018lp}, authors show that the weak geodesic $\tilde{u}_{t}$ joining $\tilde{u}_{0}$ and $\tilde{u}_{1}$ satisfies $d_{p}(\tilde{u}_{t}, \tilde{u}_{s}) = |t-s|d_{p}(\tilde{u}_{0},\tilde{u}_{1})$. Thus by the definition of $d_{p}$ on $\E^{p}(X,\theta,\psi)$, we obtain that $d_{p}(u_{t},u_{s}) = |t-s|d_{p}(u_{0},u_{1})$. Hence $(\E^{p}(X,\theta,\psi), d_{p})$ is a complete geodesic metric space, with the weak geodesics being the metric geodesics as well.
\end{proof}

Now we prove useful some properties of the metric $(\E^{p}(X,\theta,\psi),d_{p})$. 

\begin{lem}[Pythagorean formula] \label{lem: Pythagorean formula in analytic singularity type}
    For $u_{0}, u_{1} \in \E^{p}(X,\te,\psi)$, we have 
    \[
    d_{p}^{p}(u_{0},u_{1}) = d_{p}^{p}(u_{0}, P_{\te}(u_{0}, u_{1})) + d_{p}^{p}(u_{1}, P_{\te}(u_{0}, u_{1})).
    \]
\end{lem}

\begin{proof}
    The proof follows from Theorem~\ref{thm: Pythagorean identity in big and nef}, the Pythagorean identity for $d_{p}$ in the big and nef case, and the fact that $\widetilde{P_{\theta}(u_{0},u_{1})} = P_{\tte}(\tilde{u}_{0}, \tilde{u}_{1})$. This fact holds because the bijection $u \leftrightarrow \tilde{u}:= (u-\psi)\circ \mu + g$ is order-preserving. 
\end{proof}

The following Lemma says that the $d_{p}$ distance is controlled by the $I_{p}$ ``distance''. Given $u,v \in \E^{p}(X,\te, \psi)$, we define 
\[
I_{p}(u,v) = \int_{X} |u-v|^{p}(\te^{n}_{u} + \te^{n}_{v}).
\]
We have 
\begin{lem}\label{lem: Ip compares with dp in analytic singularity type as well}
    There is a constant $C > 1$, that depends only on $n$, such that for any $u_{0}, u_{1} \in \E^{p}(X,\te,\psi)$, 
    \[
    \frac{1}{C} I_{p}(u_{0}, u_{1}) \leq d_{p}^{p}(u_{0}, u_{1}) \leq CI_{p}(u_{0},u_{1}).
    \]
\end{lem}
\begin{proof}
    The proof follows from Theorem~\ref{thm: control by p-energy}, and the fact that $I_{p}(u_{0}, u_{1}) = I_{p}(\tilde{u}_{0}, \tilde{u}_{1})$. We observe
    \[
    I_{p}(\tilde{u}_{0}, \tilde{u}_{1}) = \int_{\tX} |\tilde{u}_{0} - \tilde{u}_{1}|^{p}(\tte^{n}_{\tilde{u}_{0}} + \tte^{n}_{\tilde{u}_{1}}) = \int_{\tX}|(\tilde{u}_{0} - \tilde{u}_{1})\circ\mu|(\tte^{n}_{\tilde{u}_{0}} + \tte^{n}_{\tilde{u}_{1}}).
    \]
    Pushing forward to $X$ by $\mu$ and using the fact that $\mu_{*}\tte^{n}_{\tilde{u}} = \te^{n}_{u}$, we get 
    \[
    I_{p}(\tilde{u}_{0}, \tilde{u}_{1}) = \int_{X}|u_{0} - u_{1}|^{p}(\te^{n}_{u_{0}}  +  \te^{n}_{u_{1}}) = I_{p}(u_{0}, u_{1}).
    \]
    From Theorem~\ref{thm: control by p-energy}, we know that there exists $C> 1$ such that 
    \[
    \frac{1}{C}I_{p}(\tilde{u}_{0}, \tilde{u}_{1}) \leq d_{p}^{p}(\tilde{u}_{0}, \tilde{u}_{1}) \leq C I_{p}(\tilde{u}_{0}, \tilde{u}_{1}).
    \]
    Therefore, from the above calculation we obtain that for the same $C$, we have 
    \[
    \frac{1}{C} I_{p}(u_{0}, u_{1}) \leq d_{p}^{p}(u_{0}, u_{1}) \leq C I_{p}(u_{0}, u_{1}).
    \]
    
\end{proof}

\begin{thm}\label{thm: distance between smooth objects in analytic singularity}
    Let $f_{0},f_{1} \in C^{1,\bar{1}}(X)$, $u_{0} = P_{\theta}[\psi](f_{0})$, $u_{1} = P_{\theta}[\psi](f_{1})$, and $u_{t}$ be the Mabuchi geodesic joining $u_{0}$ and $u_{1}$. Then 
    \[
    d_{p}^{p}(u_{0},u_{1}) = \int_{X} |\dot{u}_{t}|^{p}\theta^{n}_{u_{t}} \qquad \forall t \in [0,1].
    \]
\end{thm}

\begin{proof}
    From Lemma~\ref{lem: quasi-smoooth potentials have bounded entropy}, $u_{0}, u_{1} \in \text{Ent}(X,\theta,\psi)$. From Lemma~\ref{lem: entropy in desingularization}, the potentials $\tilde{u}_{0}$ and $\tilde{u}_{1}$ have finite entropy and from Corollary~\ref{cor: minimal singularity to minimal singularity} $\tilde{u}_{0}, \tilde{u}_{1}$ have minimal singularity. Thus using Theorem~\ref{thm: distance for finite entropy}, 
    \[
    d_{p}^{p}(\tilde{u}_{0}, \tilde{u}_{1}) = \int_{\tilde{X}}|\dot{\tilde{u}}_{t}|^{p}\tte^{n}_{\tilde{u}_{t}} \qquad \forall t \in [0,1]
    \]
    where $\tilde{u}_{t}$ is the weak geodesic joining $\tilde{u}_{0}$ and $\tilde{u}_{1}$. We emphasize that Theorem~\ref{thm: distance for finite entropy}, which is about geodesic distance for potentials with finite entropy, plays a crucial role here. In our procedure for importing geometry from the big and nef setting to the analytic singularity setting, we lose the property that $\tilde{u}$ is of the form $P_{\tte}(\tilde{f})$ for some $\tilde{f} \in C^{1,\bar{1}}(\tilde{X})$ where $u = P_{\te}[\psi](f)$ for some $f \in C^{1,\bar{1}}(X)$. 
    
    Since $\tilde{u}_{t} = (u_{t} - \psi)\circ \mu + g$ where $u_{t}$ is the weak geodesic joining $u_{0}$ and $u_{1}$, we have $\dot{\tilde{u}}_{t}  = \dot{u}_{t} \circ \mu$. We also have $\mu_{*} \tilde\theta^{n}_{\tilde{u}_{t}} = \theta^{n}_{u_{t}}$. Combining these we get 
    \[
    d_{p}^{p}(u_{0},u_{1}) := d^{p}_{p}(\tilde{u}_{0},\tilde{u}_{1}) = \int_{\tilde{X}} |\dot{\tilde{u}}_{t}|^{p}\theta^{n}_{\tilde{u}_{t}} = \int_{\tilde{X}} |\dot{u}_{t}\circ \mu|^{p}\theta^{n}_{\tilde{u}_{t}} = \int_{X}|\dot{u}_{t}|^{p}\theta^{n}_{u_{t}}
    \]
    for all $t \in [0,1]$. 
\end{proof}

In the special setting of $p = 1$, we have 
\begin{lem}\label{lem: E1 in analytic singularity case}
    Let  $f_{0}, f_{1} \in C^{1,\bar{1}}(X)$ satisfy $f_{0} \leq f_{1}$, and $u_{0} = P_{\theta}[\psi](f_{0})$ and $u_{1} = P_{\theta}[\psi](f_{1})$. If $u_{t}$ is the weak geodesic joining $u_{0}$ and $u_{1}$, then
    \[
    I(u_{1}) - I(u_{0}) = \int_{X} \dot{u}_{t} \theta^{n}_{u_{t}} \qquad \text{ for all } t \in [0,1].
    \]
\end{lem}

\begin{proof}
    From the proof of Theorem~\ref{thm: distance between smooth objects in analytic singularity}, we know that $\dot{u}_{t}\circ \mu = \dot{\tilde{u}}_{t}$. From the proof of Theorem~\ref{thm: Monge ampere is linear in analytic singularity case}, we know that $I(\tilde{u}_{0}) - I(g) = I(u_{0})$ and $I(\tilde{u}_{1})- I(g) = I(u_{1})$. Moreover, $\mu_{*}\tilde{\theta}_{\tilde{u}_{t}}^{n} = \theta^{n}_{u_{t}}$.  Combining these facts with Lemma~\ref{lem: comparable potential in E1 and big and nef}, we get 
    \[
    I(u_{1}) - I(u_{0}) = I(\tilde{u}_{0}) - I(\tilde{u}_{1}) = \int_{\tilde{X}} \dot{\tilde{u}}_{t}\tilde{\theta}^{n}_{\tilde{u}_{t}} = \int_{\tilde{X}} \dot{u}_{t} \circ \mu \tilde{\theta}^{n}_{\tilde{u}_{t}} = \int_{X} \dot{u}_{t} \theta^{n}_{u_{t}}
    \]
    as desired.
\end{proof}

Now we will show that the metric $d_{p}$ as defined by Equation~\eqref{eq: definition of dp in analytic singualarity setting} does not depend on the choice of resolution. 
\begin{lem}\label{lem: decreasing sequences converge}
    Let $u_{0}^{k}, u_{1}^{k}, u_{0}, u_{1} \in \E^{p}(X,\te,\psi)$ satisfy $u_{0}^{k} \searrow u_{0}$ and $u_{1}^{k} \searrow u_{1}$. Then $d_{p}(u_{0}^{k}, u_{1}^{k}) \to d_{p}(u_{0}, u_{1})$ as $k \to \infty$.
\end{lem}
\begin{proof}
    $u_{0}^{k} \searrow u_{0}$ implies that $\tilde{u}_{0}^{k} = (u_{0}^{k} - \psi)\circ \mu + g \searrow (u_{0} -\psi)\circ \mu + g = \tilde{u}_{0}$. Similarly, $\tilde{u}_{1}^{k} \searrow \tilde{u}_{1}$. We claim that in the space $(\E^{p}(\tX, \tte), d_{p})$ we have $d_{p}(\tilde{u}_{0}^{k}, \tilde{u}_{1}^{k}) \to d_{p}(\tilde{u}_{0}, \tilde{u}_{1})$. To see this, we observe by triangle inequality we have 
    \[
    d_{p}(\tilde{u}_{0}, \tilde{u}_{1}) -   d_{p}(\tilde{u}_{0}^{k} , \tilde{u}_{1}^{k}) \leq d_{p}(\tilde{u}_{0}, \tilde{u}_{0}^{k})+d_{p}(\tilde{u}_{1}^{k}, \tilde{u}_{1}).
    \]
    As the other side is obtained similarly, we have
    \[
    |d_{p}(\tilde{u}_{0}, \tilde{u}_{1}) -   d_{p}(\tilde{u}_{0}^{k} , \tilde{u}_{1}^{k})| \leq d_{p}(\tilde{u}_{0}, \tilde{u}_{0}^{k})+d_{p}(\tilde{u}_{1}^{k}, \tilde{u}_{1}).
    \]
    From \cite[Proposition 3.12]{dinezza2018lp}, we have $d_{p}(\tilde{u}_{0}^{k},\tilde{u}_{0}) \to 0$ and $d_{p}(\tilde{u}_{1}^{k}, \tilde{u}_{1}) \to 0$ as $k \to \infty$. Thus $d_{p}(\tilde{u}_{0}^{k}, \tilde{u}_{1}^{k}) \to d_{p}(\tilde{u}_{0}, \tilde{u}_{1})$ as $k \to \infty$.

    Now from Equation~\eqref{eq: definition of dp in analytic singualarity setting}, we obtain that $d_{p}(u_{0}^{k}, u_{1}^{k}) \to d_{p}(u_{0}, u_{1})$ as well. 
\end{proof}

\begin{thm}\label{thm: metric in analytic singularity type doesn't depend on resulution}
    The metric $d_{p}$ as defined by Equation~\eqref{eq: definition of dp in analytic singualarity setting} on $\E^{p}(X,\te,\psi)$ does not depend on the choice of resolution. 
\end{thm}
\begin{proof}
    If $u_{0}, u_{1} \in \E^{p}(X,\te,\psi)$ are of the form $u_{0} = P_{\te}[\psi](f_{0})$ and $u_{1} = P_{\te}[\psi](f_{1})$ for some functions $f_{0}, f_{1} \in C^{1,\bar{1}}(X)$, then from Theorem~\ref{thm: distance between smooth objects in analytic singularity}, we know that $d_{p}(u_{0}, u_{1})$ does not depend on the choice of resolution, it is determined by the weak geodesic joining them. 

    More generally, given any $u_{0}, u_{1} \in \E^{p}(X,\te,\psi)$, from \cite{BlockiKolodziejregularization}, we can find smooth functions $f_{0}^{k}, f_{1}^{k} \in C^{\infty}(X)$ such that $f_{0}^{k} \searrow u_{0}$ and $f_{1}^{k} \searrow u_{1}$. Then $P_{\te}[\psi](f_{0}^{k}) \searrow u_{0}$ and $P_{\te}[\psi](f_{1}^{k}) \searrow u_{1}$ as well. From Lemma~\ref{lem: decreasing sequences converge}, $d_{p}(P_{\te}[\psi](f_{0}^{k}), P_{\te}[\psi](f_{1}^{k})) \to d_{p}(u_{0}, u_{1})$. From the discussion in the previous paragraph, $d_{p}(P_{\te}[\psi](f_{0}^{k}), P_{\te}[\psi](f_{1}^{k}))$ does not depend on the choice of resolution, thus the distance $d_{p}(u_{0}, u_{1})$ can be determined without the choice of resolution as well. 
\end{proof}

\section{Metric on $\E^{p}(X,\theta)$} \label{sec: metric in the big case}

In this section, we define a metric $d_{p}$ on $\E^{p}(X,\theta)$ that makes $(\E^{p}(X,\theta), d_{p})$ a complete geodesic metric space. The idea is to approximate the potentials in $\E^{p}(X,\theta)$ from the potentials in $\E^{p}(X,\theta, \psi)$ and use the metric structure on $\E^{p}(X,\theta,\psi)$ as described in Section~\ref{subsec: metric space structure on analytic singularity}.

 In the big class represented by $\theta$, we can find a K\"ahler potential $\varphi \in \PSH(X,\theta)$ with analytic singularities. We can also assume, by subtracting a constant if needed, that $\varphi \leq V_{\theta}$. Define $\varphi_{j}  = \frac{1}{j} \varphi + \frac{j-1}{j} V_{\theta}$, so $\varphi_{j} \leq V_{\theta}$ and $\varphi_{j} \nearrow V_{\theta}$ outside a pluripolar set. Unfortunately, $\varphi_{j}$ does not have analytic singularities. Since $\varphi$ is a K\"ahler potential, $\varphi_{j}$ is also a K\"ahler potential. By Demailly's regularization, we can find $\phi_{j} \geq \varphi_{j}$ such that $\phi_{j}$ is a K\"ahler potential with analytic singularities. But the sequence $\phi_{j}$ is not monotone.  

Now, consider $P_{\theta}[\phi_{j}] := \sup \{ u \in \PSH(X,\theta) : u \preceq \phi_{j}, u \leq V_{\theta}\}$. As $\varphi_{j} \leq \phi_{j}$ and $\varphi_{j} \leq V_{\theta}$, we have $\varphi_{j} \leq P_{\theta}[\phi_{j}] \leq V_{\theta}$. As $\varphi_{j} \nearrow V_{\theta}$ outside a pluripolar set, we find that $P_{\theta}[\phi_{j}] \to V_{\theta}$ pointwise outside a pluripolar set. Also, since $\phi_{j}$ has analytic singularities, and $[P_{\theta}[\phi_{j}]] = [\phi_{j}]$, which follows from \cite[Proposition 2.20]{darvasxiaclosureoftestconfigurations}, we get $P_{\theta}[\phi_{j}]$ has analytic singularities as well. Now consider 
\[
\psi_{j} = \max\{ P_{\theta}[\phi_{1}], \dots, P_{\theta}[\phi_{j}]\}.
\]
Then by \cite[Lemma 2.4]{darvas2023transcendental}, $\psi_{j}$ has analytic singularities, and $\psi_{j} \nearrow V_{\theta}$ except on a pluripolar set. We fix such a sequence $\psi_{j} \nearrow V_{\te}$ for the rest of the paper. 

Following \cite{darvasgeoemtryoffiniteenergy} and \cite{dinezza2018lp}, we will first define the $d_{p}$ metric on the space of ``smooth'' potentials, and then extend it to the whole space $\E^{p}(X,\theta)$. In general, $\E^{p}(X,\theta)$ has no smooth potentials, but the space 
\[
\mathcal{H}_{\theta} = \{ u \in \PSH(X,\theta) \, | \, u = P_{\theta}(f) \text{ for some } f \in C^{1,\bar{1}}(X)\}
\]
will act as the space of ``smooth'' potentials for us.

\begin{lem}\label{lem: Hthetea closed under rooftop envelope}
    If $u_{0}, u_{1} \in \mathcal{H}_{\theta}$, then $P_{\theta}(u_{0}, u_{1}):= P_{\theta}(\min\{u_{0}, u_{1}\}) \in \mathcal{H}_{\theta}$. 
\end{lem}

\begin{proof}
    Let $f_{0}, f_{1} \in C^{1,\bar{1}}(X)$ and $u_{0} = P_{\theta}(f_{0})$ and $u_{1} = P_{\theta}(f_{1})$. Let $C$ be such that $\theta \leq C\om$. Then from \cite[Theorem 2.5]{DarvasRubinsteinrooftopobstacle} $P_{C\om}(f_{0}, f_{1})$ is a $C^{1,\bar{1}}$ function. We claim that $P_{\theta}(f_{0}, f_{1}) = P_{\theta}(P_{C\om}(f_{0},f_{1}))$. 

    Since $P_{C\om}(f_{0}, f_{1}) \leq \min\{f_{0},f_{1}\}$, we have $P_{\theta}(P_{C\om}(f_{0},f_{1})) \leq P_{\theta}(f_{0},f_{1})$. For the other direction, note that $0 \leq \theta + dd^{c}P_{\theta}(f_{0},f_{1}) \leq C\om + dd^{c}P_{\theta}(f_{0},f_{1})$. Thus $P_{\theta}(f_{0},f_{1})$ is a $C\om$-psh as well. As $P_{\theta}(f_{0},f_{1}) \leq \min\{f_{0},f_{1}\}$, we have $P_{\theta}(f_{0},f_{1}) \leq P_{C\om}(f_{0},f_{1})$. Thus $P_{\theta}(f_{0},f_{1}) \leq P_{\theta}(P_{C\om}(f_{0},f_{1}))$. 

    \sloppy Also, $P_{\theta}(u_{0},u_{1}) = P_{\theta}(f_{0},f_{1})$ by a similar argument. So $P_{\theta}(u_{0}, u_{1}) = P_{\theta}(P_{C\om}(f_{0},f_{1}))$ where $P_{C\om}(f_{0},f_{1}) \in C^{1,\bar{1}}(X)$. 
\end{proof}

\subsection{Metric on $\mathcal{H}_{\theta}$}\label{subsec: metric on smooth potentials}

In this subsection, we will construct the metric $d_{p}$ on $\mathcal{H}_{\theta}$. The idea is to approximate for $f_{0}, f_{1} \in C^{1,\bar{1}}(X)$, potentials $u_{0} = P_{\theta}(f_{0}), u_{1} = P_{\theta}(f_{1})\in \mathcal{H}_{\theta}$  via $u_{0}^{k}:= P_{\theta}[\psi_{k}](f_{0}), u_{1}^{k}:= P_{\theta}[\psi_{k}](f_{1}) \in \PSH(X,\theta,\psi_{k})$ for the increasing sequence of potentials with analytic singularity type $\psi_{k} \nearrow V_{\theta}$ as fixed in the beginning of the section. We fix the notation for $f_{0}, f_{1}, u_{0}, u_{1}, u_{0}^{k}, u_{1}^{k}$ for the rest of the section. Since $u_{0}^{k}, u_{1}^{k}$ have the same singularity type as $\psi_{k}$, $u_{0}^{k}$, we get that $u_{1}^{k} \in \E^{p}(X,\te,\psi_{k})$. We wish to define 
\begin{equation}\label{eq: definition on Htheta}
    d_{p}(u_{0}, u_{1}) := \lim_{k\to\infty}d_{p}(u_{0}^{k},u_{1}^{k}).
\end{equation}
Here $d_{p}(u_{0}^{k}, u_{1}^{k})$ is the distance defined in Section~\ref{subsec: metric space structure on analytic singularity} on $\E^{p}(X,\te,\psi_{k})$. 

In this subsection, we will first show that indeed $u_{0}^{k}$ and $u_{1}^{k}$ increase to $u_{0}$ and $u_{1}$ respectively. Moreover, the limit in Equation~\eqref{eq: definition on Htheta} exists, is independent of the choice of the approximating sequence $\psi_{k}$, and defines a metric on $\cH_{\te}$.

The following lemma shows that envelopes with respect to $V_{\theta}$ can be approximated by envelopes with respect to $\psi_{k}$.

\begin{lem}\label{Envelopes for increasing singularity type}
    Let $\psi, \psi_{k} \in \PSH(X,\theta)$ be an increasing sequence and $\psi = \left( \lim_{k \to \infty}\psi_{k}\right)^{*}$ (here $u^{*}(x) = \limsup_{y \to x}u(y)$ is the upper semicontinuous regularization). Then for any continuous $f: X \to \R$, $P_{\theta}[\psi_{k}](f)$ is an increasing sequence and $(\lim_{k \to \infty} P_{\theta}[\psi_{k}](f))^{*} = P_{\theta}[\psi](f)$. 
\end{lem}

\begin{proof}
    \sloppy If $k > l$, then $P_{\te}(\psi_{l} + C, f) \leq P_{\te}(\psi_{k} + C,f)$. Taking the limit $C \to \infty$, we get that $P_{\te}
    [\psi_{l}](f) \leq P_{\te}[\psi_{k}](f)$. Therefore, $P_{\theta}[\psi_{k}](f)$ is an increasing sequence of $\theta$-psh 
    functions. Thus, $(\lim_{k\to\infty}P_{\theta}[\psi_{k}](f))^{*}$ is a $\theta$-psh function. Since $P_{\theta}[\psi_{k}](f) 
    \leq P_{\theta}[\psi](f)$ for all $k$ and is upper semicontinuous, $(\lim_{k\to\infty} P_{\theta} [\psi_{k}](f))^{*} $ $ \leq 
    P_{\theta}[\psi](f)$. 

    To show the other direction we use Lemma~\ref{The surprising Lemma from DDLs}. Since $\psi_{k} \nearrow \psi$, \cite[Theorem 2.3]{Darvasmonotonicity} implies that $\int_{X} \te^{n}_{\psi_{k}} \nearrow \int_{X} \te^{n}_{\psi}$. From Lemma~\ref{The surprising Lemma from DDLs}, we can find $\al_{k} \to 0$ such that 
    \[
    v_{k} = P_{\te}\left( \frac{1}{\al_{k}} \psi_{k} + \left( 1-\frac{1}{\al_{k}}\right)\psi\right) \in \PSH(X,\te).
    \]
    This implies 
    \[
    \al_{k} v_{k} + (1-\al_{k})\psi \leq \psi_{k}.
    \]
    Using Lemma~\ref{lem: concavity of envelope operator}, we get
    \begin{equation}\label{eq: equation from concavity}
        \al_{k}P_{\te}[v_{k}](f) + (1-\al_{k})P_{\te}[\psi](f) \leq P_{\te}[\al_{k}v_{k} + (1-\al_{k})\psi](f) \leq P_{\te}[\psi_{k}](f).        
    \end{equation}
    Now $\sup_{X}P_{\te}[v_{k}](f)$ are bounded. As $P_{\te}[v_{k}](f) \leq f$, so they are bounded from above. Also if $f \geq C$ for some $C$, then $\sup_{X}P_{\te}[v_{k}](f) \geq C$ as $v_{k} + C_{k}$ such that $\sup_{X}(v_{k} + C_{k}) = C$ is a valid candidate for the definition of $P_{\te}[v_{k}](f)$. Therefore, $\sup_{X} P_{\te}[v_{k}](f)$ is bounded. Hence after taking the weak $L^1$-limit in Equation~\eqref{eq: equation from concavity} we get 
    \[
    P_{\te}[\psi](f) \leq \lim_{k\to\infty}P_{\te}[\psi_{k}](f)
    \]
    almost everywhere. 
    Thus we get $(\lim_{k\to \infty}P_{\te}[\psi_{k}](f))^{*} = P_{\te}[\psi](f)$. 
\end{proof}

In the following, we use Lemma~\ref{lem: weak convergence of measures}  to prove that in the approximation scheme discussed above, Monge-Amp\`ere energy and the $I_{p}$-``distance'' converge. 

\begin{lem}\label{thm: convergence of Monge-Ampere energy}
     Let $f \in C^{1,\bar{1}}$ and $\psi_{k}$ are model potentials of analytic singularity type such that $\psi_{k} \nearrow V_{\theta}$ outside a pluripolar set. Let $u_{k} = P_{\theta}[\psi_{k}](f)$ and $u = P_{\theta}(f)$, then the $\psi_{k}$-relative Monge-Amp\`ere energy of $u_{k}$ converge to the Monge-Amp\`ere energy of $u$. 
\end{lem}

\begin{proof}
    Let $C$ be such that $\sup_{X}|f| \leq C$, then $|u_{k} - \psi_{k}| \leq C$. Thus $0 \leq u_{k} - \psi_{k} + C \leq 2C$. From Lemma~\ref{Envelopes for increasing singularity type} we know that $u_{k}\nearrow u$. Thus $u_{k} - \psi_{k}+C$ are uniformly bounded quasi-continuous functions that converge in capacity to $u - V_{\theta}+C$. Moreover, as $u$ and $V_{\theta}$ have minimal singularity, we know 
    \[
    \int_{X} \theta^{j}_{u}\wedge \theta^{n-j}_{V_{\te}} \geq \limsup_{k\to \infty}\int_{X} \te^{j}_{u_{k}} \wedge \te^{n-j}_{\psi_{k}}.
    \]
    Thus from Lemma~\ref{lem: weak convergence of measures}, we know that the measures 
    \[
    (u_{k} - \psi_{k} + C) \theta^{j}_{u_{k}}\wedge \theta^{n-j}_{\psi_{k}} \to (u - V_{\theta} + C) \theta^{j}_{u}\wedge \theta^{n-j}_{V_{\theta}}
    \]
    and 
    \[
    \te^{j}_{u_{k}} \wedge \te^{n-j}_{\psi_{k}} \to \te^{j}_{u} \wedge \te^{n-j}_{V_{\te}}
    \]
    weakly as $k \to \infty$. Thus the $\psi_{k}$-relative Monge-Amp\`ere energy 
    \[
    I(u_{k}) = \frac{1}{n+1} \sum_{j=0}^{n} \int_{X} (u_{k} - \psi_{k}) \theta^{j}_{u_{k}}\wedge \theta^{n-j}_{\psi_{k}} \to \frac{1}{n+1} \sum_{j=0}^{n} \int_{X} (u-V_{\theta}) \theta^{j}_{u} \wedge \theta^{n-j}_{V_{\theta}} = I(u)
    \]
    as $k \to \infty$. 
\end{proof}

\begin{lem}\label{convergence of p-energy}
Let $u_{0}, u_{1}, u_{0}^{k}, u_{1}^{k}$ be as in the beginning of Section~\ref{subsec: metric on smooth potentials}, then the $I_{p}$ functional
\[
I_{p}(u_{0}^{k}, u_{1}^{k}) = \int_{X} |u_{0}^{k} - u_{1}^{k}|^{p}(\te^{n}_{u_{0}^{k}} + \te^{n}_{u_{1}^{k}}) \to \int_{X} |u_{0} - u_{1}|^{p}(\te^{n}_{u_{0}} + \te^{n}_{u_{1}}) = I_{p}(u_{0}, u_{1})
\]
as $k \to \infty$.
\end{lem}

\begin{proof}
    We notice that 
    \[
    |u_{0}^{k} - u_{1}^{k}| \leq \sup_{X}|f_{0} - f_{1}|.
    \]
    This is true because if $C = \sup_{X}|f_{0} - f_{1}|$, then $f_{0} - C \leq f_{1}$, and therefore $P_{\theta}[\psi_{k}](f_{0}) - C$ is a candidate function for $P_{\theta}[\psi_{k}](f_{1})$, therefore, 
    \[
    P_{\theta}[\psi_{k}](f_{0}) - C \leq P_{\theta}[\psi_{k}](f_{1})
    \]
    and the other direction is shown similarly. Thus 
    \[
    |P_{\theta}[\psi_{k}](f_{0}) - P_{\theta}[\psi_{k}](f_{1})| \leq \sup_{X}|f_{0} - f_{1}|.
    \]
    
    Moreover, from Lemma~\ref{Envelopes for increasing singularity type} the functions the functions $u_{0}^{k} \nearrow u_{0}$ ad $u_{1}^{k} \nearrow u_{1}$ away from a pluripolar set, therefore $u_{0}^{k} \to u_{0}$ and $u_{1}^{k} \to u_{1}$ in capacity as $k\to \infty$. Moreover, $|u_{0}^{k} - u_{1}^{k}|^{p}$ are quasi-continuous and uniformly bounded, therefore from Lemma~\ref{lem: weak convergence of measures}, we get that 
    \[
    \lim_{k\to \infty}\int_{X} |u_{0}^{k} - u_{1}^{k}|^{p}(\te_{u_{0}^{k}}^{n} + \te_{u_{1}^{k}}^{n}) = \int_{X} |u_{0} - u_{1}|^{p}(\te_{u_{0}}^{n} + \te_{u_{1}}^{n}).
    \]
    \end{proof}

The next theorem proves that the limit in Equation~\eqref{eq: definition on Htheta} exists in the special setting when $u_{0}\leq u_{1}$. 

\begin{thm}\label{thm: definition is well-defined for u0 less than u1} If $f_{0}, f_{1},u_{0}, u_{1}, u_{0}^{k}, u_{1}^{k}$ are as in the beginning of Section~\ref{subsec: metric on smooth potentials} along with the assumption that $f_{0} \leq f_{1}$, then 
\[
\lim_{k\to \infty} d_{p}^{p}(u_{0}^{k}, u_{1}^{k}) = \int_{X} |\dot{u}_{0}|^{p} \theta^{n}_{u_{0}} = \int_{X} |\dot{u}_{1}|^{p}\theta^{n}_{u_{1}},
\]
where $u_{t}$ is the weak geodesic joining $u_{0}$ and $u_{1}$. Thus the limit in Equation~\eqref{eq: definition on Htheta} exists  and is independent of the approximating sequence $\psi_{k}$ if $f_{0} \leq f_{1}$. 
\end{thm}

\begin{proof}
Since $f_{0} \leq f_{1}$, we have $u_{0} \leq u_{1}$ and $u_{0}^{k} \leq u_{1}^{k}$. Let $u_{t}^{k}$ be the geodesic joining $u_{0}^{k}$ and $u_{1}^{k}$. Since $f_{0}, f_{1}$ are bounded, $u_{0}, u_{1}$ have the minimal singularity type and $u_{0}^{k}, u_{1}^{k}$ have the same singularity type as $\psi_{k}$. 

Now Lemma~\ref{lem: E1 in analytic singularity case} says that 
\[
I(u_{1}^{k}) - I(u_{0}^{k}) = \int_{X}\dot{u}_{0}^{k}\theta^{n}_{u_{0}^{k}}.
\]

From Theorem~\ref{thm: convergence of Monge-Ampere energy}, we know that $I(u_{0}^{k}) \to I(u_{0})$ and $I(u_{1}^{k}) \to I(u_{1})$ as $k \to \infty$. Combining with Lemma~\ref{lem: monge ampere energy and geodesic in big case}, we get that 
\[
\lim_{k\to \infty}\int_{X}\dot{u}_{0}^{k} \theta^{n}_{u_{0}^{k}} = \int_{X}\dot{u}_{0}\theta^{n}_{u_{0}}.
\]
 From Theorem~\ref{thm: measure on contact sets} we obtain that $\theta^{n}_{u_{0}^{k}} = \mathds{1}_{D_{k}}\theta^{n}_{f_{0}}$ where $D_{k} = \{ P_{\theta}[\psi_{k}](f_{0}) = f_{0}\}$, and $\theta^{n}_{u_{0}} = \mathds{1}_{D}\theta^{n}_{u_{0}}$ where $D = \{ P_{\theta}(f_{0}) = f_{0}\}$. So we can write that
\[
\lim_{k \to \infty} \int_{X}\mathds{1}_{D_{k}}\dot{u}^{k}_{0} \theta^{n}_{f_{0}} = \int_{X}\mathds{1}_{D}\dot{u}_{0}\theta^{n}_{f_{0}}.
\]

As $u_{0}^{k} \nearrow u_{0}$ and $u_{1}^{k} \nearrow u_{1}$, we find that the geodesics $u_{t}^{k}$ joining $u_{0}^{k}$ and $u_{1}^{k}$ are also increasing. This holds because if $k < l$, then the geodesic $u_{t}^{k}$ is a candidate subgeodesic joining $u_{0}^{l}$ and $u_{1}^{l}$. Similarly, all geodesics $u_{t}^{k}$ are candidate subgeodesics joining $u_{0}$ and $u_{1}$. Therefore $u_{t}^{k}$ are increasing in $k$ and $u_{t}^{k} \leq u_{t}$. 

Similarly, we can show that the contact sets $D_{k}$ are increasing. If $k < l$, and $x \in D_{k}$, then $P_{\te}[\psi_{k}](f_{0})(x) = f_{0}(x)$ and since $P_{\te}[\psi_k](f_{0}) \leq P_{\te}[\psi_{l}](f_{0}) \leq f_{0}$, we find that $x \in D_{l}$ as well, so $D_{k} \subset D_{l}$. Moreover since $P_{\te}[\psi_{k}](f_{0}) \leq P_{\te}(f_{0}) \leq f_{0}$, we have $D_{k} \subset D$ for all $D$.

 If $x \in D_{k}$ and $k < l$, then 
 \[
 \dot{u}^{k}_{0}(x) = \lim_{t\to 0} \frac{u^{k}_{t}(x) - u^{k}_{0}(x)}{t} = \lim_{t \to 0} \frac{u_{t}^{k}(x) - f_{0}(x)}{t} \leq \lim_{t\to 0}\frac{u^{l}_{t}(x) - f_{0}(x)}{t} = \dot{u}^{l}_{0}(x).
 \]
 Similarly, if $x \in D_{k}$, then 
 \[
  \dot{u}^{k}_{0}(x) = \lim_{t\to 0} \frac{u^{k}_{t}(x) - u^{k}_{0}(x)}{t} = \lim_{t \to 0} \frac{u_{t}^{k}(x) - f_{0}(x)}{t} \leq \lim_{t\to 0}\frac{u_{t}(x) - f_{0}(x)}{t} = \dot{u}_{0}(x).
 \]
 Also by assumption $u^{k}_{0} \leq u^{k}_{1}$, so $\dot{u}^{k}_{0}, \dot{u}_{0} \geq 0$. Therefore, $\mathds{1}_{D_{k}}\dot{u}^{k}_{0}$ is an increasing sequence such that for each $k$, $\mathds{1}_{D_{k}} \dot{u}_{0}^{k} \leq \mathds{1}_{D}\dot{u}_{0}$, and  
 \[
\lim_{k\to\infty} \int_{X} \mathds{1}_{D_{k}}\dot{u}^{k}_{0} \theta^{n}_{f_{0}} = \int_{X} \mathds{1}_{D}\dot{u}_{0} \theta^{n}_{f_{0}}.
 \]
 Therefore, $\mathds{1}_{D_{k}}\dot{u}^{k}_{0} \nearrow \mathds{1}_{D}\dot{u}_{0}$ pointwise $\theta^{n}_{f_{0}}$ almost everywhere.

Also $0 \leq \dot{u}^{k}_{0} \leq u^{k}_{1} - u^{k}_{0} \leq \sup_{X}|f_{0} -f_{1}|$. Thus we have a uniform bound on $\dot{u}^{k}_{0}$. Therefore by Lebesgue's Dominated Convergence Theorem, we have 
\begin{equation}\label{eq: convergence in monotone case}
    \int_{X}\mathds{1}_{D_{k}}(\dot{u}^{k}_{0})^{p} \theta^{n}_{f_{0}} \to \int_{X} \mathds{1}_{D}(\dot{u}_{0})^{p} \theta^{n}_{f_{0}}.
\end{equation}

Now, from Theorem~\ref{thm: distance between smooth objects in analytic singularity}, 
\[
d_{p}^{p}(u_{0}^{k}, u_{1}^{k}) = \int_{X}|\dot{u}^{k}_{0}|^{p}\theta^{n}_{u^{k}_{0}} = \int_{X} \mathds{1}_{D_{k}}(\dot{u}_{0}^{k})^{p}\theta^{n}_{f_{0}} \to \int_{X} \mathds{1}_{D}(\dot{u}_{0})^{p}\theta^{n}_{f_{0}} = \int_{X} |\dot{u}_{0}|^{p}\theta^{n}_{u_{0}}
\]
as $k \to \infty$. 

The same proof works for $t =1$ as well.
\end{proof}

We now follow the proof of \cite[Theorem 3.26]{Darvas2019GeometricPT} to get
\begin{thm}\label{pythagorean formula for smooth enough potentials}
    Let $f_{0}, f_{1} \in C^{1,\bar{1}}(X)$ and $u_{0} = P_{\theta}(f_{0})$ and $u_{1} = P_{\theta}(f_{1})$. Let $u_{t}$ be the geodesic joining $u_{0}$ and $u_{1}$. Also assume that $w_{t}$ is a geodesic joining $P_{\theta}(u_{0},u_{1})$ and $u_{0}$, and $v_{t}$ is a geodesic joining $P_{\theta}(u_{0},u_{1})$ and $u_{1}$. Then 
    \[
    \int_{X}|\dot{u}_{0}|^{p} \theta^{n}_{u_{0}} = \int_{X} |\dot{v}_{0}|^{p}\theta^{n}_{P_{\theta}(u_{0},u_{1})} + \int_{X}|\dot{w}_{0}|^{p}\theta^{n}_{P_{\theta}(u_{0},u_{1})}.
    \]
\end{thm}

\begin{proof}
    Just like in \cite[Theorem 3.26]{Darvas2019GeometricPT}, we will use Lemma~\ref{rooftop envelops set and derivatives of geodesics} and Corollary~\ref{measures of rooftop potentials are usually nice} repeatedly to settle the claim. 

    We will prove that 
    \begin{equation}\label{eq: positive slope part}
    \int_{\{\dot{u}_{0}>0\}} |\dot{u}_{0}|^{p}\theta^{n}_{u_{0}} = \int_{X} |\dot{v}_{0}|^{p}\theta^{n}_{P_{\theta}(u_{0},u_{1})}   
    \end{equation}
    and that 
    \begin{equation}\label{eq: negative slope part}
    \int_{\{\dot{u}_{0}<0\}} |\dot{u}_{0}|^{p}\theta^{n}_{u_{0}} = \int_{X} |\dot{w}_{0}|^{p}\theta^{n}_{P_{\theta}(u_{0},u_{1})}  
    \end{equation}

    \begin{align*}
        \int_{\{\dot{u}_{0}>0\}} |\dot{u}_{0}|^{p}\theta^{n}_{u_{0}} &= p \int_{0}^{\infty}\tau^{p-1}\theta^{n}_{u_{0}}(\{ \dot{u}_{0} \geq \tau\})d\tau \\
        &= p\int_{X}\tau^{p-1}\theta^{n}_{u_{0}}(\{P_{\te}(u_{0},u_{1}-\tau) = u_{0}\})d\tau.
    \end{align*}

    On the other hand,
    \begin{align}
        \int_{X}|\dot{v}_{0}|^{p}\theta^{n}_{P_{\theta}(u_{0},u_{1})} &= p\int_{0}^{\infty}\tau^{p-1}\theta^{n}_{P_{\theta}(u_{0},u_{1})}(\{\dot{v}_{0} \geq \tau\})d\tau \nonumber\\
        &= p\int_{0}^{\infty}\tau^{p-1}\theta^{n}_{P_{\theta}(u_{0},u_{1})}(\{P_{\theta}(P_{\theta}(u_{0},u_{1}),u_{1}-\tau) = P_{\theta}(u_{0},u_{1})\})d\tau \nonumber\\
        &= p\int_{0}^{\infty}\tau^{p-1}\theta^{n}_{u_{0}}(\{P_{\theta}(u_{0},u_{1}-\tau)=u_{0}\})d\tau. \label{eq: expression for partial geodesic}
    \end{align}

    For the last step, we used Corollary~\ref{measures of rooftop potentials are usually nice} and the fact that $P_{\theta}(P_{\theta}(u_{0},u_{1}),u_{1}-\tau) = P_{\theta}(u_{0},u_{1}-\tau)$, $\{ P_{\theta}(u_{0},u_{1}-\tau) = u_{0}\}  = \{P_{\theta}(u_{0},u_{1}-\tau) = P_{\theta}(u_{0},u_{1}) = u_{0}\}$, and that the set $\{P_{\theta}(u_{0},u_{1}-\tau) = P_{\theta}(u_{0},u_{1}) = u_{1}\} = \emptyset$. Thus proving Equation~\eqref{eq: positive slope part}.

    For Equation~\eqref{eq: negative slope part},  we observe from Corollary~\ref{measures of rooftop potentials are usually nice} that except for countably many $\tau$, we have 
    \[
    \text{Vol}(\theta) = \theta^{n}_{u_{0}}(\{P_{\theta}(u_{0},u_{1}+\tau)=u_{0}\}) + \theta^{n}_{u_{1}}(\{P_{\theta}(u_{0},u_{1}+\tau) = u_{1}+\tau\}).
    \]
    Now,
    \begin{align*}
        \int_{\{\dot{u}_{0} < 0\}} |\dot{u}_{0}|^{p}\theta^{n}_{u_{0}} &= p\int_{0}^{\infty} \tau^{p-1}\theta^{n}_{u_{0}}(\{-\dot{u}_{0} \geq \tau\})d\tau \\
        &= p\int_{0}^{\infty}\tau^{p-1}\theta^{n}_{u_{0}}(X\setminus\{-\dot{u}_{0} <\tau\})d\tau\\
        &= p\int_{0}^{\infty}\tau^{p-1}(\text{Vol}(\theta) - \theta^{n}_{u_{0}}(\{P_{\te}(u_{0},u_{1}+\tau) = u_{0}\})d\tau\\
        &= p\int_{0}^{\infty}\tau^{p-1}\theta^{n}_{u_{1}}(\{P_{\theta}(u_{0},u_{1}+\tau) = u_{1}+\tau\})d\tau\\
        &= p\int_{0}^{\infty}\tau^{p-1}\theta^{n}_{u_{1}}(\{P_{\theta}(u_{0}-\tau,u_{1}) = u_{1}\})d\tau.
    \end{align*}

    This is the same expression as Equation~\eqref{eq: expression for partial geodesic} with the roles of $u_{0}$ and $u_{1}$ reversed. Therefore, 
    \[
    \int_{\{\dot{u}_{0} < 0\}} |\dot{u}_{0}|^{p}\theta^{n}_{u_{0}} = \int_{X} |\dot{w}_{0}|^{p}\theta^{n}_{P_{\theta}(u_{0},u_{1})}  
    \]
    proving Equation~\eqref{eq: negative slope part}.
\end{proof}

Now we can use this result to show that the limit in Equation~\eqref{eq: definition on Htheta} exists without the monotone assumption. 
\begin{thm}\label{thm: convergence for smooth enough potentials}
Let $f_{0}, f_{1} \in C^{1,\bar{1}}(X)$ and $u_{0} = P_{\theta}(f_{0})$ and $u_{1} = P_{\theta}(f_{1})$. Let $\psi_{k} \in \PSH(X,\theta)$ have analytic singularity such that $\psi_{k} \nearrow V_{\theta}$ almost everywhere. Also define $u_{0}^{k} = P_{\theta}[\psi_{k}](f_{0})$ and $u_{1}^{k} = P_{\theta}[\psi_{k}](f_{1})$. Then the limit in Equation~\eqref{eq: definition on Htheta} exists, and is independent of the choice of the approximating sequence $\psi_{k}$. Moreover, if $u_{t}$ is the weak geodesic joining $u_{0}$ and $u_{1}$, then 
\[
\lim_{k\to \infty} d_{p}^{p}(u_{0}^{k},u_{1}^{k}) = \int_{X} |\dot{u}_{0}|^{p}\theta^{n}_{u_{0}} = \int_{X} |\dot{u}_{1}|^{p}\theta^{n}_{u_{1}}.
\]
\end{thm}

\begin{proof}
    We know the result from Theorem~\ref{thm: definition is well-defined for u0 less than u1} if $f_{0} \leq f_{1}$. To prove it in general, recall that Lemma~\ref{lem: Hthetea closed under rooftop envelope} shows that $P_{\theta}(u_{0}, u_{1}) \in \mathcal{H}_{\theta}$ as well. Here $P_{\theta}(u_{0}, u_{1}) = P_{\te}(P_{C\om}(f_{0},f_{1}))$ and $h:= P_{C\om}(f_{0},f_{1}) \in C^{1,\bar{1}}(X)$. Using the notation from the previous theorem, let $w_{t}$ be the weak geodesic joining $P_{\te}(u_{0}, u_{1})$ and $u_{0}$ and $v_{t}$ be the weak geodesic joining $P_{\te}(u_{0},u_{1})$ and $u_{1}$. 

    Now, $h \leq f_{0}, f_{1}$ are $C^{1,\bar{1}}$ functions. From Theorem~\ref{thm: definition is well-defined for u0 less than u1},
    \[
 \lim_{k\to\infty}d_{p}^{p}(P_{\theta}[\psi_{k}](h), u_{0}^{k}) = \int_{X} |\dot{w}_{0}|^{p}\te^{n}_{P_{\te}(u_{0},u_{1})}
    \]
    and
    \[
    \lim_{k\to \infty} d_{p}^{p}(P_{\theta}[\psi_{k}](h), u_{1}^{k}) = \int_{X} |\dot{v}_{0}|^{p} \te^{n}_{P_{\te}(u_{0},u_{1})}.
    \]
    Lemma~\ref{lem: Pythagorean formula in analytic singularity type} says that the distance $d_{p}$ on $\E^{p}(X,\theta, \psi_{k})$ satisfies the Pythagorean formula. Observing $P_{\te}[\psi_{k}](h) = P_{\te}(u_{0}^{k}, u_{1}^{k}) = P_{\te}[\psi_{k}](f_{0},f_{1})$, we get
    \begin{align*}
        \lim_{k\to \infty}d_{p}^{p}(u_{0}^{k}, u_{1}^{k}) &= \lim_{k\to \infty} d_{p}^{p} (u_{0}^{k}, P_{\theta}[\psi_{k}](h)) + d_{p}^{p}(u_{1}^{k}, P_{\theta}[\psi_{k}](h)) \\
        &= \int_{X}|\dot{w}_{0}|^{p}\theta^{n}_{P_{\theta}(u_{0}, u_{1})} + \int_{X}|\dot{v}_{0}|^{p}\theta^{n}_{P_{\theta}(u_{0}, u_{1})}\\
        &= \int_{X} |\dot{u}_{0}|^{p} \theta^{n}_{u_{0}}.
    \end{align*}
    Here, in the last line, we used Theorem~\ref{pythagorean formula for smooth enough potentials}. Similar proof shows that $\lim_{k\to \infty}d_{p}^{p}(u_{0}^{k}, u_{1}^{k}) = \int_{X} |\dot{u}_{1}|^{p}\te^{n}_{u_{1}}.$
\end{proof}
With the help of Theorem~\ref{thm: convergence for smooth enough potentials}, we see that the limit in Equation~\eqref{eq: definition on Htheta} exists and does not depend on the choice of the approximating sequence. Thus we can define
\begin{mydef}\label{def: definition on Htheta}
    Take $u_{0}, u_{1} \in \mathcal{H}_{\theta}$ where $u_{0} = P_{\theta}(f_{0})$ and $u_{1} = P_{\theta}(f_{1})$ for $f_{0}, f_{1} \in C^{1,\bar{1}}(X)$. Let $\psi_{k} \nearrow V_{\theta}$ outside a pluripolar set be an increasing sequence of $\theta$-psh function with analytic singularities. Denote by $u_{0}^{k} = P_{\theta}[\psi_{k}](f_{0})$ and $u_{1}^{k} = P_{\theta}[\psi_{k}](f_{1})$. We define
    \[
    d_{p}(u_{0}, u_{1}):= \lim_{k\to \infty} d_{p}(u_{0}^{k}, u_{1}^{k}).
    \]
    By Theorem~\ref{thm: convergence for smooth enough potentials}, the limit exists and is independent of the choice of approximating sequence. 
\end{mydef}
The next theorem shows that Equation~\eqref{eq: definition on Htheta} indeed defines a metric on $\mathcal{H}_{\theta}$. 

\begin{thm}\label{thm: dp is a metric on Htheta}
    If $d_{p}$ is defined as in Definition~\ref{def: definition on Htheta}, then $(\mathcal{H}_{\theta}, d_{p})$ is a metric space and $d_{p}$ is comparable to $I_{p}$. This means there exits $C > 1$, depending only on dimension such that for all $u_{0}, u_{1} \in \mathcal{H}_{\theta}$, 
    \[
    \frac{1}{C} I_{p}(u_{0}, u_{1}) \leq d_{p}^{p}(u_{0}, u_{1}) \leq C I_{p}(u_{0},u_{1}).
    \]
\end{thm}

\begin{proof}
    From Lemma~\ref{convergence of p-energy}, we know that $\lim_{k\to \infty} I_{p}(u_{0}^{k}, u_{1}^{k}) = I_{p}(u_{0}, u_{1})$. From Lemma~\ref{lem: Ip compares with dp in analytic singularity type as well}, we know that there exists $C$ such that 
    \[
    \frac{1}{C} I_{p}(u_{0}^{k}, u_{1}^{k}) \leq d_{p}^{p}(u_{0}, u_{1}) \leq C I_{p}(u_{0}, u_{1})
    \]
    Taking limit $k \to \infty$, we get
    \[
    \frac{1}{C} I_{p}(u_{0}, u_{1}) \leq d_{p}^{p}(u_{0}, u_{1}) \leq C I_{p}(u_{0}, u_{1}).
    \]

    Symmetry and triangle inequality for $d_{p}$ follow from the definition and the fact that $(\E^{p}(X,\theta, \psi_{k}), d_{p})$ satisfy these properties. Non-degeneracy of $d_{p}$ follows from comparison with $I_{p}$. If $d_{p}(u_{0}, u_{1}) = 0$, then the above comparison shows that $I_{p}(u_{0}, u_{1}) = 0$. This implies that $u_{0} = u_{1}$ from the domination principle (see \cite[Proposition 2.4]{Darvas2017OnTS}). 
\end{proof}

\subsection{Extending the metric to $\E^{p}(X,\theta)$}\label{subsec: extending the metric}
In this section, we will extend the metric $d_{p}$ from $\mathcal{H}_{\theta}$ to all of $\E^{p}(X,\theta)$. We will do this by approximation. This process of approximation works identically to the one given in \cite{dinezza2018lp}. Given $u \in \E^{p}(X,\theta)$, from \cite{BlockiKolodziejregularization}, we can find smooth functions $f^{j}$ such that $f^{j} \searrow u$. By definition $u^{j} := P_{\te}(f^{j}) \in \cH_{\te}$ and $u^{j} \searrow u$. Based on this we give a tentative definition:
\begin{mydef}\label{def: extending dp to Ep}
Given $u_{0}, u_{1} \in \E^{p}(X,\te)$, we define
\begin{equation}\label{eq: definition of dp}
d_{p}(u_{0},u_{1}) := \lim_{j\to \infty} d_{p}(u_{0}^{j}, u_{1}^{j}),    
\end{equation}
where $u_{0}^{j}, u_{1}^{j} \in \cH_{\te}$ satisfy $u_{0}^{j} \searrow u_{0}$ and $u_{1}^{j} \searrow u_{1}$. 
    
\end{mydef}

Now we need to show the limit in Equation~\eqref{eq: definition of dp} exists and is independent of the choice of the approximating sequence $u_{0}^{j}$ and $u_{1}^{j}$. 

\begin{thm}\label{thm: approximation is well defined}
    The limit in Equation~\eqref{eq: definition of dp} exists and is independent of the choice of the approximating sequence $u_{0}^{j}$ and $u_{1}^{j}$. 
\end{thm}

\begin{proof}

From Theorem~\ref{thm: dp is a metric on Htheta} for $u,v \in \mathcal{H}_{\theta}$, there exists $C  > 1$, depending only on $n$, such that 
\[
\frac{1}{C} I_{p}(u,v) \leq d^{p}_{p}(u,v) \leq CI_{p}(u,v).
\]
We will show that $\{d_{p}(u_{0}^{j}, u_{1}^{j})\}$ is a Cauchy sequence. By the triangle inequality we have 
\begin{align*}
    d_{p}(u_{0}^{j}, u_{1}^{j}) &\leq d_{p}(u_{0}^{j}, u_{0}^{k}) + d_{p}(u_{0}^{k}, u_{1}^{k}) + d_{p}(u_{1}^{k}, u_{1}^{j}) \\
   \implies d_{p}(u_{0}^{j}, u_{1}^{j}) - d_{p}(u_{0}^{k}, u_{1}^{k}) &\leq C(I_{p}^{1/p}(u_{0}^{j}, u_{0}^{k}) + I^{1/p}_{p}(u_{1}^{k}, u_{1}^{j}))..
   \intertext{Since the other side is obtained identically, we get}
   |d_{p}(u_{0}^{j}, u_{1}^{j}) - d_{p}(u_{0}^{k}, u_{1}^{k})| &\leq C(I_{p}^{1/p}(u_{0}^{j}, u_{0}^{k}) + I^{1/p}_{p}(u_{1}^{k}, u_{1}^{j})).
\end{align*}
From Theorem~\ref{thm: Ip converges along decreasing sequences} we get $I_{p}(u_{0}^{j}, u_{0}^{k}) \to 0$ and $I_{p}(u_{1}^{j}, u_{1}^{k}) \to 0$ as $j,k \to \infty$. Thus we have $|d_{p}(u_{0}^{j}, u_{1}^{j}) - d_{p}(u_{0}^{k}, u_{1}^{k})| \to 0$. Thus the limit in Equation~\eqref{eq: definition of dp} exists. Now we will show that the limit is unique. For that let $\tilde{u}_{0}^{j},\tilde{u}_{1}^{j} \in \cH_{\te}$ be another sequence of functions decreasing to $u_{0}$ and $u_{1}$ respectively. To show that the definition of $d_{p}$ does not depend on the choice of functions approximating $u_{0}$ and $u_{1}$, we will show that $|d_{p}(u_{0}^{j}, u_{1}^{j}) - d_{p}(\tilde{u}_{0}^{j}, \tilde{u}_{1}^{j})| \to 0$ as $j \to \infty$. The proof is similar to the proof before. 

\begin{align*}
    d_{p}(u_{0}^{j}, u_{1}^{j}) &\leq d_{p}(u_{0}^{j}, \tilde{u}_{0}^{j}) + d_{p}(\tilde{u}_{0}^{j}, \tilde{u}_{1}^{j}) + d_{p}(\tilde{u}_{1}^{j}, u_{1}^{j}) \\
   \implies |d_{p}(u_{0}^{j}, u_{1}^{j}) - d_{p}(\tilde{u}_{0}^{j}, \tilde{u}_{1}^{j})| &\leq C(I_{p}^{1/p}(u_{0}^{j}, \tilde{u}_{0}^{j}) + I_{p}^{1/p}(u_{1}^{j}, \tilde{u}_{1}^{j})).
\end{align*}
Since $u_{0}^{j}$ and $\tilde{u}_{0}^{j}$ both decrease to $u_{0}$, Theorem~\ref{thm: Ip converges along decreasing sequences} implies that $I_{p}(u_{0}^{j}, \tilde{u}_{0}^{j}) \to 0$. Similarly $I_{p}(u_{1}^{j}, \tilde{u}_{1}^{j}) \to 0$ as well. So $d_{p}$ is well defined on $\E^{p}(X,\theta)$ by Equation~\eqref{eq: definition of dp}.
\end{proof}

\begin{lem}\label{lem: compares to Ip}
    There exists $C > 1$ that depends only on the dimension of $X$, such that for all $u_{0},u_{1} \in \E^{p}(X,\theta)$, 
    \[
    \frac{1}{C}I_{p}(u_{0},u_{1}) \leq d_{p}^{p}(u_{0},u_{1}) \leq CI_{p}(u_{0},u_{1}). 
    \]
\end{lem}

\begin{proof}
    The statement is true for potentials in $\mathcal{H}_{\theta}$. Let $u_{0}^{j} \searrow u_{0}$ and $u_{1}^{j} \searrow u_{1}$. Then 
 \begin{align*}
       \frac{1}{C} I_{p}(u_{0}^{j}, u_{1}^{j}) &\leq d_{p}^{p}(u_{0}^{j}, u_{1}^{j}) \leq CI_{p}(u_{0}^{j},u_{1}^{j})
       \intertext{Taking the limit $j \to \infty$ and applying Theorem~\ref{thm: Ip converges along decreasing sequences} and using Equation~\eqref{eq: definition of dp} we get}
      \frac{1}{C} I_{p}(u_{0},u_{1}) &\leq d_{p}(u_{0},u_{1}) \leq CI_{p}(u_{0},u_{1}).
   \end{align*}
\end{proof}

\begin{thm}\label{thm: dp is a metric}
    Equation~\eqref{eq: definition of dp} defines a metric on $d_{p}$ on $\E^{p}(X,\theta)$. 
\end{thm}

\begin{proof}
    Again, using approximation, we can show the triangle inequality. Let $u,v,w \in \E^{p}(X,\theta)$ and $u^{j}, v^{j}, w^{j} \in \mathcal{H}_{\theta}$ approximate $u,v$, and $w$ respectively. Then
     \begin{align*}
         d_{p}(u,v) &= \lim_{j\to\infty}d_{p}(u^{j}, v^{j}) \\
         &\leq \lim_{j\to\infty}(d_{p}(u^{j},w^{j}) + d_{p}(w^{j},v^{j})) \\
         &= d_{p}(u,w) + d_{p}(w,v).
     \end{align*}
     This shows the triangle inequality for $d_{p}$. Symmetry also follows from symmetry of $d_{p}$ on $\mathcal{H}_{\theta}$. Non-degeneracy of $d_{p}$ follows from Lemma~\ref{lem: compares to Ip}. If $u,v\in \E^{p}(X,\theta)$ have satisfy $d_{p}(u,v) = 0$, then Lemma~\ref{lem: compares to Ip} says $I_{p}(u,v) = 0$, which implies that $u = v$ by the domination principle (see \cite[Proposition 2.4]{Darvas2017OnTS}). 
\end{proof}

\section{Properties of the metric}\label{sec: properties}

In this section, we will show that the metric space $(\E^{p}(X,\theta), d_{p})$ is a complete geodesic metric space. 

\begin{thm}\label{thm: Ep is complete}
    The metric space $(\E^{p}(X,\theta), d_{p})$ is a complete metric space. 
\end{thm}
\begin{proof}
    From Lemma~\ref{lem: compares to Ip}, there exists a $C > 1$ such that for any $u_{0}, u_{1} \in \E^{p}(X,\theta)$
    \[
    \frac{1}{C} I_{p}(u_{0}, u_{1}) \leq d_{p}^{p}(u_{0}, u_{1}) \leq CI_{p}(u_{0},u_{1}).
    \]
    In \cite{GuptaCompletemetrictopology}, the author showed that the quasi-metric space $(\E^{p}(X,\theta), I_{p})$ induces a complete metric topology. This means given a $I_{p}$-Cauchy sequence $\{u_{k}\}$, there exits $u \in \E^{p}(X,\theta)$ such that $I_{p}(u_{k}, u) \to 0$. 

    From the above inequality, a sequence $\{ u_{k}\}$ is Cauchy in $I_{p}$ iff it is Cauchy in $d_{p}$ and similarly, a sequence $u_{k}$ converges to $u$ in $I_{p}$ iff $u_{k}$ converges to $u$ in $d_{p}$.

    This shows that any $d_{p}$-Cauchy sequence $\{u_{k}\}$ converges to some $u \in \E^{p}(X,\theta)$. 
\end{proof}

Now we want to show that the Mabuchi geodesics in $\E^{p}(X,\theta)$ are the metric geodesics as well. For that, we need to better understand the metric space structure of $\E^{p}(X,\theta)$. 
\begin{lem} \label{lem: decreasing sequence in Ep}
    If $u_{0}, u_{1}, u_{0}^{j}, u_{1}^{j} \in \E^{p}(X,\theta)$ satisfy $u_{0}^{j} \searrow u_{0}$ and $u_{1}^{j} \searrow u_{1}$, then $\lim_{j\to \infty} d_{p}(u_{0}^{j}, u_{1}^{j}) = d_{p}(u_{0}, u_{1})$. 
\end{lem}

\begin{proof}
    Recall that \cite[Propostion 1.9]{Guedj2019PlurisubharmonicEA} implies that $I_{p}(u_{0}^{j}, u_{0}) \to 0$ and $I_{p}(u_{1}^{j}, u_{1}) \to 0$. As before, we use triangle inequality to write
    \begin{align*}
        d_{p}(u_{0}, u_{1}) &\leq d_{p}(u_{0}, u_{0}^{j}) + d_{p}(u_{0}^{j}, u_{1}^{j}) + d_{p}(u_{1}^{j}, u_{1}). 
        \intertext{Using Lemma~\ref{lem: compares to Ip}}
        d_{p}(u_{0}, u_{1}) - d_{p}(u_{0}^{j}, u_{1}^{j}) &\leq C\left( I_{p}^{1/p}(u_{0}^{j}, u_{0}) + I_{p}^{1/p}(u_{1}^{j}, u_{1})\right).
        \intertext{Noticing that the other side is obtained similarly, and then we take the limit to obtain}
        \lim_{j\to \infty}|d_{p}(u_{0}, u_{1}) - d_{p}(u_{0}^{j}, u_{1}^{j})| &\leq \lim_{j\to \infty}C\left(I_{p}^{1/p}(u_{0}^{j}, u_{0}) + CI_{p}^{1/p}(u_{1}^{j}, u_{1})\right) = 0.
    \end{align*}
\end{proof}

The following is the extension of \cite[Lemma 3.13]{dinezza2018lp} to the big case. The proof is identical.
\begin{lem}\label{lem: distance is speed of geodesic}
    If $u_{0} \in \mathcal{H}_{\theta}$, $u_{1} \in \E^{p}(X,\theta)$, and $u_{t}$ is the weak geodesic joining $u_{0}$ and $u_{1}$, then 
    \[
    d_{p}^{p}(u_{0}, u_{1}) = \int_{X} |\dot{u}_{0}|^{p}\theta^{n}_{u_{0}}.
    \]
\end{lem}

\begin{proof}
    First, assume that $u_{0} \geq u_{1} + 1$. We can find $u_{1}^{j} \in \cH_{\te}$ such that $u_{1}^{j} \searrow u_{1}$ and $u_{0} \geq u_{1}^{j}$. Let $u_{t}^{j}$ be the weak geodesic joining $u_{0}$ and $u_{1}^{j}$. Moreover, $\dot{u}_{0},\dot{u}_{0}^{j} \leq 0$. We claim $\dot{u}_{0}^{j} \searrow \dot{u}_{0}$. Since the geodesics $u_{t}^{j} \searrow u_{t}$, and they start at the same point $u_{0}$, we know that $\dot{u}_{0}^{j}$ is decreasing. To see that they decrease to $\dot{u}_{0}$, notice that 
    \[
    \dot{u}_{0} = \lim_{t \to 0}\frac{u_{t} - u_{0}}{t} \leq \lim_{t\to 0} \frac{u_{t}^{j} - u_{0}}{t} = \dot{u}_{0}^{j} \leq \frac{u_{t}^{j}- u_{0}}{t}.
    \]
    Here in the last inequality, we used the convexity of the geodesic. Now, taking limit $j \to \infty$, we get 
    \[
    \dot{u}_{0} \leq \lim_{j\to \infty} \dot{u}_{0}^{j} \leq \frac{u_{t} - u_{0}}{t}.
    \]
    Now taking limit $t \to 0$, we get 
    \[
    \dot{u}_{0} \leq \lim_{j\to \infty} \dot{u}_{0}^{j} \leq \dot{u}_{0}.
    \]
    Thus $\lim_{j\to \infty}\dot{u}_{0}^{j} = \dot{u}_{0}.$

    Now, by definition, $d_{p}(u_{0}, u_{1}) = \lim_{j\to \infty} d_{p}(u_{0}, u_{1}^{j})$ and 
    \[
    d_{p}^{p}(u_{0}, u_{1}^{j}) = \int_{X} (-\dot{u}_{0}^{j})^{p}\theta^{n}_{u_{0}}.
    \]
    By the monotone convergence theorem, 
    \[
    d_{p}^{p}(u_{0}, u_{1}) = \lim_{j\to \infty}d_{p}^{p}(u_{0}, u_{1}^{j}) = \lim_{j\to \infty} \int_{X}(-\dot{u}_{0}^{j})^{p}\theta^{n}_{u_{0}} = \int_{X} (-\dot{u}_{0})^{p}\theta^{n}_{u_{0}}.
    \]

    For the general case, let $C > 0$ satisfy $u_{1} \leq u_{0} + C$. Again choose $u_{1}^{j} \in \cH_{\te}$ such that $u_{1}^{j} \searrow u_{1}$. Consider $w_{0} = u_{0}$ and $w_{1} = u_{1} - C - 1 \leq u_{1} \leq u_{1}^{j}$. Now $w_{0} \geq w_{1} + 1$. If $w_{t}$ is the geodesic joining $w_{0}$ and $w_{1}$ and if $u_{t}^{j}$ are the geodesics joining $u_{0}$ and $u_{t}^{j}$, then we have 
    \[
    \dot{w}_{0} \leq \dot{u}_{0}^{j} \leq u_{1}^{j} - u_{0} \leq (u_{1}^{j} - V_{\theta}) + (V_{\theta} - u_{0}) \leq C,
    \]
    where $C$ is a uniform bound (independent of $j$). Thus, $|\dot{u}_{0}^{j}|^{p} \leq C_{1} + |\dot{w}_{0}|^{p}$. By the same argument as before, $\dot{u}_{0}^{j} \to \dot{u}_{0}$ pointwise  outside the pluripolar set $\{u_{1} = -\infty\}$. Moreover, from the previous calculation
    \[
    \int_{X}|\dot{w}_{0}|^{p}\theta^{n}_{u_{0}} = d_{p}^{p}(u_{0}, u_{1}-C-1).
    \]
    Thus $|\dot{w}_{0}|^{p}$ is integrable with respect to $\theta^{n}_{u_{0}}$. Thus applying Lebesgue's Dominated Convergence Theorem, we obtain
    \[
    d_{p}^{p}(u_{0}, u_{1}) = \lim_{j\to \infty}d_{p}^{p}(u_{0}, u_{1}^{j}) = \lim_{j\to \infty} \int_{X}|\dot{u}_{j}|^{p}\theta^{n}_{u_{0}} = \int_{X}|\dot{u}_{0}|^{p}\theta^{n}_{u_{0}}.
    \]
\end{proof}

\begin{thm} \label{thm: weak geodesics are metric geodesics}
    Take $u_{0}, u_{1} \in \E^{p}(X,\theta)$ and let $u_{t}$ be the weak geodesic joining $u_{0}$ and $u_{1}$. Then $u_{t}$ is a metric geodesic for $(\E^{p}(X,\theta), d_{p})$. This means that for any $0 \leq t \leq s \leq 1$, $d_{p}(u_{t}, u_{s}) = |t-s|d_{p}(u_{0}, u_{1})$. 
\end{thm}

\begin{proof}
    The same proof as in \cite[Theorem 3.17]{dinezza2018lp} works in our case as well. We will first show that given $0 \leq t \leq 1$, we have 
    \[
    d_{p}(u_{0}, u_{t}) = t d_{p}(u_{0}, u_{1}) \qquad \text{ and } \qquad d_{p}(u_{1}, u_{t}) = (1-t)d_{p}(u_{0}, u_{1}).
    \]
    First, assume that $u_{0}, u_{1} \in \mathcal{H}_{\theta}$. Let $w_{s} = u_{ts}$ be the geodesic joining $u_{0}$ and $u_{t}$. By Lemma~\ref{lem: distance is speed of geodesic}, we obtain that 
    \[
    d_{p}^{p}(u_{0}, u_{t}) = \int_{X} |\dot{w}_{0}|^{p}\theta^{n}_{u_{0}} = t^{p}\int_{X}|\dot{u}_{0}|^{p}\theta^{n}_{u_{0}} = t^{p}d_{p}^{p}(u_{0},u_{1}).
    \]
    The other equality is proved similarly. 

    For arbitrary $u_{0}, u_{1} \in \E^{p}(X,\theta)$, find $u_{0}^{j}, u_{1}^{j} \in \cH_{\te}$ such that $u_{0}^{j} \searrow u_{0}$ and $u_{1}^{j} \searrow u_{1}$. If $u_{t}^{j}$ is the weak geodesic joining $u_{0}^{j}$ and $u_{1}^{j}$, then $u_{t}^{j} \searrow u_{t}$. As $d_{p}(u_{0}^{j}, u_{t}^{j}) = td_{p}(u_{0}^{j}, u_{1}^{j})$, taking limit $j\to \infty$ using Lemma~\ref{lem: decreasing sequence in Ep}, we obtain $d_{p}(u_{0}, u_{t}) = t d_{p}(u_{0}, u_{1})$. 

    Now, for a more general case, if $0 < t < s < 1$, then applying the above result twice, we get 
    \[
    d_{p}(u_{t}, u_{s}) = \frac{s-t}{s}d_{p}(u_{0}, u_{s}) = (s-t) d_{p}(u_{0}, u_{1}).
    \]
\end{proof}

Lastly, we prove that the metric $d_{p}$ satisfies the Pythagorean identity. 
\begin{thm}\label{thm: Pythagorean formula for big classes}
    If $u_{0}, u_{1} \in \E^{p}(X,\te)$, then 
    \[
    d_{p}^{p}(u_{0}, u_{1}) = d_{p}^{p}(u_{0}, P_{\te}(u_{0}, u_{1})) + d_{p}^{p}(u_{1}, P_{\te}(u_{0}, u_{1})).
    \]
\end{thm}

\begin{proof}
    If $u_{0}, u_{1} \in \cH_{\te}$, then this is the content of Theorem~\ref{pythagorean formula for smooth enough potentials} when combined with Theorem~\ref{thm: convergence for smooth enough potentials}.  

    More generally, if $u_{0}, u_{1} \in \E^{p}(X,\te)$, then we can find $u_{0}^{k}, u_{1}^{k} \in \cH_{\te}$ such that $u_{0}^{k} \searrow u_{0}$ and $u_{1}^{k} \searrow u_{1}$. Then $P_{\te}(u_{0}^{k}, u_{1}^{k}) \searrow P_{\te}(u_{0}, u_{1})$ as well. Thus 
    \begin{align*}
        d_{p}^{p}(u_{0}, u_{1}) &= \lim_{k \to \infty} d_{p}^{p}(u_{0}^{k}, u_{1}^{k}) \\
        &= \lim_{k \to \infty} d_{p}^{p}(u_{0}^{k}, P_{\te}(u_{0}^{k}, u_{1}^{k})) + d_{p}^{p} (u_{1}^{k}, P_{\te}(u_{0}^{k}, u_{1}^{k}))\\
        &= d_{p}^{p}(u_{0}, P_{\te}(u_{0},u_{1})) + d_{p}^{p}(u_{1}, P_{\te}(u_{0}, u_{1})).
    \end{align*}
\end{proof}

\subsection{Connection with the metric in the literature}
In this subsection, we prove that when $\te$ is big and nef, or when $p = 1$, then the metric $d_{p}$ on $\E^{p}(X,\te)$ constructed in Section~\ref{sec: metric in the big case} coincides with the metric $d_{p}$ constructed in \cite{dinezza2018lp} and \cite{darvas2021l1}. 

\begin{thm}\label{thm: dp agrees with dinezzalp}
    If $\bb$ is a smooth closed real $(1,1)$-form representing a big and nef cohomology class, then the metric $d_{p}$ constructed in Section~\ref{sec: metric in the big case} agrees with the one constructed in \cite{dinezza2018lp}. 
\end{thm}

\begin{proof}
    Let us use $D_{p}$ to represent the metric constructed in \cite{dinezza2018lp}. In case $u_{0}, u_{1} \in \cH_{\te}$, then by Theorem~\ref{thm: convergence for smooth enough potentials} and by \cite[Theorem 3.7]{dinezza2018lp}
    \[
    d_{p}(u_{0}, u_{1}) = \int_{X}|\dot{u}_{0}|^{p} \bb^{n}_{u_{0}} = D_{p}(u_{0}, u_{1}),
    \]
    where $u_{t}$ is the weak geodesic joining $u_{0}$ and $u_{1}$. 

    By Definition~\ref{def: extending dp to Ep} and the definition of $D_{p}$ in \cite[Equation above Proposition 3.12]{dinezza2018lp} when $u_{0}, u_{1} \in \E^{p}(X,\bb)$, then 
    \[
    d_{p}(u_{0}, u_{1}) = \lim_{k \to \infty} d_{p}(u_{0}^{k}, u_{1}^{k}) = \lim_{k \to \infty} D_{p}(u_{0}^{k}, u_{1}^{k}) = D_{p}(u_{0}, u_{1})
    \]
    where $u_{0}^{k}, u_{1}^{k} \in \cH_{\bb}$ such that $u_{0}^{k} \searrow u_{0}$ and $u_{1}^{k} \searrow u_{1}$. 
\end{proof}

\begin{thm}\label{thm: d1 agrees with DNL d1}
    When $p = 1$, then $u_{0}, u_{1} \in \E^{1}(X,\te)$ satisfy
    \[
    d_{1}(u_{0}, u_{1}) = I(u_{0}) + I(u_{1}) - 2I(P_{\te}(u_{0}, u_{1})).
    \]
    Thus $d_{1}$ agrees with the metric constructed in \cite{darvas2021l1}. 
\end{thm}
\begin{proof}
    The proof is the same as in \cite[Propositiion 3.18]{dinezza2018lp}. We recall the steps for completion. If $u_{0}, u_{1} \in \cH_{\te}$ and $u_{0} \leq u_{1}$, then from Lemma~\ref{lem: monge ampere energy and geodesic in big case} and Theorem~\ref{thm: convergence for smooth enough potentials}, 
    \[
    d_{1}(u_{0}, u_{1}) = \int_{X}\dot{u}_{0} \te^{n}_{u_{0}} = \int_{X}\dot{u}_{1}\te^{n}_{u_{1}} = I(u_{1}) - I(u_{0}). 
    \]
    If $u_{0}, u_{1} \in \cH_{\te}$ be arbitrary (we drop the condition that $u_{0} \leq u_{1}$), then by the Pythagorean identity (see Theorem~\ref{thm: Pythagorean formula for big classes}), 
    \begin{align*}
        d_{1}(u_{0}, u_{1}) &= d_{1}(u_{0}, P_{\te}(u_{0}, u_{1})) + d_{1}(u_{1}, P_{\te}(u_{0}, u_{1})) \\ 
        &= I(u_{0}) + I(u_{1}) - 2 I(P_{\te}(u_{0}, u_{1})).
    \end{align*}

    More generally, when $u_{0}, u_{1} \in \E^{p}(X,\te)$, then using $u_{0}^{k}, u_{1}^{k} \in \cH_{\te}$ such that $u_{0}^{k} \searrow u_{0}$ and $u_{1}^{k} \searrow u_{1}$, we can prove that 
    \begin{align*}
        d_{1}(u_{0}, u_{1}) &= \lim_{k \to \infty} d_{1}(u_{0}^{k}, u_{1}^{k}) \\
        &= \lim_{k \to \infty} I(u_{0}^{k}) + I(u_{1}^{k}) - 2I(P_{\te}(u_{0}^{k}, u_{1}^{k})) \\
        &= I(u_{0}) + I(u_{1}) - 2I(P_{\te}(u_{0}, u_{1})).
    \end{align*}    
\end{proof}

\section{Uniform Convexity for the big and nef classes}\label{sec: uniform convexity for big and nef}
On a compact K\"ahler manifold $(X,\om)$, in \cite{DarvasLuuniformconvexity} Darvas-Lu proved that for $p> 1$, $u,v_{0}, v_{1} \in \E^{p}(X,\om)$, and $(0,1) \ni \lambda \mapsto v_{\lambda} \in \E^{p}(X,\om)$, the weak geodesic joining $v_{0}$ and $v_{1}$, satisfy
\begin{align}\label{eq: convexity in Kahler case}
d_{p}(u,v_{\lambda})^{2} &\leq (1-\la)d_{p}(u,v_{0})^{2} + \la d_{p}(u,v_{1})^{2} - (p-1)\la(1-\la)d_{p}(v_{0},v_{1})^{2}, \text{ if } 1< p \leq 2 \text{ and }\\
d_{p}(u,v_{\la})^{p} &\leq (1-\la)d_{p}(u,v_{0})^{p} + \la d_{p}(u,v_{1})^{p} - \la^{p/2}(1-\la)^{p/2}d_{p}(v_{0},v_{1})^{p}, \text{ if } p \leq 2.
\end{align}

If $\bb$ represents a big and nef cohomology class, using the approximation method used to construct the metric $d_{p}$ on $\E^{p}(X,\bb)$, in this section we will extend these inequalities to $\E^{p}(X,\bb)$.

First, we will show the convexity property on $\cH_{\bb}$ (see Equation~\eqref{eq: definition of Hbb}). If $u, v_{0}, v_{1} \in \cH_{\bb}$ and $\lambda \mapsto v_{\lambda}$ is the weak geodesic joining $v_{0}$ and $v_{1}$, then we claim 
\begin{align}
    d_{p}(u,v_{\lambda})^{2} &\leq (1-\la)d_{p}(u,v_{0})^{2} + \la d_{p}(u,v_{1})^{2} - (p-1)\la(1-\la)d_{p}(v_{0},v_{1})^{2}, \text{ if } 1< p \leq 2 \text{ and }\\
d_{p}(u,v_{\la})^{p} &\leq (1-\la)d_{p}(u,v_{0})^{p} + \la d_{p}(u,v_{1})^{p} - \la^{p/2}(1-\la)^{p/2}d_{p}(v_{0},v_{1})^{p}, \text{ if } p \leq 2.
\end{align}

We will show it by the approximation process. Let $u = P_{\bb}(f)$, $v_{0} = P_{\bb}(f_{0})$, and $v_{1} = P_{\bb}(f_{1})$ for $f,f_{0}, f_{1} \in C(X)$ such that $dd^{c}f, dd^{c}f_{0}, dd^{c}f_{2} \leq C\om$. Recall that we defined $\om_{\ee} := \bb + \ee \om$. Let $u_{\ee} = P_{\om_{\ee}}(f)$, $v_{0,\ee} = \Poe(f_{0})$ and $v_{1,\ee} = \Poe(f_{1})$. Let $v_{\la,\ee}$ be the geodesic joining $v_{0,\ee}$ and $v_{1,\ee}$. From the result in the K\"ahler case, we know that 
\begin{align}
     d_{p}(u_{\ee},v_{\lambda,\ee})^{2} &\leq (1-\la)d_{p}(u_{\ee},v_{0,\ee})^{2} + \la d_{p}(u_{\ee},v_{1,\ee})^{2} - (p-1)\la(1-\la)d_{p}(v_{0,\ee},v_{1,\ee})^{2}, \text{ if } 1< p \leq 2 \text{ and } \label{eq: convexity for approximating sequence}\\
d_{p}(u_{\ee},v_{\la,\ee})^{p} &\leq (1-\la)d_{p}(u_{\ee},v_{0,\ee})^{p} + \la d_{p}(u_{\ee},v_{1,\ee})^{p} - \la^{p/2}(1-\la)^{p/2}d_{p}(v_{0,\ee},v_{1,\ee})^{p}, \text{ if } p \leq 2. \label{eq: convexity for approximating sequence p > 2}
\end{align}

By the definition of $d_{p}$ on $\cH_{\bb}$ (see Section~\ref{subsec: metric geometry in big and nef case}), we know that the $\lim_{\ee \to 0} d_{p}(u_{\ee}, v_{0,\ee}) = d_{p}(u,v_{0})$, $\lim_{\ee \to 0} d_{p}(u_{\ee}, v_{1,\ee}) = d_{p}(u,v_{1})$, and $\lim_{\ee\to 0} d_{p}(v_{0,\ee}, v_{1,\ee}) = d_{p}(v_{0},v_{1})$. So we are done if we can prove that 
\[
\lim_{\ee \to 0}d_{p}(u_{\ee}, v_{\la,\ee}) = d_{p}(u, v_{\la}).
\]

Let $w_{t}$ be the weak geodesic joining $u$ and $v_{\la}$ and $w_{t,\ee}$ be the weak geodesic joining $u_{\ee}$ and $v_{\la, \ee}$. 
Since $u \in \cH_{\bb}$ and $u_{\ee} \in \cH_{\om_{\ee}}$ from \cite[Lemma 3.13]{dinezza2018lp} we get that 
\[
d_{p}(u,v_{\la})^{p} = \int_{X} |\dot{w}_{0}|^{p}(\bb + dd^{c}u)^{n},
\]
and 
\[
d_{p}(u_{\ee}, v_{\la,\ee})^{p} = \int_{X} |\dot{w}_{0,\ee}|^{p}(\om_{\ee} + dd^{c}u_{\ee})^{n}.
\]
Using \cite[Lemma 3.5]{dinezza2018lp}, we get that if $(\bb + dd^{c}u)^{n} = \rho \om^{n}$ and $(\om_{\ee} + dd^{c}u_{\ee})^{n} = \rho_{\ee}\om^{n}$, then $\ee \mapsto \rho_{\ee}$ is increasing, uniformly bounded and $\rho_{\ee} \searrow \rho$ pointwise as $\ee \to 0$. Moreover, from Theorem~\ref{thm: measure on contact sets}, the measure $(\bb + dd^{c}u)^{n}$ is supported on $D:= \{P_{\bb}(f) = f\}$ and the measures $(\om_{\ee} + dd^{c}u_{\ee})^{n}$ are supported on $D_{\ee}:= \{ \Poe(f) = f\}$. Moreover, $\cap_{\ee >0}D_{\ee} = D$. 

We will show that 
\[
\lim_{\ee \to 0} \int_{X} |\dot{w}_{0,\ee}|^{p}(\om_{\ee} + dd^{c}u_{\ee})^{n} = \int_{X} |\dot{w}_{0}|^{p}(\bb + dd^{c}u)^{n}.
\]
The proof is similar to \cite[ Lemma 3.6, and Theorem 3.7]{dinezza2018lp}.

\begin{lem}\label{lem: convergence of integrands}
Let $w_{t}$ and $w_{t,\ee}$ be the weak geodesics joining $u, v_{\la}$ and $u_{\ee}, v_{\la,\ee}$ respectively as described above. Then
\[
\lim_{\ee \to 0}\mathds{1}_{D_{\ee}}|\dot{w}_{0,\ee}|^{p} = \mathds{1}_{D}|\dot{w}_{0}|^{p}.
\]
\end{lem}

\begin{proof}
    First we observe that $u_{\ee} \searrow u$ and $v_{\la,\ee} \searrow v_{\la}$ as $\ee \to 0$. We will explain why $v_{\la, \ee} \searrow v_{\la}$. This follows because $v_{0,\ee} \searrow v_{0}$ and $v_{1,\ee} \searrow v_{1}$ as $\ee \to 0$. If $\ee_{1} < \ee_{2}$, then the geodesic $v_{\la, \ee_{1}}$ is a candidate subgeodesic joining $v_{0, \ee_{2}}$ and $v_{1,\ee_{2}}$. Therefore the geodesics $v_{\la,\ee}$ are decreasing sequence of $\om_{\ee}$-psh functions. If $v_{\la,\ee} \searrow \phi_{\la}$, then $\phi_{\la} \geq v_{\la}$ because $v_{\la,\ee} \geq v_{\la}$. But $\phi_{\la}$ is a candidate subgeodesic joining $v_{0}$ and $v_{1}$, therefore $\phi_{\la} \leq v_{\la}$. Thus $v_{\la,\ee} \searrow v_{\la}$ as $\ee \to 0$.  

    We can obtain that $w_{t,\ee} \searrow w_{t}$ as well, by the same reasoning. 

    If $x \in D$, then 
    \[
    \dot{w}_{0}(x) = \lim_{t\to 0}\frac{w_{t}(x) - f(x)}{t} \leq \lim_{t\to 0}\frac{w_{t,\ee}(x) - f(x)}{t} = \dot{w}_{0,\ee} \leq \frac{w_{t,\ee}(x) - w_{0,\ee}(x)}{t}.
    \]
    Here we used $w_{t,\ee} \geq w_{t}$ and the convexity of the geodesic $w_{t,\ee}$ for the last inequality. Taking $\ee \to 0$, and using $w_{t,\ee} \searrow w_{t}$ we obtain
    \[
    \dot{w}_{0}(x) \leq \lim_{\ee \to 0} \dot{w}_{0,\ee}(x) \leq \frac{w_{t}(x) - w_{0}(x)}{t}.
    \]
    Taking $t \to 0$, we get 
    \[
    \dot{w}_{0}(x) \leq \lim_{\ee \to 0} \dot{w}_{0,\ee}(x) \leq \dot{w}_{0}(x).
    \]
    Thus if $x \in D$, then $\lim_{\ee \to 0}\dot{w}_{0,\ee}(x) = \dot{w}_{0}(x)$. If $x \notin D$, then for $\ee$ small enough $x \notin D_{\ee}$. Thus we get $\mathds{1}_{D_{\ee}} |\dot{w}_{0,\ee}|^{p} = \mathds{1}_{D} |\dot{w}_{0}|^{p} $ as $\ee \to 0$ pointwise.
\end{proof}

\begin{thm} \label{thm: convergence of geodesics in Hbb}
    \[
\lim_{\ee \to 0} \int_{X} |\dot{w}_{0,\ee}|^{p}(\om_{\ee} + dd^{c}u_{\ee})^{n} = \int_{X} |\dot{w}_{0}|^{p}(\bb + dd^{c}u)^{n}.
\]
\end{thm}
\begin{proof}
    We first notice that since $(\om_{\ee} + dd^{c}u_{\ee})^{n} = \rho_{\ee} \om^{n}$ and is supported on $D_{\ee}$, therefore 
    \[
    \int_{X} |\dot{w}_{0,\ee}|^{p}(\om_{\ee} + dd^{c}u_{\ee})^{n} = \int_{X}\mathds{1}_{D_{\ee}}|\dot{w}_{0,\ee}|^{p}\rho_{\ee}\om^{n}.
    \]
    Similarly 
    \[
    \int_{X} |\dot{w}_{0}|^{p}(\bb + dd^{c}u)^{n} = \int_{X}\mathds{1}_{D}|\dot{w}_{0}|^{p}\rho \om^{n}.
    \]
    As $\rho_{\ee} \to \rho$ pointwise and are uniformly bounded (from \cite[Lemma 3.5]{dinezza2018lp}), we just need to show that $|\dot{w}_{0,\ee}|$ are uniformly bounded in $\ee$. 

    From convexity of $v_{\la,\ee}$ in $\la$, we obtain that $v_{\la,\ee} \leq (1-\la)v_{0,\ee} + \la v_{1,\ee} \leq \max(v_{0,\ee}, v_{1,\ee})$. Also $\Poe(v_{0,\ee},v_{1,\ee}) \leq v_{\la,\ee}$. Combining we have $\Poe(v_{0,\ee}, v_{1,\ee}) \leq v_{\la,\ee} \leq \max\{v_{0,\ee},v_{1,\ee}\}$. 

    From \cite[Lemma 3.1]{Darvas2017OnTS}, we obtain
    \begin{align*}
          |\dot{w}_{0,\ee}| &\leq |w_{1,\ee} - w_{0,\ee}|\\
    &= |u_{\ee} - v_{\la,\ee}| \\
    &\leq \max\{|u_{\ee} - \max\{v_{0,\ee},v_{1,\ee}\}|, |u_{\ee}-\Poe(v_{0,\ee},v_{1,\ee})|\}\\
    &= \max\{|u_{\ee} - v_{0,\ee}|, |u_{\ee} - v_{1,\ee}|, |u_{\ee} - \Poe(v_{0,\ee}, v_{1,\ee})|\}.
    \end{align*}
  
    Since $\Poe(v_{0,\ee}, v_{1,\ee}) = \Poe(\min\{f_{0},f_{1}\})$ and $u_{0,\ee} = \Poe(f)$, $v_{1,\ee} = \Poe(f_{1})$ and $v_{0,\ee} = \Poe(f_{0})$, and using that for any continuous $h_{1}, h_{2} \in C(X)$, $|\Poe(h_{1} - \Poe(h_{2})| \leq \sup_{X} |h_{1}-h_{2}|$, we obtain that 
    \[
    |\dot{w}_{0,\ee}| \leq \max\{ \sup_{X}|f-f_{0}|, \sup_{X}|f-f_{1}|, \sup_{X}|f-\min\{f_{0},f_{1}\}|\}.
    \]
    Therefore by Lebesgue's Dominated Convergence theorem, and Lemma~\ref{lem: convergence of integrands},
    \[
    \lim_{\ee \to 0} \int_{X}\mathds{1}_{D_{\ee}}|\dot{w}_{0,\ee}|^{p}\rho_{\ee}\om^{n} = \int_{X}\mathds{1}_{D}|\dot{w}_{0}|^{p}\rho \om^{n}.
    \]
\end{proof}

Now the previous theorem proves
\begin{thm}\label{thm: convexity in big and nef case for Hbb}
    If $u, v_{0}, v_{1} \in \cH_{\bb}$, and $v_{\la}$ is the weak geodesic joining $v_{0}$ and $v_{1}$, then 
   \begin{align*}
          d_{p}(u,v_{\lambda})^{2} &\leq (1-\la)d_{p}(u,v_{0})^{2} + \la d_{p}(u,v_{1})^{2} - (p-1)\la(1-\la)d_{p}(v_{0},v_{1})^{2}, \text{ if } 1< p \leq 2 \text{ and }\\
d_{p}(u,v_{\la})^{p} &\leq (1-\la)d_{p}(u,v_{0})^{p} + \la d_{p}(u,v_{1})^{p} - \la^{p/2}(1-\la)^{p/2}d_{p}(v_{0},v_{1})^{p}, \text{ if } p \leq 2.
   \end{align*}
    
\end{thm}

\begin{proof}
   We only needed to prove that $\lim_{\ee \to 0}d_{p}(u_{\ee}, v_{\la,\ee}) = d_{p}(u,v_{\la})$  which is proved by Theorem~\ref{thm: convergence of geodesics in Hbb}. Now taking the limit $\ee \to 0$ in Equations~\eqref{eq: convexity for approximating sequence} and~\eqref{eq: convexity for approximating sequence p > 2} proves the result. 
\end{proof}

Now we will extend this proof to all of $\E^{p}(X,\bb)$. 
\begin{thm}\label{thm: uniform convexity in big and nef}
    Let $u,v_{0},v_{1} \in \E^{p}(X,\bb)$ and $v_{\la}$ be the weak geodesic joining $v_{0}$ and $v_{1}$, then 
    \begin{align*}
           d_{p}(u,v_{\lambda})^{2} &\leq (1-\la)d_{p}(u,v_{0})^{2} + \la d_{p}(u,v_{1})^{2} - (p-1)\la(1-\la)d_{p}(v_{0},v_{1})^{2}, \text{ if } 1< p \leq 2 \text{ and }\\
d_{p}(u,v_{\la})^{p} &\leq (1-\la)d_{p}(u,v_{0})^{p} + \la d_{p}(u,v_{1})^{p} - \la^{p/2}(1-\la)^{p/2}d_{p}(v_{0},v_{1})^{p}, \text{ if } p \leq 2.
    \end{align*}
\end{thm}

\begin{proof}
    We give the proof by approximation from $\cH_{\bb}$. Let $u^{j}, v_{0}^{j}, v_{1}^{j} \in \cH_{\bb}$ satisfy $u^{j} \searrow u$, $v_{0}^{j} \searrow v_{0}$, and $v_{1}^{j} \searrow v_{1}$. If $v_{\la}^{j}$ is the weak geodesic joining $v_{0}^{j}$ and $v_{1}^{j}$, then by Theorem~\ref{thm: convexity in big and nef case for Hbb} 
    \begin{align*}
           d_{p}(u^{j},v^{j}_{\lambda})^{2} &\leq (1-\la)d_{p}(u^{j},v^{j}_{0})^{2} + \la d_{p}(u^{j},v^{j}_{1})^{2} - (p-1)\la(1-\la)d_{p}(v^{j}_{0},v^{j}_{1})^{2}, \text{ if } 1< p \leq 2 \text{ and }\\
d_{p}(u^{j},v^{j}_{\la})^{p} &\leq (1-\la)d_{p}(u^{j},v^{j}_{0})^{p} + \la d_{p}(u^{j},v^{j}_{1})^{p} - \la^{p/2}(1-\la)^{p/2}d_{p}(v^{j}_{0},v^{j}_{1})^{p}, \text{ if } p \leq 2.
    \end{align*}
    By definition of $d_{p}$ on $\E^{p}(X,\bb)$, we know that as we take the limit $j \to \infty$, $d_{p}(u^{j},v_{0}^{j}) \to d_{p}(u,v_{0})$, $d_{p}(u^{j}, v_{1}^{j}) \to d_{p}(u,v_{1})$ and $d_{p}(v_{0}^{j}, v_{1}^{j}) \to d_{p}(v_{0},v_{1})$. So we are done if we could prove that $d_{p}(u^{j}, v_{\la}^{j}) \to d_{p}(u,v_{\la})$. 

    By the same reasoning as in the proof of Lemma~\ref{lem: convergence of integrands}, we can see that $v_{\la}^{j} \searrow v_{\la}$. Since $u^{j} \searrow u$ and $v_{\la}^{j} \searrow v_{\la}$, from \cite[Proposition 3.12]{dinezza2018lp}, we get that $d_{p}(u^{j},u) \to 0$ and $d_{p}(v_{\la}^{j}, v_{\la}) \to 0$. Combining with the triangle inequality we get $d_{p}(u_{j},v_{\la}^{j}) \to d_{p}(u, v_{\la})$. 
\end{proof}

\section{Contraction Property and a Consequence}\label{sec: contraction property and a consequence}

Let $(X,\om)$ be a compact K\"ahler manifold, $\te$ be a smooth closed real $(1,1)$- form representing a big cohomology class, and $\psi \in \PSH(X,\te)$ have analytic singularities. In this section, we will prove that the map $\E^{p}(X,\te) \ni u \mapsto P_{\te}[\psi](u) \in \E^{p}(X,\te,\psi)$ is well defined and is a contraction, i.e., $d_{p}(P_{\te}[\psi](u), P_{\te}[\psi](v)) \leq d_{p}(u,v)$. When $p = 1$, the results in this section were proved in \cite[Section 4.1]{TrusianiL1metric}. 

We need a technical lemma, whose proof can be obtained by modifying the proof in \cite[Lemma 5.1]{GuptaCompletemetrictopology} by changing the weight function.
\begin{lem}\label{lem: limit of a decreasing sequence}
If $u_{j} \in \E^{p}(X,\te,\psi)$ is a decreasing sequence of functions such that for some $\varphi \in \E^{p}(X,\te, \psi)$, 
\[
\sup_{j} \int_{X}|u_{j} - \varphi|^{p}\te^{n}_{u_{j}} < \infty,
\]
then $u := \lim_{j\to\infty}u_{j} \in \E^{p}(X,\te,\psi)$. 
\end{lem}

\begin{lem}\label{lem: the projection is defined on Ep}
    If $\psi$ is a model potential, i.e., $P_{\te}[\psi] = \psi$, then $u \in \E^{p}(X,\te)$ implies $P_{\te}[\psi](u) \in \E^{p}(X,\te, \psi)$. 
\end{lem}

\begin{proof}
    If $u$ has minimal singularity type, then $|V_{\te} - u| \leq D$ for some constant $D > 0$. Therefore, 
    \[
    P_{\te}(\psi+ C, u) \leq P_{\te}(\psi + C, V_{\te} + D) = P_{\te}(\psi+C-D, V_{\te}) + D.
    \]
    Taking the limit $C \to \infty$ we get
    \[
    \lim_{C\to\infty}P_{\te}(\psi+C,u) \leq \lim_{C\to \infty}P_{\te}(\psi+C-D,V_{\te}) + D= \lim_{C\to \infty}P_{\te}(\psi+ C, V_{\te}) + D.
    \]
    Taking upper semicontinuous regularization we get
    \[
    P_{\te}[\psi](u) \leq P_{\te}[\psi](V_{\te}) + D = \psi + D.
    \]
    Similarly, 
    \[
    \psi \leq P_{\te}[\psi](u) + D.
    \]
    Thus $P_{\te}[\psi](u)$ has the same singularity type as $\psi$, thus $P_{\te}[\psi](u) \in \E^{p}(X,\te, \psi)$. 

    More generally, if $u \in \E^{p}(X,\te)$, then $u_{j}:= \max(u, V_{\te}- j)$ has the minimal singularity type. Then $P_{\te}[\psi](u_{j})$ has minimal singularity type and we claim that $P_{\te}[\psi](u_{j}) \searrow P_{\te}[\psi](u)$. Moreover, we will show that 
    \[
    \sup_{j} \int_{X}|P_{\te}[\psi](u_{j}) - \psi|^{p}\te^{n}_{P_{\te}[\psi]u_{j}} < \infty
    \]
    concluding with Lemma~\ref{lem: limit of a decreasing sequence} that $P_{\te}[\psi](u) \in \E^{p}(X,\te,\psi)$. 

    Let $u\leq K$. Then $u_{j} \leq K$ and $P_{\te}[\psi](u_{j}) \leq K$ as well. Therefore, $P_{\te}[\psi](u_{j}) - K \leq 0$ and $P_{\te}[\psi](u_{j})$ has the same singularity type as $\psi$, which is a model singularity, thus $P_{\te}[\psi](u_{j}) - K \leq \psi$. Hence we get 
    \[
    P_{\te}[\psi](u_{j}) - K - V_{\te} \leq P_{\te}[\psi](u_{j}) - K - \psi \leq 0.
    \]
    We also need \cite[Theorem 3.8]{Darvasmonotonicity} that says $\te^{n}_{P_{\te}[\psi](u_{j})} \leq \mathds{1}_{\{P_{\te}[\psi](u_{j})= u_{j}\}} \te^{n}_{u_{j}}$. Using $(a + b)^{p} \leq 2^{p-1}(a^{p} + b^{p})$, we get 
    \begin{align*}
    \int_{X}|P_{\te}[\psi](u_{j}) - \psi|^{p}\te^{n}_{P_{\te}[\psi](u_{j})} &\leq 2^{p-1}\int_{X}( |P_{\te}[\psi](u_{j}) - K - \psi|^{p} + K^{p}) \te^{n}_{P_{\te}[\psi](u_{j})}\\
    &\leq 2^{p-1}\left( \int_{X} |P_{\te}[\psi](u_{j})  - K - V_{\te}|^{p}\te^{n}_{P_{\te}[\psi](u_{j})} + K^{p}\int_{X}\te^{n}_{\psi} \right)\\
    &\leq 2^{p-1}\left(\int_{\{P_{\te}[\psi](u_{j})= u_{j}\}} |u_{j} - K - V_{\te}|^{p} \te^{n}_{u_{j}}+ K^{p}\int_{X}\te^{n}_{\psi}\right)\\
    &\leq 2^{p-1}\left( \int_{X} |u_{j} - K - V_{\te}|^{p}\te^{n}_{u_{j}} + K^{p}\int_{X}\te^{n}_{\psi}\right) 
    \end{align*}
    is uniformly bounded. We obtain the uniform boundedness of the integral in the last equation by combining the quasi-triangle inequality \cite[Theorem 1.6]{Guedj2019PlurisubharmonicEA} and \cite[Lemma 1.9]{Guedj2019PlurisubharmonicEA}.

    Since $P_{\te}[\psi](u_{j})$ is a decreasing sequence of potentials in $\E^{p}(X,\te,\psi)$, the above calculation and Lemma~\ref{lem: limit of a decreasing sequence} imply that $v := \lim_{j\to \infty}P_{\te}[\psi](u_{j}) \in \E^{p}(X,\te,\psi)$. As $v \leq u_{j}$ for all $j$, we get that $v \leq u$. Moreover, $v \preceq \psi$, thus $v$ is a candidate for $P_{\te}[\psi](u)$. Hence $P_{\te}[\psi](u)$ exists and $v \leq P_{\te}[\psi](u)$. Since $u_{j} \geq u$, we also have that $P_{\te}[\psi](u_{j}) \geq P_{\te}[\psi](u)$ for each $j$. Taking limit we get $v \geq P_{\te}[\psi](u)$. Thus $\lim_{j\to \infty}P_{\te}[\psi](u_{j}) = P_{\te}[\psi](u) \in \E^{p}(X,\te, \psi)$. 
\end{proof}

\begin{thm}\label{thm: contraction property}
    If $\psi \in \PSH(X,\te)$ has analytic singularities, then the map $P_{\te}[\psi](\cdot) : (\E^{p}(X,\te),d_{p}) \to (\E^{p}(X,\te,\psi),d_{p})$ is a contraction. This means for any $u_{0}, u_{1} \in \E^{p}(X,\te)$, 
    \begin{equation}\label{eq: contraction formula}
     d_{p}(P_{\te}[\psi](u_{0}), P_{\te}[\psi](u_{1})) \leq d_{p}(u_{0}, u_{1}).    
    \end{equation}
\end{thm}

\begin{proof}
    First we assume that there are functions $f_{0}, f_{1} \in C^{1,\bar{1}}(X)$ such that $u_{0} = P_{\te}(f_{0})$ and $u_{1} =P_{\te}(f_{1})$. Let $v_{0}:= P_{\te}[\psi](u_{0}) = P_{\te}[\psi](f_{0})$ and $v_{1}:= P_{\te}[\psi](u_{1}) = P_{\te}[\psi](f_{1})$. Moreover assume that $u_{0} \leq u_{1}$. In this case, we know from Theorem~\ref{thm: convergence for smooth enough potentials} combined with Theorem~\ref{thm: measure on contact sets} that
    \begin{equation}\label{eq: dp in minimal singularity}
    d_{p}^{p}(u_{0}, u_{1}) =\int_{X}|\dot{u}_{0}|^{p}\te^{n}_{u_{0}} =  \int_{X} \mathds{1}_{\{P_{\te}(f_{0}) = f_{0}\}}(\dot{u}_{0})^{p}\te^{n}_{f_{0}}    
    \end{equation}
    and from Theorem~\ref{thm: distance between smooth objects in analytic singularity} combined with Theorem~\ref{thm: measure on contact sets} that
    \begin{equation}\label{eq: dp in analytic singularity}
    d_{p}^{p}(v_{0}, v_{1}) = \int_{X}|\dot{v}_{0}|^{p}\te^{n}_{v_{0}} = \int_{X} \mathds{1}_{\{P_{\te}[\psi](f_{0}) = f_{0}\}} (\dot{v}_{0})^{p}\te^{n}_{f_{0}},    
    \end{equation}
    where $u_{t}$ and $v_{t}$ are the weak geodesics joining $u_{0}, u_{1}$ and $v_{0}, v_{1}$ respectively. Since $P_{\te}[\psi](f_{0}) \leq P_{\te}(f_{0}) \leq f_{0}$, we know $\{ P_{\te}[\psi](f_{0}) = f_{0}\} \subset \{ P_{\te}(f_{0}) = f_{0}\}$. As $u_{0} \geq v_{0}$ and $u_{1} \geq v_{1}$, we have $u_{t} \geq v_{t}$. If $x \in \{ P_{\te}[\psi](f_{0}) = f_{0}\}$, then 
    \[
    \dot{v}_{0}(x) = \lim_{t\to\infty} \frac{v_{t}(x) - v_{0}(x)}{t} \leq  \lim_{t\to\infty}\frac{u_{t}(x) - f_{0}(x)}{t}  = \lim_{t\to\infty}\frac{u_{t}(x) - u_{0}(x)}{t} = \dot{u}_{0}(x).
    \]
    Thus $\mathds{1}_{\{P_{\te}[\psi](f_{0})=f_{0}\}} (\dot{v}_{0})^{p} \leq \mathds{1}_{P_{\te}(f_{0}) = f_{0}\}}(\dot{u}_{0})^{p}$. Now Equations~\eqref{eq: dp in minimal singularity} and~\eqref{eq: dp in analytic singularity} give 
    \[
    d_{p}(P_{\te}[\psi](u_{0}), P_{\te}[\psi](u_{1})) = d_{p}(v_{0},v_{1}) \leq d_{p}(u_{0}, u_{1}).
    \]

    Now we will remove the assumption that $u_{0} \leq u_{1}$. We still assume that $u_{0} = P_{\te}(f_{0})$ and $u_{1} = P_{\te}(f_{1})$ for some $f_{0}, f_{1} \in C^{1,\bar{1}}(X)$. We will use the Pythagorean formula for $d_{p}$ metrics to establish Equation~\eqref{eq: contraction formula} in this case. As before let $v_{0}:= P_{\te}[\psi](u_{0}) = P_{\te}[\psi](f_{0})$ and $v_{1} := P_{\te}[\psi](u_{1}) = P_{\te}[\psi](f_{1})$. Let $C>0$ be a constant such that $\te \leq C\om$. Then $h = P_{C\om}(f_{0},f_{1}) \in C^{1,\bar{1}}(X)$,  $P_{\te}(u_{0},u_{1}) = P_{\te}(h)$, and $P_{\te}(v_{0},v_{1}) = P_{\te}[\psi](h)$. Also observe that $P_{\te}[\psi](P_{\te}(u_{0},u_{1})) = P_{\te}(v_{0}, v_{1})$. Applying the result in the previous paragraph we obtain $d_{p}(u_{0}, P_{\te}(u_{0},u_{1})) \geq d_{p}(v_{0}, P_{\te}(v_{0},v_{1}))$ and $d_{p}(u_{1}, P_{\te}(u_{0},u_{1})) \geq d_{p}(v_{1},P_{\te}(v_{0},v_{1}))$. 

    Using the Pythagorean formula, we write
    \begin{align*}
        d_{p}^{p}(u_{0}, u_{1}) &= d_{p}^{p}(u_{0}, P_{\te}(u_{0},u_{1})) + d_{p}^{p}(u_{1}, P_{\te}(u_{0},u_{1})) \\
        &\geq d_{p}^{p}(v_{0},P_{\te}(v_{0},v_{1})) + d_{p}^{p}(v_{1}, P_{\te}(v_{0},v_{1}))\\
        &= d_{p}^{p}(v_{0},v_{1}).
    \end{align*}

    Thus we have shown that Equation~\eqref{eq: contraction formula} holds when $u_{0} = P_{\te}(f_{0})$ and $u_{1} = P_{\te}(f_{1})$ for $f_{0},f_{1} \in C^{1,\bar{1}}(X)$. We show it more generally by approximation. 

    If $u_{0}, u_{1} \in \E^{p}(X,\te)$, then we can find $u_{0}^{j}, u_{1}^{j} \in \cH_{\te}$ such that $u_{0}^{j}\searrow u_{0}$ and $u_{1}^{j} \searrow u_{1}$. Moreover, $P_{\te}[\psi](u_{0}^{j}) \searrow P_{\te}[\psi](u_{0})$ and $P_{\te}[\psi](u_{1}^{j}) \searrow P_{\te}[\psi](u_{1})$. The proof is the same as in Lemma~\ref{lem: limit of a decreasing sequence}. Thus from Lemma~\ref{lem: decreasing sequence in Ep} we get that $\lim_{j\to\infty}d_{p}(P_{\te}[\psi](u_{0}^{j}), P_{\te}[\psi](u_{1}^{j})) = d_{p}(P_{\te}[\psi](u_{0}), P_{\te}[\psi](u_{1}))$. Hence from the result in the previous paragraph we have 
    \begin{align*}
    d_{p}(u_{0}, u_{1}) &= \lim_{j\to \infty}d_{p}(u_{0}^{j}, u_{1}^{j}) \\
    &\geq \lim_{j\to \infty}d_{p}(P_{\te}[\psi](u_{0}^{j}), P_{\te}[\psi](u_{1}^{j}))\\
    &= d_{p}(P_{\te}[\psi](u_{0}), P_{\te}[\psi](u_{1})).    
    \end{align*}
    This proves Equation~\eqref{eq: contraction formula} in the full generality as desired. 
\end{proof}

A consequence of this contraction formula is that the approximation formula for $d_{p}$ on $\cH_{\te} \subset \E^{p}(X,\te)$ from potentials in analytic singularity type can be extended to any potentials in $\E^{p}(X,\te)$. More precisely, we can prove 
\begin{thm}\label{thm: general approximation from analytic singularity type}
  Let $\psi_{k} \in \PSH(X,\te)$ have analytic singularities and $\psi_{k} \nearrow V_{\te}$ as described in the beginning of Section~\ref{sec: metric in the big case}. If $u_{0}, u_{1} \in \E^{p}(X,\te)$, then 
    \begin{equation}\label{eq: general equation by approximation}
    d_{p}(u_{0}, u_{1}) = \lim_{k\to\infty} d_{p}(P_{\te}[\psi_{k}](u_{0}), P_{\te}[\psi_{k}](u_{1})).    
    \end{equation}
    
\end{thm}

\begin{proof}
    Equation~\eqref{eq: general equation by approximation} is the definition of $d_{p}$ when $u_{0}, u_{1} \in \cH_{\te}$. Here we want to prove it more generally. 

    Let $u_{0}^{j}, u_{1}^{j} \in\cH_{\te}$ be such that $u_{0}^{j} \searrow u_{0}$ and $u_{1}^{j} \searrow u_{1}$. Then by definition of $d_{p}$ on $\E^{p}(X,\te)$, we have 
    \[
    d_{p}(u_{0}, u_{1}) = \lim_{j\to\infty} d_{p}(u_{0}^{j}, u_{1}^{j})
    \]
    and 
    \[
    d_{p}(u_{0}^{j}, u_{1}^{j}) = \lim_{k\to \infty} d_{p}(P_{\te}[\psi_{k}](u_{0}^{j}), P_{\te}[\psi_{k}](u_{1}^{j})).
    \]
    Combining the two we get 
    \begin{equation}\label{eq: to exchange the limit}
    d_{p}(u_{0}, u_{1}) = \lim_{j\to \infty} \lim_{k \to \infty} d_{p}(P_{\te}[\psi_{k}](u_{0}^{j}), P_{\te}[\psi_{k}](u_{1}^{j})).  
    \end{equation}
    
     We want to exchange the limit. First, observe that as $j \to \infty$, $P_{\te}[\psi_{k}](u_{0}^{j}) \searrow P_{\te}[\psi_{k}](u_{0})$ and $P_{\te}[\psi_{k}](u_{1}^{j}) \searrow P_{\te}[\psi_{k}](u_{1})$. Thus from Lemma~\ref{lem: decreasing sequences converge}, 
     \begin{equation}\label{eq: equation for uniform limit}
     \lim_{j\to \infty} d_{p}(P_{\te}[\psi_{k}](u_{0}^{j}), P_{\te}[\psi_{k}](u_{1}^{j})) = d_{p}(P_{\te}[\psi_{k}
     (u_{0}), P_{\te}[\psi_{k}](u_{1})).    
     \end{equation}
    Now we will show that the limit in Equation~\eqref{eq: equation for uniform limit} is uniform in $k$. For that, we observe by triangle inequality that 
    \begin{align*}
        &|d_{p}(P_{\te}[\psi_{k}](u_{0}^{j}), P_{\te}[\psi_{k}](u_{1}^{j})) - d_{p}(P_{\te}[\psi_{k}](u_{0}), P_{\te}[\psi_{k}](u_{1}))| \\ \leq  &d_{p}(P_{\te}[\psi_{k}](u_{0}^{j}), P_{\te}[\psi_{k}](u_{0})) + d_{p}(P_{\te}[\psi_{k}](u_{1}^{j}), P_{\te}[\psi_{k}](u_{1})) \\
        \leq &d_{p}(u_{0}^{j}, u_{0}) + d_{p}(u_{1}^{j},u_{0}),
    \end{align*}
    where in the last line we used Theorem~\ref{thm: contraction property}. Moreover, $\lim_{j \to \infty}d_{p}(u_{0}^{j}, u_{0}) = 0$ and $\lim_{j\to\infty} d_{p}(u_{1}^{j}, u_{1}) = 0$. 
 Using the uniform convergence in $k$, we obtain that we can exchange the limits in Equation~\eqref{eq: to exchange the limit}. Thus 
    \begin{align*}
         d_{p}(u_{0}, u_{1}) &= \lim_{k\to\infty}\lim_{j\to \infty} d_{p}(P_{\te}[\psi_{k}](u_{0}^{j}), P_{\te}[\psi_{k}](u_{1}^{j}))\\  &= \lim_{k\to\infty}d_{p}(P_{\te}[\psi_{k}](u_{0}), P_{\te}[\psi_{k}](u_{1})),
    \end{align*}
   as desired.
    \end{proof}

\section{Uniform Convexity in the big case}\label{sec: uniform convexity in the big case}
With the help of Theorem~\ref{thm: general approximation from analytic singularity type}, we can prove uniform convexity in the big case as well. First, we see that the uniform convexity extends to the analytic singularity setting as well. 
\begin{thm}\label{thm: uniform convexity in analytic singularity setting}
    If $\te$ represents a big cohomology class and $\psi \in \PSH(X,\te)$ has analytic singularities, then the metric space $(\E^{p}(X,\te, \psi), d_{p})$ for $p > 1$, as described in Section~\ref{sec: analytic singularities to big and nef}, is uniformly convex. 
\end{thm}

\begin{proof}
    Recall that we constructed the metric $d_{p}$ on $\E^{p}(X,\te,\psi)$ in Section~\ref{subsec: metric space structure on analytic singularity} by resolving the singularities of $\psi$.  Let $\mu : \tilde{X} \to X$ be the resolution as described in Section~\ref{sec: analytic singularities to big and nef}. Recall that from Theorem~\ref{thm: bijection between analytic singularity and big and nef} there is a smooth closed real $(1,1)$-form $\tte$ on $\tX$ and a bounded function $g \in \PSH(\tX, \tte)$ such that the map $\PSH(X,\te,\psi) \ni u \mapsto \tilde{u}:= (u-\psi)\circ \mu + g \in \PSH(\tX, \tte)$ is an order preserving bijection and from Corollary~\ref{cor: bijection preserves mass and energy}, $\E^{p}(X,\te,\psi) \ni u \mapsto \tilde{u} \in \E^{p}(\tX, \tte)$ is a bijection as well. 

    If $u, v_{0}, v_{1} \in \E^{p}(X,\te, \psi)$ and $v_{\la}$ is the weak geodesic joining $v_{0}$ and $v_{1}$, then $\tilde{u}, \tilde{v}_{0}, \tilde{v}_{1} \in \E^{p}(\tX, \tte)$ and by Theorem~\ref{thm: Mabuchi geodesics are preserved} $\tilde{v}_{\la}:= (v_{\la} - \psi)\circ \mu + g$ is the weak geodesic joining $\tilde{v}_{0}$ and $\tilde{v}_{1}$.  From Theorem~\ref{thm: uniform convexity in big and nef}, $(\E^{p}(\tX,\tte),d_{p})$ is uniformly convex, thus 
    \begin{align*}
        d_{p}(\tilde{u},\tilde{v}_{\lambda})^{2} &\leq (1-\la)d_{p}(\tilde{u},\tilde{v}_{0})^{2} + \la d_{p}(\tilde{u},\tilde{v}_{1})^{2} - (p-1)\la(1-\la)d_{p}(\tilde{v}_{0},\tilde{v}_{1})^{2}, \text{ if } 1< p \leq 2 \text{ and }\\
d_{p}(\tilde{u},\tilde{v}_{\la})^{p} &\leq (1-\la)d_{p}(\tilde{u},\tilde{v}_{0})^{p} + \la d_{p}(\tilde{u},\tilde{v}_{1})^{p} - \la^{p/2}(1-\la)^{p/2}d_{p}(\tilde{v}_{0},\tilde{v}_{1})^{p}, \text{ if } p \leq 2.
    \end{align*} 

        For $u_{0}, u_{1} \in \E^{p}(X,\te,\psi)$, we defined $d_{p}(u_{0}, u_{1}) := d_{p}(\tilde{u}_{0}, \tilde{u}_{1})$ in Equation~\eqref{eq: definition of dp in analytic singualarity setting}. Applying this we get
        \begin{align*}
        d_{p}(u,v_{\lambda})^{2} &\leq (1-\la)d_{p}(u,v_{0})^{2} + \la d_{p}(u,v_{1})^{2} - (p-1)\la(1-\la)d_{p}(v_{0},v_{1})^{2}, \text{ if } 1< p \leq 2 \text{ and }\\
d_{p}(u,v_{\la})^{p} &\leq (1-\la)d_{p}(u,v_{0})^{p} + \la d_{p}(u,v_{1})^{p} - \la^{p/2}(1-\la)^{p/2}d_{p}(v_{0},v_{1})^{p}, \text{ if } p \leq 2
    \end{align*} 
implying uniform convexity of $(\E^{p}(X,\te, \psi), d_{p})$ for $p > 1$. 
\end{proof}

We would also need the analytic singularity version of \cite[Proposition 3.6]{DarvasLuuniformconvexity} which holds true, because the proof in \cite{DarvasLuuniformconvexity} only relies on the uniform convexity of Theorem~\ref{thm: uniform convexity in analytic singularity setting}. 
\begin{thm} \label{thm: points close to geodesic}
Let $\psi \in \PSH(X,\te)$ have analytic singularities. Let $u_{0}, u_{1} \in \E^{p}(X,\te,\psi)$ for $p > 1$, and $u_{t}$ be the weak geodesic joining $u_{0}$ and $u_{1}$. If $v \in \E^{p}(X,\te,\psi)$ satisfies $d_{p}(u_{0},v) \leq (t+\ee) d_{p}(u_{0},u_{1})$ and $d_{p}(u_{1},v) \leq (1-t+ \ee)d_{p}(u_{0},u_{1})$, for some $\ee >0$ and $t \in [0,1]$, then for some constant $C(p) > 0$, 
\[
d_{p}(v,u_{t}) \leq \ee^{1/r}Cd_{p}(u_{0},u_{1})
\]
where $r = \max\{2,p\}$.
\end{thm}

Now we can prove one of our main results:
\begin{thm}\label{thm: uniform convexity in the big case}
If $\te$ represents a big cohomology class, then the metric space $(\E^{p}(X,\te),d_{p})$ for $p > 1$ is uniformly convex. This means for $u, v_{0}, v_{1} \in \E^{p}(X,\te)$, if $v_{\la}$ is the geodesic joining $v_{0}$ and $v_{1}$, then
\begin{align*}
        d_{p}(u,v_{\lambda})^{2} &\leq (1-\la)d_{p}(u,v_{0})^{2} + \la d_{p}(u,v_{1})^{2} - (p-1)\la(1-\la)d_{p}(v_{0},v_{1})^{2}, \text{ if } 1< p \leq 2 \text{ and }\\
d_{p}(u,v_{\la})^{p} &\leq (1-\la)d_{p}(u,v_{0})^{p} + \la d_{p}(u,v_{1})^{p} - \la^{p/2}(1-\la)^{p/2}d_{p}(v_{0},v_{1})^{p}, \text{ if } p \leq 2.
    \end{align*} 
\end{thm}
\begin{proof}
    Let $\psi_{k} \nearrow V_{\te}$ be the increasing sequence of $\te$-psh functions with analytic singularities. Let $u^{k} = P_{\te}[\psi_{k}](u)$, $v_{0}^{k} = P_{\te}[\psi_{k}](v_{0})$, and $v_{1}^{k} = P_{\te}[\psi_{k}](v_{1})$. Let $v_{\la}^{k}$ be the weak geodesic joining $v_{0}^{k}$ and $v_{1}^{k}$. From Theorem~\ref{thm: uniform convexity in analytic singularity setting}, we know that 
    \begin{align*}
        d_{p}(u^{k},v^{k}_{\lambda})^{2} &\leq (1-\la)d_{p}(u^{k},v^{k}_{0})^{2} + \la d_{p}(u^{k},v^{k}_{1})^{2} - (p-1)\la(1-\la)d_{p}(v^{k}_{0},v^{k}_{1})^{2}, \text{ if } 1< p \leq 2 \text{ and }\\
d_{p}(u^{k},v^{k}_{\la})^{p} &\leq (1-\la)d_{p}(u^{k},v^{k}_{0})^{p} + \la d_{p}(u^{k},v^{k}_{1})^{p} - \la^{p/2}(1-\la)^{p/2}d_{p}(v^{k}_{0},v^{k}_{1})^{p}, \text{ if } p \leq 2.
    \end{align*}
    From Theorem~\ref{thm: general approximation from analytic singularity type} we know that $\lim_{k\to \infty} d_{p}(u^{k}, v_{0}^{k}) = d_{p}(u, v_{0})$, $\lim_{k\to \infty} d_{p}(u^{k}, v_{1}^{k}) = d_{p}(u, v_{1})$, and $d_{p}(v_{0}^{k}, v_{1}^{k}) = d_{p}(v_{0}, v_{1})$. Thus to finish the proof by taking the limit $k \to \infty$, we need to show that $d_{p}(u^{k}, v_{\la}^{k}) \to d_{p}(u,v_{\la})$. Unfortunately, it may not be true that $P_{\te}[\psi_{k}](v_{\la}) = v_{\la}^{k}$. But by using Theorem~\ref{thm: contraction property}, and Theorem~\ref{thm: points close to geodesic}, we can show that $v_{\la}^{k}$ and $P_{\te}[\psi_{k}](v_{\la})$ are $d_{p}$-close. 

    From Theorem~\ref{thm: contraction property}, and the fact that $(\E^{p}(X,\te,\psi_{k}), d_{p})$ is a geodesic metric space, we know that 
    \[
    d_{p}(v_{0}^{k}, P_{\te}[\psi_{k}](v_{\la})) \leq d_{p}(v_{0},v_{\la}) = \la d_{p}(v_{0},v_{1})
    \]
    and 
    \[
    d_{p}(v_{1}^{k}, P_{\te}[\psi_{k}](v_{\la})) \leq d_{p}(v_{1}, v_{\la}) = (1-\la) d_{p}(v_{0},v_{1}).
    \]
    Again by the contraction theorem $d_{p}(v_{0}^{k}, v_{1}^{k}) \leq d_{p}(v_{0}, v_{1})$, moreover by Theorem~\ref{thm: general approximation from analytic singularity type}, $\lim_{k \to \infty}d_{p}(v_{0}^{k}, v_{1}^{k}) = d_{p}(v_{0},v_{1})$. Thus we can write 
    \[
    \frac{d_{p}(v_{0},v_{1})}{d_{p}(v_{0}^{k}, v_{1}^{k})} \leq  1+\ee_{k}
    \]
    where $\ee_{k} > 0$ and $\ee_{k} \to 0$ as $k \to \infty$. Thus we have 
    \[
    d_{p}(v_{0}^{k}, P_{\te}[\psi_{k}](v_{\la})) \leq (\la+\la \ee_{k})d_{p}(v_{0}^{k}, v_{1}^{k}) \leq (\la + \ee_{k})d_{p}(v_{0}^{k}, v_{1}^{k})
    \]
    and 
    \[
    d_{p}(v_{1}^{k}, P_{\te}[\psi_{k}](v_{\la})) \leq (1-\la)(1+\ee_{k}) d_{p}(v_{0}^{k}, v_{1}^{k}) \leq (1-\la +\ee_{k})d_{p}(v_{0}^{k},v_{1}^{k}).
    \]
    Applying Theorem~\ref{thm: points close to geodesic} we get that 
    \[
    d_{p}(v_{\la}^{k}, P_{\te}[\psi_{k}](v_{\la})) \leq (\ee_{k})^{1/r}Cd_{p}(v_{0}^{k}, v_{1}^{k}) \leq (\ee_{k})^{1/r}Cd_{p}(v_{0},v_{1}).
    \]
    Taking the limit $k \to \infty$ and using that $\ee_{k} \to 0$, we get 
    \begin{equation}\label{eq: convergence of close points}
    \lim_{k\to \infty} d_{p}(v_{\la}^{k}, P_{\te}[\psi_{k}](v_{\la})) = 0.     
    \end{equation}

    Now we will show that $d_{p}(u^{k}, v_{\la}^{k}) \to d_{p}(u,v_{\la})$ as $k \to \infty$. By applying the triangle inequality twice we get
    \begin{align*}
        |d_{p}(u^{k}, v_{\la}^{k}) - d_{p}(u,v_{\la})| &\leq |d_{p}(u^{k}, v_{\la}^{k}) - d_{p}(u^{k}, P_{\te}[\psi_{k}](v_{\la}))| + |d_{p}(u^{k}, P_{\te}[\psi_{k}](v_{\la})) - d_{p}(u,v_{\la})|\\
        &\leq d_{p}(v_{\la}^{k}, P_{\te}[\psi_{k}](v_{\la})) +  |d_{p}(u^{k}, P_{\te}[\psi_{k}](v_{\la})) - d_{p}(u,v_{\la})|
    \end{align*}
    As $k \to \infty$, the first term goes to 0 due to Equation~\eqref{eq: convergence of close points}, and the second term goes to 0 due to Theorem~\ref{thm: general approximation from analytic singularity type}.
\end{proof}

The same proofs as in \cite[Theorem 3.5]{DarvasLuuniformconvexity} gives
\begin{cor}
In the metric space $(\E^{p}(X,\te), d_{p})$ for $p > 1$, the weak geodesics are the only metric geodesics. 
\end{cor}

The same proof as in \cite[Theorem 3.6]{DarvasLuuniformconvexity} proves that 
\begin{cor}
    Let $u,v_{0}, v_{1} \in \E^{p}(X,\te)$ for $p > 1$. Let $t \in [0,1]$ and $\ee > 0$ such that $d_{p}(u,v_{0}) \leq (t+\ee) d_{p}(v_{0},_{1})$ and $d_{p}(u,v_{1}) \leq (1-t+\ee)d_{p}(v_{0},v_{1})$. If $v_{s}$ is the weak geodesic joining $v_{0}$ and $v_{1}$, then there exists $C(p) > 0$ such that
    \[
    d_{p}(u,v_{t}) \leq \ee^{\frac{1}{r}}d_{p}(v_{0},v_{1})
    \]
    where $r  = \max\{2,p\}$. 
\end{cor}
\emergencystretch 2em
\printbibliography
\end{document}